\documentclass[letterpaper, 11pt,  reqno]{amsart}
\usepackage{graphicx} 
\usepackage{amsmath,amssymb,amscd,amsthm,amsxtra,esint}
\usepackage{color}
\usepackage[margin=1.1in,marginparwidth=1.5cm, marginparsep=0.5cm]{geometry} 
\usepackage{tikz-cd}
\tikzset{smalltext/.style={"\textup{\small #1}" description}}

\usepackage[markup=default]{changes}
\definechangesauthor[name={Reviewer 1}, color=orange]{R1} 
\definechangesauthor[name={Reviewer 2}, color=blue]{R2} 
\definechangesauthor[name={Reviewer 3}, color=green]{R3}

\usepackage[implicit=true]{hyperref}

\allowdisplaybreaks[4]

\definecolor{gr}{rgb}   {0.,   0.69,   0.23 }
\definecolor{bl}{rgb}   {0.,   0.5,   1. }
\definecolor{mg}{rgb}   {0.85,  0.,    0.85}
\definecolor{yl}{rgb}   {0.8,  0.7,   0.}
\definecolor{or}{rgb}  {0.7,0.2,0.2}

\definecolor{Magenta}{cmyk}{0,1,0,0}

\newtheorem{theorem}{Theorem}[section]
\newtheorem{lemma}[theorem]{Lemma}
\newtheorem{proposition}[theorem]{Proposition}
\newtheorem{definition}[theorem]{Definition}

\newtheorem{remark}[theorem]{Remark}

\DeclareMathOperator{\Law}{Law}

\newcommand{\1}{\hspace{0.5mm}\text{I}\hspace{0.2mm}}
\newcommand{\II}{\text{I \hspace{-2.8mm} I} }

\newcommand{\noi}{\noindent}
\newcommand{\R}{\mathbb{R}}
\newcommand{\T}{\mathbb{T}}
\newcommand{\Z}{\mathbb{Z}}
\newcommand{\N}{\mathbb{N}}

\newcommand{\E}{\mathbb{E}}
\newcommand{\PP}{\mathbb{P}}
\newcommand{\D}{\mathcal{D}}
\renewcommand{\H}{\mathcal{H}}
\newcommand{\Id}{\textbf{Id}}

\newcommand{\al}{\alpha}
\newcommand{\be}{\beta}
\newcommand{\ta}{\theta}
\newcommand{\s}{\sigma}
\newcommand{\eps}{\varepsilon}
\newcommand{\om}{\omega}
\newcommand{\Dl}{\Delta}
\newcommand{\dl}{\delta}
\newcommand{\ld}{\lambda}

\newcommand{\dt}{\partial_t}
\newcommand{\nb}{\nabla}
\newcommand{\ind}{\mathbf 1}
\newcommand{\ft}{\widehat}
\newcommand{\wt}{\widetilde}
\newcommand{\les}{\lesssim}

\newcommand{\cj}{\overline}

\newcommand{\G}{\Gamma}
\newcommand{\g}{\gamma}

\renewcommand{\P}{\mathbf{P}}
\newcommand{\A}{\mathcal{A}}

\newcommand{\plan}{\mathfrak{p}}

\newcommand{\rhoo}{\vec \rho_N}

\newcommand{\muN}{\mu_1^{\otimes N}}

\newcommand{\EE}{\mathcal{E}}
\newcommand{\B}{\mathcal{B}}

\newcommand{\F}{\mathcal{F}}

\newcommand{\HLS}{$\text{HLSM}_N$ }
\newcommand{\HLSS}{$\text{HLSM}_N$}

\newcommand{\HLSe}{$\text{HLSM}_{\eps, N}$ }
\newcommand{\HLSSe}{$\text{HLSM}_{\eps, N}$}

\newcommand{\PLS}{$\text{PLSM}_N$ }
\newcommand{\PLSS}{$\text{PLSM}_N$}

\newcommand{\mf}{\mathfrak{m}}

\newcommand{\EN}{E_N}
\newcommand{\ENM}{E_{N, M}}

\newcommand{\ZZ}{\mathcal{Z}}
\newcommand{\Y}{\mathcal{Y}}

\newcommand{\jb}[1]{\langle #1 \rangle}
\newcommand{\jbb}[1]{[\hspace{-0.6mm}[ #1 ]\hspace{-0.6mm}]}
\newcommand{\wick}[1]{:\!{#1}\!:}
\newcommand{\deff}{\stackrel{\textup{def}}{=}}

\newcommand{\too}{\longrightarrow}

\newcommand{\I}{\mathcal{I}}

\newcommand{\vv}{\mathbf{v}}
\newcommand{\ww}{\mathbf{w}}
\newcommand{\uu}{\mathbf{u}}

\renewcommand{\O}{\Omega}
\renewcommand{\o}{\omega}

\renewcommand{\l}{\ell}

\makeatletter
\DeclareRobustCommand\widecheck[1]{{\mathpalette\@widecheck{#1}}}
\def\@widecheck#1#2{%
   \setbox\z@\hbox{\m@th$#1#2$}%
   \setbox\tw@\hbox{\m@th$#1%
      \widehat{%
         \vrule\@width\z@\@height\ht\z@
         \vrule\@height\z@\@width\wd\z@}$}%
   \dp\tw@-\ht\z@
   \@tempdima\ht\z@ \advance\@tempdima2\ht\tw@ \divide\@tempdima\thr@@
   \setbox\tw@\hbox{%
      \raise\@tempdima\hbox{\scalebox{1}[-1]{\lower\@tempdima\box\tw@}}}%
   {\ooalign{\box\tw@ \cr \box\z@}}}
\makeatother

\newtheorem*{ackno}{Acknowledgements}

\numberwithin{equation}{section}
\numberwithin{theorem}{section}

\makeatletter
\@namedef{subjclassname@2020}{%
  \textup{2020} Mathematics Subject Classification}
\makeatother

\title[Hyperbolic $O (N)$ linear sigma model]{Hyperbolic $O (N)$ linear sigma model\\
and its mean-field limit}

\author[R.~Liu, S.~Liu, and T.~Oh]{Ruoyuan Liu, Shao Liu, and Tadahiro Oh}

\address{
Ruoyuan Liu, Mathematical Institute\\
University of Bonn\\
Endenicher Allee 60\\
53115\\
Bonn\\
Germany}

\email{ruoyuanl@math.uni-bonn.de}

\address{
Shao Liu, 
School of Mathematics and Statistics, Beijing Institute of Technology, Beijing 100081, China\\
and 
Mathematical Institute\\
University of Bonn\\
Endenicher Allee 60\\
53115\\
Bonn\\
Germany}

\email{shaoliu@math.uni-bonn.de}

\address{
Tadahiro Oh, School of Mathematics\\
The University of Edinburgh\\
and The Maxwell Institute for the Mathematical Sciences\\
James Clerk Maxwell Building\\
The King's Buildings\\
Peter Guthrie Tait Road\\
Edinburgh\\ 
EH9 3FD\\
 United Kingdom, 
 and 
 School of Mathematics and Statistics, Beijing Institute of Technology,
Beijing 100081, China
}

\email{hiro.oh@ed.ac.uk}


\begin{document}

\baselineskip = 14pt


\begin{abstract} 

We study large $N$ limits of 
the hyperbolic $O(N)$ linear sigma model
(\HLSS) on the two-dimensional torus $\T^2$, 
namely, 
a  system of
$N$ interacting 
stochastic damped nonlinear wave equations (SdNLW) with 
coupled cubic nonlinearities.
After establishing (pathwise) global well-posedness
of \HLS and the limiting equation, called the mean-field SdNLW, 
we first establish global-in-time convergence of 
\HLS  to 
the mean-field SdNLW 
with general initial data
(under a suitable assumption).
In particular, for the local-in-time convergence, 
we obtain 
a convergence rate of order $N^{- \frac 12}$
under an additional integrability assumption on initial data.
We then show that
 the invariant Gibbs dynamics for \HLS  converges to 
 that for the mean-field SdNLW
 with 
a convergence rate of order $N^{- \frac 12}$ 
on any large time intervals.

\end{abstract}


\maketitle


\tableofcontents

\section{Introduction}
\label{SEC:1}

\subsection{Hyperbolic $O (N)$ linear sigma model} 
\label{SUBSEC:1.1}
In this paper, we consider the following system of 
$N$ interacting 
stochastic damped nonlinear wave equations (SdNLW)
with coupled cubic nonlinearities, posed on the two-dimensional torus $\T^2 = (\R / 2 \pi \Z)^2$:\footnote{In this  formal discussion, we ignore the issue of renormalization.}
\begin{align}
    (\dt^2 + \dt + m - \Dl) u_{N, j} = - \frac 1N \sum_{k = 1}^N u_{N, k}^2 u_{N, j} + \sqrt{2} \xi_j,
\qquad j = 1, \dots, N, 
\label{NLW1}
\end{align}

\noi
where $m > 0$ and $\{\xi_j\}_{j \in \N}$ is a family of  independent space-time white noises on 
$\R_+ \times \T^2$.
Our main goal is to 
study the large $N$ limit of solutions to \eqref{NLW1}
and show that (under suitable assumptions on initial data), for each fixed $j \in \N$, 
 a solution $u_{N, j}$ to \eqref{NLW1}
 converges in probability to a solution  $u_j$  to the following mean-field SdNLW on $\T^2$
 as $N \to \infty$:
\begin{align}
    (\dt^2 + \dt + m - \Dl) u_j = - \E [ u_j^2 ] \,u_j + \sqrt{2} \xi_j.
\label{MF1}
\end{align}

\noi
See
Theorems \ref{THM:conv1}
and \ref{THM:conv2} below for the main statements.

In quantum field theory, there are natural  models where the field $\vec u_N = (u_{N, 1}, \dots, u_{N, N}$) takes values in an
$N$-dimensional space, 
and it is of importance to investigate
their asymptotic  behavior as $N \to \infty$.
Such a large $N$ limit often
leads to 
a substantial 
simplification and has been studied for many models.
In \cite{Wilson}, 
Wilson considered the following
$N$-component generalization of the $\Phi^4_2$-measure:
\begin{align}
\begin{split}
d \rho_N  
= Z_N^{-1} & \exp \bigg( - \int_{\T^2} \frac 12 \sum_{j = 1}^N |\nb u_{N, j}|^2 
+ \frac{m}{2} \sum_{j = 1}^N u_{N, j}^2 \,dx \\
& \hphantom{XXX}- \frac{1}{4N} \int_{\T^2} \Big( \sum_{j = 1}^N u_{N, j}^2 \Big)^2 dx \bigg) 
    d \vec u_N, 
\end{split}
\label{Gibbs1}
\end{align}

\noi
known as
 the $O (N)$ linear sigma model.
 Here,  the word ``linear'' refers to the fact that the target space $\R^N$ is a linear space.
In \cite{SSZZ}, Shen, Smith, Zhu, and Zhu
studied the large $N$ limit
of the $O (N)$ linear sigma model $\rho_N$
via the method of stochastic quantization.
Namely, they studied the following
 coupled system of stochastic nonlinear heat equations (SNLH) on $\T^2$:
\begin{align}
    (\dt + m - \Dl) u_{N, j} = - \frac{1}{N} \sum_{k = 1}^N u_{N, k}^2 u_{N, j} + \sqrt{2} \xi_j.
\label{NLH1}
\end{align}

\noi
It is easy to see that 
 the measure $\rho_N$
 in \eqref{Gibbs1} is 
formally invariant under the dynamics
of~\eqref{NLH1} 
and, for this reason, 
we refer to \eqref{NLH1}
as the parabolic $O(N)$ linear sigma model
(\PLSS).
In \cite{SSZZ},
the authors 
 showed that a solution $u_{N, j}$ to \eqref{NLH1} converges
to the following mean-field SNLH:
\begin{align}
    (\dt + m - \Dl) u_j = - \E [ u_j^2 ] \,u_j + \sqrt{2} \xi_j
\label{NLH2}
\end{align}

\noi
(under suitable assumptions on initial data).
This dynamical result in turn implied
convergence   of the invariant measure $\rho_N$ in \eqref{Gibbs1}
to the limiting Gaussian measure $\mu_1^{\otimes \N}$
(in a suitable sense), 
provided that $m > 0$ is sufficiently large.
Here, 
$\mu_1$
denotes the massive
Gaussian free field measure with mass $m > 0$:
\begin{align}
    d \mu_{1} = Z^{-1} \exp\bigg(-\frac 12  \| (m - \Dl)^\frac 12 u \|_{L^2}^2 \bigg) du.
\label{gauss0}
\end{align}

\noi
See \cite[Theorem 1.2]{SSZZ} for a precise statement.
In a recent work \cite{DS}, 
Delgadino and Smith 
extended the latter result on the measure convergence to 
arbitrarily small masses $m > 0$
by making use of 
Talagrand's inequality
\cite{Tala, OV, FU}.
See \cite{SZZ1, SZZ2}
for further results 
on the (parabolic) $O (N)$ linear sigma model.
See also a nice survey paper 
 \cite{Shen} 
 on the subject, including 
physical and mathematical backgrounds on the large $N$ methods
and $\frac 1N$ expansions,
and the references therein.

As in the scalar case ($N = 1$)
studied in \cite{GKOT}, 
the SdNLW system \eqref{NLW1} corresponds to  the hyperbolic 
counterpart of the 
parabolic $O(N)$ linear sigma model \eqref{NLH1}.
In particular, 
 the SdNLW system~\eqref{NLW1} is the so-called canonical stochastic quantization of the $O(N)$ linear sigma model $\rho_N$ in \eqref{Gibbs1}; see \cite{RSS} and also Subsection~\ref{SUBSEC:1.3} below. For this reason, we  refer to the SdNLW system~\eqref{NLW1} as the {\it hyperbolic $O(N)$ linear sigma model}
 (\HLSS)
 in the following.

 Starting from the work \cite{OTh2, GKO}, 
there has been a significant progress on the study 
of singular stochastic (damped) nonlinear wave equations, 
broadly interpreted with stochastic forcing and\,/\,or random initial data
in recent years; 
see
 \cite{
Deya, GKO2, OPTz, GKOT,  Tolo1, OOR, OO,  OOT1, Bring2,  OOT, ORTz,  OOTz, OWZ, 
STzX, BLL,  OTWZ, BDNY,
Zine}.
Nevertheless, to the best of our knowledge, 
there is no well-posedness result  on a wave equation with a mean-field nonlinearity,  
in particular in the singular setting
(such as the mean-field SdNLW \eqref{MF1}), 
 not to mention the large $N$ limit (= the mean-field limit) of a system of nonlinear wave equations. 
Thus, the current paper presents the first such results on the hyperbolic $O(N)$ linear sigma model 
\eqref{NLW1} and its mean-field limit \eqref{MF1}.

\subsection{Large $N$ limit of the hyperbolic $O (N)$ linear sigma model}
\label{SUBSEC:1.2}

We now take a closer look at \HLS  \eqref{NLW1} and its (formal) mean-field limit \eqref{MF1}. Given  $j \in \N$, let $\Psi_j$ 
denote the stochastic convolution, satisfying  the following linear stochastic damped wave equation with zero initial data:
\begin{align}
\begin{cases}
(\dt^2 + \dt + m - \Dl) \Psi_j = \sqrt{2} \xi_j\\
(\Psi_j, \dt \Psi_j) |_{t = 0} = (0, 0), 
\end{cases}    
\label{psi1}
\end{align}

\noi
where $\{\xi_j \}_{j \in \N}$
is the family of independent space-time white noises as in \eqref{NLW1}.
As seen in~\cite{GKO, GKOT}, 
 $\Psi_j(t)$ is 
 almost surely a distribution of regularity\footnote{In describing regularities of functions and distributions, 
we use $\eps > 0$ to denote an arbitrarily small constant.}
 $- \eps$ and thus we need to introduce renormalization
 in considering its powers.

Let $\uu_N = \{u_{N, j}\}_{j = 1}^N$ be a solution to \HLS \eqref{NLW1}.
 Proceeding  with the  first order expansion \cite{McK, BO96, DPD}:
\begin{align}
    u_{N, j} = \Psi_j + v_{N, j},
\qquad j = 1, \dots, N, 
\label{exp}
\end{align}

\noi
we see that  the residual term $\vv_N = \{v_{N, j}\}_{j = 1}^N$ satisfies the following SdNLW system:
\begin{align}
\begin{split}
(\dt^2 + \dt + m - \Dl) v_{N, j} 
& = - \frac{1}{N} \sum_{k = 1}^N  \Big( 
v_{N, k}^2 v_{N, j}
+ 2 \Psi_k v_{N, k} v_{N, j}
  + v_{N, k}^2 \Psi_j \\
& 
\hphantom{XXXXX}
 + \wick{\Psi_k^2} v_{N, j} + 2 v_{N, k} \wick{ \Psi_k \Psi_j } 
+ \wick{\Psi_k^2 \Psi_j} \Big), 
\end{split}
\label{NLW3}
\end{align}

\noi
where we have already applied Wick renormalization 
to the powers of the stochastic convolution~$\Psi_j$:
\begin{align}
\begin{split}
 \wick{ \Psi_k \Psi_j } & = 
\begin{cases}
\wick{\Psi_k^2}\,, & \text{if } k = j, \\
\Psi_k \Psi_j, & \text{if } k \ne j, 
\end{cases}\\
 \wick{\Psi_k^2 \Psi_j} 
&  = 
\begin{cases}
\wick{\Psi_k^3}\,, & \text{if } k = j, \\
\wick{\Psi_k^2} \Psi_j, & \text{if } k \ne j, 
\end{cases}
\end{split}
\label{Wick0}
\end{align}

\noi
where 
$\wick{\Psi_k^\l}$ on the right-hand side denotes
the standard Wick power.
See 
\eqref{psi3} and \eqref{psi4}
for the precise definition
via a regularization and limiting procedure.
See Subsection \ref{SUBSEC:2.4} for a further discussion.
Then, by setting
\begin{align*}
\wick{ u_{N, k}^2 u_{N, j} } 
& \deff 
v_{N, k}^2 v_{N, j}
+ 2 \Psi_k v_{N, k} v_{N, j}
  + v_{N, k}^2 \Psi_j \\
& 
\quad \,
 + \wick{\Psi_k^2} v_{N, j} + 2 v_{N, k} \wick{ \Psi_k \Psi_j } 
+ \wick{\Psi_k^2 \Psi_j} , 
\end{align*}

\noi
we obtain the renormalized version of
\HLS \eqref{NLW1}:
\begin{align}
    (\dt^2 + \dt + m - \Dl) u_{N, j} = - \frac 1N \sum_{k = 1}^N \wick{u_{N, k}^2 u_{N, j}} + \, \sqrt{2} \xi_j.
\label{NLW4}
\end{align}

\noi
As usual in practice
\cite{GKO, GKOT}, 
we say that 
 $\uu_N = \{u_{N, j}\}_{j = 1}^N$ is a solution to \eqref{NLW4} if $u_{N, j}$, $j = 1, \dots, N$,  has the structure \eqref{exp}, 
 where  $\vv_N =  \{v_{N, j}\}_{j = 1}^N$ is  a solution to \eqref{NLW3}.

Before proceeding further, we introduce 
some notations; see also Subsection \ref{SUBSEC:2.1}.
Given $s \in \R$, we set 
\begin{align*}
    \H^s (\T^2) = H^s (\T^2) \times H^{s - 1} (\T^2).
\end{align*}

\noi
 For a function space $B$, we denote by $B^{\otimes N}$  the $N$-fold product of the space $B$ with itself: $B^{\otimes N} = \underbrace{B \times \cdots \times B}_{N \text{ times}} $.
Given 
a Banach space $B$ and 
$N \in \N$, we define 
the $\A_N B$-norm, denoting the $\l^2$-average, by setting
\begin{align}
    \|  a \|_{\A_N B}^2 =  \frac{1}{N} \sum_{j = 1}^N \| a_j \|_B^2
\label{AN0} 
\end{align} 

\noi
for $ a = \{ a_j \}_{j = 1}^N \in B^{\otimes N}$. 
We may write $\A_{N, j}$ to emphasize 
that the $\l^2$-average is taken over the index $j$.
When $B = \R$, we simply set
$\A_N = \A_N \R$.

We  state pathwise local and global well-posedness of \HLS \eqref{NLW4}.

\begin{theorem} 
\label{THM:GWP1}
\textup{(i)}
Let $m > 0$ and $\frac 12 \le s < 1$.
Then, given $N \in \N$, \HLS \eqref{NLW4} is locally well-posed
in $\big(\H^s (\T^2)\big)^{\otimes N}$.
More precisely, 
given $\big(\vv_N^{(0)}, \vv_N^{(1)}\big)  = \big\{\big(v^{(0)}_{N, j}, v^{(1)}_{N, j}\big) \big\}_{j = 1}^N \in 
\big(\H^s (\T^2)\big)^{\otimes N}$, 
there exists 
a unique solution 
\begin{align}
 (\vv_N, \dt \vv_N) 
 = \{(v_{N, j}, \dt v_{N, j})\}_{j = 1}^N
  \in \A_N C([0, \tau];  \H^s(\T^2))
\label{class1}
\end{align}

\noi
  to \eqref{NLW3}
with $(\vv_N, \dt \vv_N)|_{t = 0}
= \big(\vv_N^{(0)}, \vv_N^{(1)}\big)$
on the time interval $[0, \tau]$
for some almost surely positive time $\tau  > 0$.

\medskip

\noi
\textup{(ii)}
Let $m > 0$ and $ \frac 45 < s < 1$.
Then, given $N \in \N$, \HLS \eqref{NLW4} is globally well-posed
in $\big(\H^s (\T^2)\big)^{\otimes N}$.
More precisely, 
given  $\big(\vv_{N}^{(0)},  \vv_{N}^{(1)}\big) = \big\{\big(v_{N, j}^{(0)},  v_{N, j}^{(1)}\big)\big\}_{j = 1}^N \in \big(\H^s (\T^2)\big)^{\otimes N}$, 
 there exists a unique global-in-time solution 
$(\vv_N, \dt \vv_N) \in \big(  C (\R_+; \H^s (\T^2) )\big)^{\otimes N}$
 to~\eqref{NLW3} with initial data
$\big(\vv_{N}^{(0)},  \vv_{N}^{(1)}\big)$, 
almost surely.

\end{theorem}

Theorem \ref{THM:GWP1}
follows from a straightforward modification
of the corresponding results in the scalar case (i.e.~$N = 1$)
in \cite{GKOT}.
We present its proof in Section \ref{SEC:3}.
The local well-posedness result 
(Theorem \ref{THM:GWP1}\,(i))
follows from Proposition \ref{PROP:LWP1}
on local well-posedness of the perturbed SdNLW system \eqref{VN2}
with 
 Lemma \ref{LEM:sto1}
on the regularity properties
of the Wick powers of the stochastic convolution $\Psi_j$
defined in \eqref{psi1}.
Our proof does not make use of any auxiliary function space
and thus the uniqueness of a solution 
$ (\vv_N, \dt \vv_N) $
holds in the entire class \eqref{class1}.
See Remark~\ref{REM:LWP1}\,(i).
We also point out that our argument yields
a uniform (in $N$)
local existence time (in the sense of Proposition \ref{PROP:LWP1}).

Let us now turn our attention 
to 
the  pathwise global well-posedness result 
(Theorem \ref{THM:GWP1}\,(ii)).
Given
$(\uu_N, \dt \uu_N)
= \{(u_{N, j}, \dt u_{N, j})\}_{j = 1}^N$, 
define the energy functional 
$E_N (\uu_N, \dt \uu_N) $
by setting
\begin{align}
\begin{split}
E_N (\uu_N, \dt \uu_N) 
&= \frac{1}{2N} \sum_{j = 1}^N \int_{\T^2} |\nb u_{N, j}|^2 + m u_{N, j}^2 + (\dt  u_{N, j})^2 dx\\
& \quad  + \frac{1}{4N^2} \int_{\T^2} \Big( \sum_{j = 1}^N  u_{N, j}^2 \Big)^2 dx \\
&= \frac{1}{2}  \int_{\T^2} \|\nb  u_{N, j}\|_{\A_N}^2 
+ m \|u_{N, j}\|_{\A_N}^2 + \|\dt  u_{N, j}\|_{\A_N}^2 dx\\
& \quad  + \frac{1}{4} \int_{\T^2} \|  u_{N, j}\|_{\A_N}^4 dx, 
\end{split}
\label{E1}
\end{align}

\noi
where $\A_N = \A_N \R$ is as in \eqref{AN0}.
The energy functional 
$E_N (\uu_N, \dt \uu_N)$ 
is obtained
by simply replacing  the absolute value (= the norm on $\R$)
in the energy:
\begin{align}
E(u, \dt u)
= \frac{1}{2}  \int_{\T^2} |\nb  u|^2
+ m u^2 + (\dt  u)^2 dx + \frac{1}{4} \int_{\T^2} u^4 dx
\label{EW1}
\end{align}

\noi
 for the scalar (deterministic) NLW:
 \begin{align}
 (\dt^2 + m - \Dl) u= - u^3
 \label{EW2}
 \end{align}
 
 \noi
by the $\A_N \R$-norm.
As such, it is easy to check that 
$E_N (\uu_N, \dt \uu_N)$ is conserved
under the following (deterministic) NLW system:
\begin{align}
    (\dt^2 + m - \Dl) u_{N, j} = - \frac{1}{N} \sum_{k = 1}^N u_{N, k}^2 u_{N, j}, 
\qquad     j = 1, \dots, N.
\label{NLWz1}
\end{align}

\noi
With $\ww_N = \dt \uu_N$, we can write \eqref{NLWz1}
as 
the following Hamiltonian formulation:
\begin{align}
\dt \begin{pmatrix}
\uu_N  \\ \ww_N
\end{pmatrix}
= \begin{pmatrix}
0 & \Id_N\\
- \Id_N & 0
\end{pmatrix}
\nb_{(\A_N L^2_x)^{\otimes 2}} E_N(\uu_N, \ww_N), 
\label{NLWz2}
\end{align}

\noi
where $\Id_N$ is the $N\times N$ identity matrix
and 
$\nb_{(\A_N L^2_x)^{\otimes 2}} E_N(\uu_N, \ww_N)$ is the $L^2$-gradient
(= the Riesz representation of the Fr\'echet derivative of 
$E_N(\uu_N, \ww_N)$)
with respect to the $\big(\A_N L^2(\T^2)\big)^{\otimes 2}$-inner product.

As observed in \cite{GKOT}, 
there are two difficulties
in proving global well-posedness
of \eqref{NLW3}
(see Proposition \ref{PROP:GWP1}), 
using the energy 
$E_N (\uu_N, \dt \uu_N) $ in \eqref{E1}.
First of all, 
due to the roughness of the space-time white noises $\{\xi_j \}_{j = 1}^N$, 
the local-in-time solution $(\vv_N, \dt \vv_N)(t)$
constructed in Theorem \ref{THM:GWP1}\,(i)
does {\it not} belong to 
$\big(\H^1 (\T^2)\big)^{\otimes N}$, 
namely
$E_N (\vv_N, \dt \vv_N) (t) = \infty$
for $t > 0$, almost surely.
The second issue comes from the fact that 
$(\vv_N, \dt \vv_N)$
does not solve the NLW system~\eqref{NLWz1}
but it satisfies
the perturbed NLW system \eqref{NLW3}, 
where we view the terms involving (the Wick powers of) $\Psi_j$
as a perturbation.
If each of these difficulties
appears alone, then we could overcome the issue
by applying the $I$-method \cite{CKSTT1} in the former case, 
and by applying a Gronwall-type argument after
Burq and Tzvetkov \cite{BT2} in the latter case.
In our current setting, however, 
we need to handle both of the difficulties at the same time.
Our strategy is to adapt the hybrid approach
combining the $I$-method with a Gronwall-type argument, 
introduced in~\cite{GKOT}
by Gubinelli, Koch, Tolomeo, and the third author, 
to the current vector-valued setting.
In Subsections \ref{SUBSEC:3.2}
and~\ref{SUBSEC:3.3},
we present details of a proof of Theorem \ref{THM:GWP1}\,(ii)
by establishing uniform (in $N$) estimates.
Our argument also yields
a uniform (in $N$) double exponential growth bound; see Remark \ref{REM:growth1}.

\medskip

Next, we turn our attention to the well-posedness
issue of the mean-field SdNLW \eqref{MF1}.
Take $N \to \infty$ in \eqref{NLW3}.
Heuristically, in view of the law of large numbers
(see Subsection~\ref{SUBSEC:5.1}), 
 the empirical average 
$\frac 1N \sum_{1 \leq k \leq N}$ in~\eqref{NLW3}
gets replaced by 
an expectation $\E$.
Thus, in the limit as $N \to \infty$, 
we formally obtain the following limiting mean-field equation:
\begin{align}
(\dt^2 + \dt + m - \Dl) v_j = 
- \E [ v_j^2 ] v_j
- 2 \E [ \Psi_j v_j ] v_j 
- \E [ v_j^2 ] \Psi_j 
- 2 \E [ v_j \Psi_j  ] \Psi_j .
\label{NLW5}
\end{align}

\noi
We note that the fourth and sixth terms on the right-hand side
of \eqref{NLW3} vanish in the limit,
since they (formally)
converge
to 
$\E[ \wick{\Psi_j^2}] v_j $
and $\E[ \wick{\Psi_j^2}] \Psi_j$, respectively,
which are both $0$
in view of 
$\E[ \wick{\Psi_j^2}] = 0$.
By setting  
\begin{align}
    u_j = \Psi_j + v_j,
\label{expj}
\end{align}

\noi
we obtain  the following renormalized version of the mean-field SdNLW \eqref{MF1}:
\begin{align}
    (\dt^2 + \dt + m - \Dl) u_j = - \E [ u_j^2 - \Psi_j^2 ] \,u_j + \sqrt{2} \xi_j.
\label{MF2}
\end{align}

\noi
As in the \HLS case, 
we say that 
 $u_j$ is a solution to \eqref{MF2} if $u_j$ has the structure \eqref{expj}, 
 where $v_j$ is  a solution to \eqref{NLW5}.

Next, we state  local and global well-posedness 
of the  mean-field SdNLW \eqref{MF2}.

\begin{theorem} 
\label{THM:GWP2}
\textup{(i)} Let $m > 0$, $\frac 12 \le s < 1$, 
and $2 \le p < \infty$.
Then, 
the mean-field SdNLW~\eqref{MF2} is locally well-posed
in $L^p (\Omega; \H^s (\T^2) )$.
More precisely, 
given  
  $(v^{(0)}, v^{(1)}) \in L^p (\Omega; \H^s (\T^2) )$, 
there exists a unique solution 
\begin{align}
(v, \dt v) \in L^p(\O; C([0, \tau]; \H^s(\T^2)))
\label{class2}
\end{align}

\noi
to \eqref{NLW5}
with $(v, \dt v)|_{ t= 0} = (v^{(0)}, v^{(1)})$
on the time interval $[0, \tau]$ for some $\tau > 0$.

\medskip

\noi
\textup{(ii)}
Let $m > 0$ and $ \frac 45 < s < 1$.
Then, 
the mean-field SdNLW~\eqref{MF2} is globally well-posed
in $L^2 (\Omega; \H^s (\T^2) )$.
More precisely, 
given  any
  $(v^{(0)}, v^{(1)}) \in L^2 (\Omega; \H^s (\T^2) )$
and $T > 0$,   
there exists a unique global-in-time solution 
$(v, \dt v) \in L^2(\O; C([0, T]; \H^s(\T^2)))$
to \eqref{NLW5}
with $(v, \dt v)|_{ t= 0} = (v^{(0)}, v^{(1)})$.

\end{theorem}

We note that 
 the uniqueness of a solution $(v, \dt v)$ to \eqref{NLW5}
holds in the entire class~\eqref{class2}.
See Remark \ref{REM:LWP2}\,(i).
As in the case of \HLS \eqref{NLW4}
(Theorem \ref{THM:GWP1}), 
Theorem \ref{THM:GWP2}
follows from adapting the argument for the scalar case \cite{GKOT}
to the current mean-field setting.
There is, however, one
seemingly  problematic term
$\E [ v_j \Psi_j  ] \Psi_j $
in \eqref{NLW5} which does not come with renormalization. 
Such a term does not appear in the scalar case \cite{GKOT}
or in \eqref{NLW3}, 
representing 
a new aspect of the current mean-field model, 
 as compared to
the hyperbolic $O(N)$ linear sigma model.
See a proof of 
Proposition \ref{PROP:LWP2}.
See also Subsection \ref{SUBSEC:4.2}
for a  globalization argument
where we encounter similar issues, specific
to the current mean-field setting.

As for the global well-posedness (Theorem \ref{THM:GWP2}\,(ii)), 
the main strategy is to adapt
the hybrid $I$-method approach \cite{GKOT}
to the current mean-field setting, starting
with the following energy functional:
\begin{align}
\begin{split}
\EE (u, \dt  u) 
& = 
\frac 12 \E \bigg[ \int_{\T^2} 
|\nb  u|^2  + m u^2 +  (\dt u)^2 d x \bigg] + \frac 14 \int_{\T^2} \big(\E [u^2]\big)^2 d x\\
& = 
\frac 12  \int_{\T^2} 
\|\nb  u\|_{L^2_\o}^2  + m \|u\|_{L^2_\o}^2 +  \|\dt u\|_{L^2_\o}^2 d x  + \frac 14 \int_{\T^2} 
\|u\|_{L^2_\o}^4 d x.
\end{split}
\label{MFE1}
\end{align}

\noi
Note that 
the energy $\EE (u, \dt  u) $
is obtained
by simply replacing  the absolute value (=~the norm on~$\R$)
in the definition \eqref{EW1} of the energy $E(u, \dt u)$  
 for the scalar (deterministic) NLW~\eqref{EW2}
by the $L^2(\O)$-norm.
It is easy to  check that 
$\EE (u, \dt  u) $ is conserved
under the mean-field NLW:
\begin{align}
    (\dt^2 + m - \Dl) u = - \E[u^2] u .
\label{NLWz3}
\end{align}

\begin{remark}\rm

With $w = \dt u $,  \eqref{NLWz3}
is equivalent to the following Hamiltonian formulation:
\begin{align*}
\dt \begin{pmatrix}
u  \\ w
\end{pmatrix}
= \begin{pmatrix}
0 & 1\\
- 1 & 0
\end{pmatrix}
\nb_{(L^2_{\o, x})^{\otimes 2}} \EE(u,w), 
\end{align*}

\noi
where 
$\nb_{(L^2_{\o, x})^{\otimes 2}}\EE(u, w)$ is the $L^2$-gradient
with respect to the $\big(L^2(\O; L^2(\T^2))\big)^{\otimes 2}$-inner product. 

\end{remark}

\medskip
We now turn our attention to the convergence problem
for \HLS \eqref{NLW4}.
Before proceeding further, 
we first introduce  the notion of convergence in probability at a (given)
rate.

\begin{definition}\label{DEF:conv1}
\rm 
Let $B$ be a Banach space
and $\{r_N\}_{N \in \N}$
be a sequence of positive numbers, 
tending to $0$.
We say that a sequence $\{ X_N\}_{N \in \N}$ of 
$B$-valued random variables converges
in probability to $X$ in $B$ at  the rate $\{r_N\}_{N \in \N}$
as $N \to \infty$,  
if, given any $\dl > 0$, there exist $N_\dl \in \N$ and  $L = L(\dl) > 0$
such that 
\begin{align}
\PP\Big(\| X_N - X\|_B > L r_N \Big) < \dl
\label{PR1}
\end{align}

\noi
for any $N \ge  N_\dl$.
Namely, given any $\dl > 0$, 
there exist 
$N_\dl \in \N$,
$L = L(\dl) > 0$, and a set $\O_{\dl, N} \subset \O$
for each $N \ge N_\dl$
with $\PP(\O_{\dl, N}^c) < \dl$
such that 
\[ \| X_N(\o) - X(\o) \|_B \le L r_N \]

\noi
for each $\o \in \O_{\dl, N}$, $N \ge N_\dl$.

\end{definition}

\begin{remark}\label{REM:conv2}\rm
(i) The  convergence 
\eqref{PR1}
in probability at the rate
$\{r_N\}_{N \in \N}$
is often  denoted by 
\begin{align}
 \|X_N - X\|_B = O_p(r_N)
 \label{PR2}
\end{align}

\noi
as $N \to \infty$.
If we do not assume that $r_N\to 0$, 
then  \eqref{PR2}
is referred to as 
``bounded
in probability at the rate  
$\{r_N\}_{N \in \N}$''.
See, for example, 
\cite[Section~2.2]{Vaart}.
In Definition \ref{DEF:conv1}, 
we additionally assume that $r_N \to 0$
and thus refer to \eqref{PR1}
as ``convergence in probability at the rate $\{r_N\}_{N \in \N}$''.

\medskip

\noi
(ii) Suppose that 
$\{ X_N\}_{N \in \N}$ converges to $X$ in $L^p(\O;B)$
at the rate $\{r_N\}_{N \in \N}$
(tending to $0$)
for some $1 \le p \le \infty$:
\[ \|X_N - X\|_{L^p_\o B} \le r_N\]

\noi
for any $N \in \N$.
Then, 
 \eqref{PR1}
follows from
Chebyshev's inequality
and taking $L = L(\dl) > 0$ sufficiently large.
Thus, convergence in $L^p(\O)$
at the rate $\{r_N\}_{N \in \N}$
implies convergence in probability
with the rate $\{r_N\}_{N \in \N}$
as defined in Definition \ref{DEF:conv1}.
In this sense, Definition \ref{DEF:conv1} is a natural extension
of the usual convergence in probability (without a rate)
which follows from convergence in  $L^p(\O)$
(without a rate).

%
%
%
%
%

\end{remark}

Finally, we are ready to state convergence results
for \HLS \eqref{NLW4} to the mean-field SdNLW \eqref{MF2}
with general initial data (under some assumptions).

\begin{theorem}
\label{THM:conv1}

\textup{(i)}
Let $m > 0$ and  $\frac 45 <  s < 1$.
Let 
$\big\{ \big(u_{N, j}^{(0)}, u_{N, j}^{(1)}\big)\big\}_{j = 1}^N
\in \big(\H^s (\T^2)\big)^{\otimes N}$, $N \in \N$, 
and 
$\big\{\big(u_{j}^{(0)}, u_{ j}^{(1)}\big)\big\}_{j \in \N}
\in \big(\H^s (\T^2)\big)^{\otimes \N}$
be 
 sequences of random functions
such that 

\smallskip

\begin{itemize}
\item
$\big(u_{j}^{(0)}, u_{ j}^{(1)}\big)
\in L^2(\O; \H^s (\T^2))$, $j \in \N$, 
 are independent and identically distributed,

\smallskip

\item 
for each fixed $j \in \N$, we have 
\begin{align}
 \big\| \big( u_{N, j}^{(0)}, u_{N, j}^{(1)}\big) -  \big(u_{j}^{(0)},  u_{j}^{(1)} \big) \big\|_{\H^s (\T^2)} 
& \too 0  
\label{IV1}
\end{align}

\noi
in probability as $N \to \infty$, and 
\begin{align}
 \big\| 
\big( u_{N, j}^{(0)}, u_{N, j}^{(1)}\big)  - \big(u_{j}^{(0)},  u_{j}^{(1)} \big)
 \big\|_{\A_N \H^s (\T^2)} 
 & \longrightarrow 0
\label{IV2}
\end{align}

\noi
in probability as $N \to \infty$.

\end{itemize}

\smallskip

\noi
Let 
$\{ (u_{N, j}, \dt  u_{N, j})\}_{j = 1}^N$
and $\{(u_j, \dt u_j)\}_{j \in \N}$
be the global-in-time solutions to \eqref{NLW4} 
and \eqref{MF2}
with initial data $\big\{ \big(u_{N, j}^{(0)}, u_{N, j}^{(1)}\big)\big\}_{j = 1}^N$
and $\big\{\big(u_{j}^{(0)}, u_{ j}^{(1)}\big)\big\}_{j \in \N}$
constructed in Theorems \ref{THM:GWP1}
and~\ref{THM:GWP2}, respectively.
Then, given any $T > 0$,  we have, for each fixed $j \in \N$, 
\begin{align}
\| (u_{N, j}, \dt u_{N, j})   - (u_j, \dt u_j) \|_{ C_T  \H^s_x}
\too 0
\label{IV3}
\end{align}

\noi
in probability as $N \to \infty$, and 
\begin{align}
\| (u_{N, j}, \dt u_{N, j})   - (u_j, \dt u_j) \|_{\A_{N} C_T  \H^s_x}
\too 0
\label{IV4}
\end{align}

\noi
in probability as $N \to \infty$.

\medskip

\noi
\textup{(ii)}
Let $m > 0$ and  $\frac 12 \le s < 1$.
Let 
$\big\{ \big(u_{N, j}^{(0)}, u_{N, j}^{(1)}\big)\big\}_{j = 1}^N
\in \big(\H^s (\T^2)\big)^{\otimes N}$, $N \in \N$, 
and 
$\big\{\big(u_{j}^{(0)}, u_{ j}^{(1)}\big)\big\}_{j \in \N}
\in \big(\H^s (\T^2)\big)^{\otimes \N}$ 
be sequences of random functions
such that 

\smallskip

\begin{itemize}
\item
$\big(u_{j}^{(0)}, u_{ j}^{(1)}\big)
\in L^6(\O; \H^s (\T^2))$, $j \in \N$, 
 are independent and identically distributed,

\smallskip

\item 
for each fixed $j \in \N$, we have 
\begin{align}
 \big\| \big( u_{N, j}^{(0)}, u_{N, j}^{(1)}\big) - \big(u_{j}^{(0)},  u_{j}^{(1)} \big) \big\|_{\H^s (\T^2)} 
& \too 0  
\label{IV5}
\end{align}

\noi
in probability at the rate $N^{-\frac 12}$ as $N \to \infty$, and 
\begin{align}
 \big\| 
\big( u_{N, j}^{(0)},  u_{N, j}^{(1)} \big) -\big( u_{j}^{(0)}, u_{j}^{(1)} \big)
 \big\|_{\A_N \H^s (\T^2)} 
 & \longrightarrow 0
\label{IV6}
\end{align}

\noi
in probability 
at the rate $N^{-\frac 12}$
as $N \to \infty$.

\end{itemize}

\noi
Let 
$\{ (u_{N, j}, \dt  u_{N, j})\}_{j = 1}^N$
and $\{(u_j, \dt u_j)\}_{j \in \N}$
be the local-in-time solutions to \eqref{NLW4} 
and \eqref{MF2}
with initial data $\big\{ \big(u_{N, j}^{(0)}, u_{N, j}^{(1)}\big)\big\}_{j = 1}^N$
and $\big\{\big(u_{j}^{(0)}, u_{ j}^{(1)}\big)\big\}_{j \in \N}$
  constructed in Theorems \ref{THM:GWP1}
and~\ref{THM:GWP2}, respectively.
Suppose that $u_j$, $j \in \N$,  exists on the time interval $[0, T]$
for some $T > 0$
and that 
\begin{align}
v_j  = u_j - \Psi_j \in L^6 (\Omega; C([0, T]; H^{s} (\T^2))), \qquad j \in \N.
\label{IV6a}
\end{align}

\noi
Then,    for each fixed $j \in \N$, we have
\begin{align*}
 \| (u_{N, j}, \dt u_{N, j})   - (u_j, \dt u_j) \|_{ C_T  \H^s_x}
\too 0
\end{align*}

\noi
in probability 
at the rate $N^{-\frac 12}$
as $N \to \infty$, and we have
\begin{align*}
  \| (u_{N, j}, \dt u_{N, j})   - (u_j, \dt u_j) \|_{\A_{N} C_T  \H^s_x}
\too 0
\end{align*}

\noi
in probability 
at the rate $N^{-\frac 12}$
as $N \to \infty$.

\end{theorem}

Theorem~\ref{THM:conv1}\,(i)
establishes global-in-time convergence
of \HLS \eqref{NLW4} to the mean-field SdNLW
\eqref{MF2}
but it does not provide a rate of convergence.
On the other hand, 
Theorem~\ref{THM:conv1}\,(ii)
only establishes local-in-time convergence
but it provides an explicit  convergence rate of $N^{- \frac12}$.
Here, the rate $N^{- \frac 12}$ essentially  comes from the convergence
rate of the empirical averages
of the purely stochastic terms; see, for example,  Lemma \ref{LEM:LLN1}. 
In view of Theorem \ref{THM:GWP2}\,(i), 
we note that 
 the higher moment assumption~\eqref{IV6a}
indeed holds true, at least for small $T > 0$
(and hence is not an assumption
for short times).
We point out that, at the Gibbs equilibria, 
we can extend the  local-in-time convergence in probability 
at the rate $N^{-\frac 12}$
in Theorem \ref{THM:conv1}\,(ii)
to a global one; see Theorem \ref{THM:conv2}\,(iii).
In Table \ref{TAB:1}, 
we compare and summarize
 the convergence results
  in Theorems \ref{THM:conv1} and \ref{THM:conv2}.

\vspace{1.5mm}

\begin{table}[h]
  \begin{tabular}{|c||c|c|c|c|c|}
\hline
&  
\rule{0mm}{4mm} 
Theorem \ref{THM:conv1}\,(i)
& Theorem \ref{THM:conv1}\,(ii)
& Theorem \ref{THM:conv2}\,(iii)
\\
\hline
 time interval & 
\rule[-1.5mm]{0mm}{5.5mm}
global in time & local in time 
& global in time\\
\hline
rate & 
\rule[-1.5mm]{0mm}{5.5mm}
$\times$ &  $N^{-\frac 12}$
&  $N^{-\frac 12}$
\\
\hline
\rule[-2.2mm]{0mm}{6.2mm}
initial data& 
general 
&general  with the $L^6(\O)$-assumption
& 
Gibbsian\\
\hline
  \end{tabular}

\vspace{4mm}
  \caption{Summary of the convergence results
  in Theorems \ref{THM:conv1} and \ref{THM:conv2}}
\label{TAB:1}
\end{table}

We present a proof of 
Theorem~\ref{THM:conv1} in Section~\ref{SEC:conv}. 
As compared to the parabolic case studied in \cite{SSZZ}, 
the current hyperbolic problem presents several new difficulties 
due to the lack of the strong parabolic smoothing
and requires a new strategy.

\smallskip

\begin{itemize}
\item
First of all, the solution $(v_j(t), \dt v_j(t))$
to the mean-field SdNLW \eqref{NLW5}
does {\it not} belong to $\H^1(\T^2)$
and hence 
the energy $\EE(v_j(t), \dt v_j(t))$ in~\eqref{MFE1}
is infinite.
This forces
us to work with  the modified energy
$\EE(Iv_j, \dt I v_j)$ in~\eqref{WW1}, 
where $I$ denotes the $I$-operator defined in \eqref{I2}).
Compare this with the parabolic case \cite{SSZZ}, 
where the residual part (for the solution $u_j$ to \eqref{NLH2} after the first order expansion)
belongs to $H^1(\T^2)$.

\smallskip

\item 
The main difficulty for proving 
the global-in-time convergence result (Theorem~\ref{THM:conv1}\,(i))
comes from the lack of a higher moment control
on a global solution $v_j$ 
to the mean-field SdNLW \eqref{NLW5}.
Namely, 
due to the use of the  modified energy
$\EE(Iv_j, \dt I v_j)$ in~\eqref{WW1}
which comes with the $L^2(\O)$-norm, 
Theorem \ref{THM:GWP2}\,(ii)
only provides  the second moment bound
for the residual part $v_j = u_j - \Psi_j$.
This is in sharp contrast to the 
 parabolic setting studied in~\cite{SSZZ},
 where a higher moment bound
 is available 
 thanks to the coercive structure of the parabolic equations;
see \cite[Lemma~3.4]{SSZZ}.

\end{itemize}

\smallskip

\noi
We overcome the latter difficulty 
in the current hyperbolic\,/\,dispersive problem
by directly establishing 
 law-of-large-numbers type lemmas
 (Lemmas \ref{LEM:LLN1}, \ref{LEM:LLN2}, and \ref{LEM:LLN3})
in Subsection~\ref{SUBSEC:5.1}, 
assuming only the second moment bound on $v_j$.
We note that, due to the lack of a higher moment control, 
we can not rely on a simple orthogonality argument (as in Lemmas \ref{LEM:LLN4} and~\ref{LEM:LLN5};
see a discussion below)
and thus we need to proceed with more care, 
  invoking the (strong) law of large numbers
in the Banach space setting; see \cite[Chapter 7]{LT}.
Once we prove
the key 
 law-of-large-numbers type lemmas
(Lemmas \ref{LEM:LLN1}, \ref{LEM:LLN2}, and \ref{LEM:LLN3}), 
Theorem~\ref{THM:conv1}\,(i)
follows from 
iteratively applying a modification of the pathwise local well-posedness
argument 
with Lemmas \ref{LEM:LLN1}, \ref{LEM:LLN2}, and \ref{LEM:LLN3};
see Subsection~\ref{SUBSEC:5.2} for details.
Note that, in the parabolic case~\cite{SSZZ}, 
the convergence was shown by 
an energy estimate
which is not available
in the current hyperbolic\,/\,dispersive setting
(recall that  the energy is almost surely infinite
since $(v_{N, j}(t), \dt v_{N, j}(t)) \notin \H^1(\T^2)$, 
almost surely).

We point out that, in handling the term $I_3^N$ in \cite[(4.5)]{SSZZ}, 
the authors also established a law-of-large-numbers type result;
see \cite[Step 4 on p.\,163]{SSZZ}. 
Their argument, however,  relies on a square root gain 
via orthogonality
(see \cite[(4.1)]{SSZZ}) 
which requires a higher moment bound. 
In proving 
Theorem~\ref{THM:conv1}\,(ii) 
under the higher moment assumption,  
we also 
establish 
 law-of-large-numbers type lemmas
(Lemmas \ref{LEM:LLN4} and \ref{LEM:LLN5})
but their proofs are much simpler
than those in Subsection \ref{SUBSEC:5.1}
thanks to the higher moment control, 
which allows us to 
exploit orthogonality
and obtain a square root gain
just as in \cite{SSZZ}.
See Subsection \ref{SUBSEC:5.3}.

\begin{remark} \rm
 The results stated in this subsection
(Theorems \ref{THM:GWP1},  \ref{THM:GWP2},
and \ref{THM:conv1})
hold for general initial data (under suitable assumptions).
In particular, we do not make use of invariant Gibbs measures.
As such, 
the damping term $\dt u_{N, j}$
and 
$\sqrt 2$ on the noises in \eqref{NLW4}
(and also in  \eqref{NLW3}, \eqref{NLW5}, and \eqref{MF2})
do not play any role, 
and the corresponding results
hold for the undamped case.
We also note that
 these results also hold for the case  $m = 0$ (without the damping term), where a separate (but not difficult) analysis at the frequency $0$ is required; see, for example, \cite{GKO}.
In the next subsection, 
we consider invariant Gibbs dynamics, 
for which the damping term, 
$\sqrt 2$ on the noises, 
and the positivity of $m$ are all necessary.
We chose to treat the damped case
in this subsection 
such that we have a uniform presentation throughout the paper.

\end{remark}

\begin{remark}\label{REM:conv3}\rm
In \cite{SSZZ}, 
the authors proved 
global-in-time $L^2(\O)$-convergence at the rate $N^{-\frac 12}$ of solutions
to 
the coupled system of SNLH
\eqref{NLH1} to 
a solution to 
the mean-field SNLH \eqref{NLH2}
(at the level of  the residual
terms
after the first order expansion 
with respect to a suitable space-time function space);
see, for example,  \cite[(5.14)]{SSZZ}.\footnote{We note that 
\cite[(5.14)]{SSZZ} is for invariant dynamics.
While not stated explicitly in \cite{SSZZ}, 
we expect that, in \cite{SSZZ},  
global-in-time $L^2(\O)$-convergence at the rate $N^{-\frac 12}$
also holds
for non-invariant dynamics (under an appropriate assumption on initial data).
}
In view of 
Remark \ref{REM:conv2}\,(ii),
this result implies
convergence in probability at the  rate $N^{-\frac 12 }$,
which agrees with Theorem \ref{THM:conv1}\,(ii)
for \HLS~\eqref{NLW4}, at least locally in time.

In \cite{SSZZ}, 
the proof is based on an energy estimate, 
computing the growth of  the $H^1$-energy functional of the residual part, 
which directly provides
a global-in-time control.
Such an argument is not possible
in the current 
wave context, 
since 
we do not have 
 higher moment bounds
 for large times
due to the lack of the strong parabolic smoothing.
As mentioned above, we instead prove Theorem \ref{THM:conv1}
by iterating a local-in-time argument.
See Section~\ref{SEC:conv} for details.

\end{remark}

\subsection{Large $N$ limit of the invariant Gibbs dynamics} 
\label{SUBSEC:1.3}

In this subsection, 
we consider
 the hyperbolic $O(N)$ linear sigma model \eqref{NLW4} and its mean-field limit \eqref{MF2}
at Gibbs equilibria.

The Gibbs measure $\rhoo$ for the 
coupled NLW system
\eqref{NLWz1}
is formally given by
\begin{align*}
d\rhoo 
(\uu_N, \dt \uu_N)
= Z_N^{-1} e^{-N\cdot  E_N (\uu_N, \dt \uu_N) } 
d \uu_N d (\dt \uu_N),  
\end{align*}

\noi
where
$(\uu_N, \dt \uu_N)
= \{(u_{N, j}, \dt u_{N, j})\}_{j = 1}^N$
and 
$E_N (\uu_N, \dt \uu_N) $ is as in \eqref{E1}.
Let us rewrite $\rhoo$ as a weighted Gaussian measure in the following.
Given $m > 0$ and $s \in \R$, 
define the fractional 
Gaussian field  $\mu_s$ with mass $m > 0$:
\begin{align}
    d \mu_{s} = Z^{-1} \exp\bigg(-\frac 12  \| (m - \Dl)^\frac s2 u \|_{L^2}^2 \bigg) du.
\label{gauss1}
\end{align}

\noi
Recall that $\mu_s$ is a Gaussian probability measure
on $H^{s - 1-\eps}(\T^2) \setminus H^{s - 1}(\T^2)$.
When $s = 1$, 
it corresponds to the massive Gaussian free field $\mu_1$
defined in \eqref{gauss0}, 
while it corresponds to the white noise measure when $s = 0$.
With this notation, 
we can write $\rhoo$ as 
\begin{align}
\begin{split}
& d \rhoo 
(\uu_N, \dt \uu_N) \\
& \quad = Z_N^{-1} \exp \bigg( - \frac{1}{4N} \int_{\T^2} 
\Big( \sum_{j = 1}^N u_{N, j}^2 \Big)^2 dx \bigg) 
d \mu_1^{\otimes N} 
    (\uu_N) 
\otimes d  \mu_0^{\otimes N} (\dt \uu_N) \\
& \quad = 
d \rho_N
    (\uu_N) 
\otimes d  \mu_0^{\otimes N} (\dt \uu_N), 
\end{split}
\label{Gibbs}
\end{align}

\noi
where
$\rho_N$
is the 
$O(N)$ linear sigma model in \eqref{Gibbs1}.
Using the notation in \eqref{NLWz2}
(with $\ww_N = \dt \uu_N$), 
let us now rewrite 
the (formal)\footnote{Namely, without renormalization.}   \HLS \eqref{NLW1}
as 
\begin{align*}
\dt \begin{pmatrix}
\uu_N  \\ \ww_N
\end{pmatrix}
= \begin{pmatrix}
0 & \Id_N\\
- \Id_N & 0
\end{pmatrix}
\nb_{(\A_N L^2_x)^{\otimes 2}} E_N(\uu_N, \ww_N) 
    + \begin{pmatrix}
        0 \\ - \ww_N + \sqrt{2}\, \pmb\xi_N
    \end{pmatrix}, 
\end{align*}

\noi
where $\pmb \xi_N = \{\xi_j\}_{j = 1}^N$.
Namely, 
\HLS \eqref{NLW1}
is written as a superposition
of the coupled NLW system \eqref{NLWz1}
and 
an Ornstein-Uhlenbeck process (for the $\ww_N$-component), 
both of which (formally) preserve
the Gibbs measure $\rhoo$.
As a result, we see that 
the Gibbs measure $\rhoo$
in~\eqref{Gibbs}
is (formally) invariant under the dynamics of 
\HLS \eqref{NLW1}
(which can be made rigorous with a renormalization
and a frequency-truncation; see \cite[Section 4]{GKOT}).

Note that $\mu_1^{\otimes N}$ 
in \eqref{Gibbs} is supported on (vectors of) distributions
of negative regularity.
Namely, the interaction potential in \eqref{Gibbs}
requires a renormalization.
At the level of 
the 
$O(N)$ linear sigma model, this leads to 
\begin{align}
d \rho_N (\uu_N)
&  = Z_N^{-1} \exp \bigg( - \frac{1}{4N} \int_{\T^2} 
\wick{\Big( \sum_{j = 1}^N u_{N, j}^2 \Big)^2} dx \bigg) 
d \mu_1^{\otimes N} 
    (\uu_N) , 
\label{Gibbsx}
\end{align}

\noi
where 
\begin{align*}
\wick{\Big( \sum_{j = 1}^N u_{N, j}^2 \Big)^2} 
= 
\sum_{k, j = 1}^N 
\wick{u_{N, k}^2 u_{N, j}^2 }
\end{align*}

\noi
with 
\begin{align*}
\begin{split}
 \wick{ u_{N, k}^2u_{N, j}^2 } & = 
\begin{cases}
\wick{u_{N, k}^4}\,, & \text{if } k = j, \\
\wick{u_{N, k}^2}\, \wick{u_{N, j}^2}, & \text{if } k \ne j.
\end{cases}
\end{split}
\end{align*}

\noi
Here, 
$\wick{u_k^\l}$ on the right-hand side denotes
the standard Wick power.
By adapting the argument from the scalar case
(\cite{Nelson2, DPT, OTh}) to the current vector-valued setting, 
we can realize $\rho_N$ in~\eqref{Gibbsx}
as a probability measure
on $\big(H^{-\eps} (\T^2)\big)^{\otimes N}$;
see also \cite[Remark~5.8]{SSZZ}.
The renormalized Gibbs measure $\rhoo
= \rho_N \otimes \mu_0^{\otimes N}$ 
(with $\rho_N$ in \eqref{Gibbsx}) is then given by 
\begin{align}
\begin{split}
& d \rhoo 
(\uu_N, \dt \uu_N) \\
& \quad = Z_N^{-1} \exp \bigg( - \frac{1}{4N} \int_{\T^2} 
\wick{\Big( \sum_{j = 1}^N u_{N, j}^2 \Big)^2} dx \bigg) 
d \mu_1^{\otimes N} 
    (\uu_N) 
\otimes d  \mu_0^{\otimes N} (\dt \uu_N).
\end{split}
\label{Gibbsy}
\end{align}

Let us now turn our attention to the dynamical problem.
Let 
$\big\{ \big(\phi^{(0)}_j, \phi^{(1)}_j\big)\big\}_{j \in \N}$
be  a family of independent $\D'(\T^2)$-valued random variables, 
independent of the family $\{\xi_j\}_{j \in \N}$
of the space-time white noises, 
satisfying 
\begin{align}
\Law \big(\phi^{(0)}_j\big) = \mu_1
\qquad \text{and}\qquad  
\Law \big(\phi^{(1)}_j\big) = \mu_0, \quad j \in \N, 
\label{gauss2}
\end{align}

\noi
where $\mu_s$ is as in 
\eqref{gauss1}.
Here, $\Law(X)$ denotes the law of a random variable $X$.
Given $j \in \N$, 
let 
 $\Phi_j$ be the solution to the following linear stochastic damped wave equation:
\begin{align}
\begin{cases}
    (\dt^2 + \dt + m - \Dl) \Phi_j = \sqrt{2} \xi_j \\
    (\Phi_j, \dt \Phi_j) |_{t = 0} = \big(\phi^{(0)}_j, \phi^{(1)}_j\big).
\end{cases}
\label{phi1}
\end{align}

\noi
  A straightforward computation shows that 
\begin{align*}
\Law (\Phi_j (t) , \dt \Phi_j(t)) = \Law \big(\phi^{(0)}_j, \phi^{(1)}_j\big) = \mu_1 \otimes \mu_0
\end{align*}

\noi
for any $t \in \R_+$.

Our goal is to study the convergence issue
of (the renormalized version of) \HLS \eqref{NLW1} with the Gibbsian initial data:
\begin{align}
\Law \Big( \{(u_{N, j}(0), \dt u_{N, j}(0))\}_{j = 1}^N \Big) = \rhoo
= \rho_N \otimes \mu_0^{\otimes N}, 
\label{law1}
\end{align}

\noi
which 
we assume to be independent of the family $\{\xi_j\}_{j \in \N}$
of the space-time white noises.
Moreover, we impose that 
\begin{align}
\dt u_{N, j}(0) = \phi_j^{(1)}, \quad j = 1, \dots, N.
\label{law2}
\end{align}

\noi
By  the first order expansion:
\begin{align*}
    u_{N, j} = \Phi_j + v_{N, j},
\end{align*}

\noi
 the residual term $v_{N, j}$ satisfies
\begin{align}
\begin{split}
(\dt^2 + \dt + m - \Dl) v_{N, j} 
& = - \frac{1}{N} \sum_{k = 1}^N \Big( 
v_{N, k}^2 v_{N, j}
+ 2 \Phi_k v_{N, k} v_{N, j}
  + v_{N, k}^2 \Phi_j \\
& 
\hphantom{XXXXX}
 + \wick{\Phi_k^2} v_{N, j} + 2 v_{N, k} \wick{ \Phi_k \Phi_j } 
+ \wick{\Phi_k^2 \Phi_j} \Big)\\
\end{split}
\label{NLW6}
\end{align}

\noi
with initial data (recall \eqref{law2})
\begin{align*}
(v_{N, j}, \dt v_{N, j})|_{t = 0} = 
\big(v_{N, j}^{(0)}, 0\big)
= 
\big(u_{N, j}(0) - \phi_j^{(0)}, 0\big), 
\end{align*}

\noi
where we have already applied Wick renormalization 
to the powers of the stochastic convolution~$\Phi_j$;
 see \eqref{def_wick2} for the more precise definition. 
Then, by setting
\begin{align*}
\wick{ u_{N, k}^2 u_{N, j} } 
& \deff 
v_{N, k}^2 v_{N, j}
+ 2 \Phi_k v_{N, k} v_{N, j}
  + v_{N, k}^2 \Phi_j \\
& 
\quad \,
 + \wick{\Phi_k^2} v_{N, j} + 2 v_{N, k} \wick{ \Phi_k \Phi_j } 
+ \wick{\Phi_k^2 \Phi_j} , 
\end{align*}

\noi
we obtain the renormalized version of
\HLS \eqref{NLW1}:
\begin{align}
    (\dt^2 + \dt + m - \Dl) u_{N, j} = - \frac{1}{N} \sum_{k = 1}^N \wick{ u_{N, k}^2 u_{N, j} } + \sqrt{2} \xi_j
\label{SdNLW1}
\end{align}

\noi
with the Gibbsian initial data \eqref{law1}.
In the current setting, the limiting mean-field equation is given by 
\begin{align}
    (\dt^2 + \dt + m - \Dl) u_j = - \E [ u_j^2 - \Phi_j^2 ] \, u_j + \sqrt{2} \xi_j.
\label{MFx}
\end{align}

\noi
Compare \eqref{MFx} with \eqref{MF2}.
Consider \eqref{MFx}
with  initial data $\big(\phi^{(0)}_j, \phi^{(1)}_j\big)$.
We first note that  $u_j \equiv \Phi_j$
satisfies \eqref{MFx}
with $(u_j, \dt u_j) = \big(\phi^{(0)}_j, \phi^{(1)}_j\big)$.
Under the first order expansion
$u_j = \Phi_j + v_j$, 
this implies that 
$(v_j, \dt v_j) \equiv (0, 0)$
is a global-in-time solution to 
\begin{align}
(\dt^2 + \dt + m - \Dl) v_j = 
- \E [ v_j^2 ] v_j
- 2 \E [ \Phi_j v_j ] v_j 
- \E [ v_j^2 ] \Phi_j 
- 2 \E [ v_j \Phi_j  ] \Phi_j 
\label{MFy}
\end{align}

\noi
with zero initial data,
which trivially implies
\begin{align}
(v_j, \dt v_j)  \in 
L^p(\O; C([0, T]; \H^s(\T^2)))
\label{class5}
\end{align}

\noi
for any finite $p\ge 1$, $s \in \R$,  and $T > 0$.
By noting that  Theorem \ref{THM:GWP2}
also applies to \eqref{MFx} and~\eqref{MFy}
and iteratively applying
Theorem \ref{THM:GWP2}\,(i), 
we then conclude that 
$(u_j, \dt u_j)  \equiv (\Phi_j, \dt \Phi_j) $ is the unique solution
in the class:
\begin{align}
( \Phi_j, \dt \Phi_j)  + L^p(\O; C([0, T]; \H^s(\T^2)))
 \label{class4}
\end{align}

\noi
for any finite $p \ge 2$, 
$\frac 12 \le s < 1$, 
and $T > 0$.

We  now  state a convergence  result 
of the invariant Gibbs dynamics
of \HLS \eqref{SdNLW1}.

\begin{theorem}
\label{THM:conv2}

\textup{(i)} 
Let $m > 0$ and $ \frac 45 < s < 1$.
Then, given $N \in \N$, \HLS \eqref{SdNLW1} is globally well-posed
in $\big(\H^s (\T^2)\big)^{\otimes N}$.
More precisely, 
given  $\big(\vv_{N}^{(0)},  \vv_{N}^{(1)}\big) = \big\{\big(v_{N, j}^{(0)},  v_{N, j}^{(1)}\big)\big\}_{j = 1}^N \in \big(\H^s (\T^2)\big)^{\otimes N}$, 
 there exists a unique global-in-time solution 
\begin{align*}
(\vv_N, \dt \vv_N) 
= \{(v_{N, j}, \dt v_{N, j})\}_{j = 1}^N
\in \big(  C (\R_+; \H^s (\T^2) )\big)^{\otimes N}
\end{align*}
 to~\eqref{NLW6} with initial data
$\big(\vv_{N}^{(0)},  \vv_{N}^{(1)}\big)$, 
almost surely.

\medskip

\noi
\textup{(ii)} 
Let $m > 0$. Then, for each $N \in \N$, the Gibbs measure $\rhoo = \rho_N \otimes \mu_0^{\otimes N}$ 
defined in~\eqref{Gibbsy} is invariant under the dynamics of~\eqref{SdNLW1}. More precisely, for each $t \in \R_+$, we have 
\[\Law 
 \Big( \{(u_{N, j}(t), \dt u_{N, j}(t))\}_{j = 1}^N \Big) = \rhoo.\]

\medskip

\noi
\textup{(iii)} 
Let $m > 0$.
There exists a probability space $(\O', \F', \PP')$
such that 
the following statements hold\textup{:}

\smallskip
\begin{itemize}
\item[(iii.a)]
There exists a family 
$\big\{ \big(\phi^{(0)}_j, \phi^{(1)}_j\big)\big\}_{j \in \N}$
of 
 independent $\D'(\T^2)$-valued random variables
 on $\O'$, 
independent of the family $\{\xi_j\}_{j \in \N}$
of the space-time white noises, 
satisfying \eqref{gauss2}.

\smallskip
\item [(iii.b)]
For each $N \in \N$, 
there exists  a family $\big\{v^{(0)}_{N, j}\big\}_{j = 1}^N$
of $H^1$-valued 
random variables
 on $\O'$, 
independent of the family $\{\xi_j\}_{j \in \N}$
of the space-time white noises, satisfying
\begin{align}
\sup_{j =1, \dots, N}\Big\|  \big\|v^{(0)}_{N, j} \big\|_{H^1_x}\Big\|_{L^2(\O')}
 \les N^{-\frac 12 } , 
\label{randx}
\end{align}

\noi
such that 
\[
\Law \Big( \big\{\big(v^{(0)}_{N, j} + \phi^{(0)}_j, \phi^{(1)}_j\big)\big\}_{j = 1}^N\Big)
= \rhoo.\]

\smallskip
\item [(iii.c)]
Let 
$\{ (u_{N, j}, \dt  u_{N, j})\}_{j = 1}^N$
be the global-in-time solution to \eqref{SdNLW1} 
with the Gibbsian initial data 
$\big\{\big(v^{(0)}_{N, j} + \phi^{(0)}_j, \phi^{(1)}_j\big)\big\}_{j = 1}^N$
constructed in \textup{(iii.b)}.
Fix 
 $\frac 12 \le s < 1$ and 
$T > 0$.
Then,    for each fixed $j \in \N$, 
we have 
\begin{align}
 \| (u_{N, j}, \dt u_{N, j})   - (\Phi_j, \dt \Phi_j) \|_{ C_T  \H^s_x}
\too 0
\label{IV11}
\end{align}

\noi
in probability at the rate $N^{-\frac 12}$ as $N \to \infty$, and we have 
\begin{align}
  \| (u_{N, j}, \dt u_{N, j})   - (\Phi_j, \dt \Phi_j) \|_{\A_{N, j} C_T  \H^s_x}
\too 0
\label{IV12}
\end{align}

\noi
in probability at the rate $N^{-\frac 12}$  as $N \to \infty$.

\end{itemize}

\end{theorem}

Theorem \ref{THM:conv2}\,(iii)
implies propagation of chaos
for \HLS \eqref{SdNLW1}
in the following sense.
In view of \eqref{randx}
and the independence of $\big\{ \big(\phi^{(0)}_j, \phi^{(1)}_j\big)\big\}_{j \in \N}$, 
the Gibbsian initial data
$\big\{\big(v^{(0)}_{N, j} + \phi^{(0)}_j, \phi^{(1)}_j\big)\big\}_{j = 1}^N$
(which is coupled)
is asymptotically chaotic
(namely, components become independent  as $N \to \infty$).
The convergence claims \eqref{IV11} and \eqref{IV12}
show that the solution 
$\{ (u_{N, j}(t), \dt  u_{N, j}(t))\}_{j = 1}^N$
is also 
asymptotically chaotic
for each $t > 0$.
See 
\cite[Definition 2.1 on p.\,177 and Theorem 3.1 on p.\,203]{Sznit}
for a further discussion on propagation of chaos.

Noting that the regularity properties of $\Phi_j$
are essentially the same as those of $\Psi_j$
(see Lemmas~\ref{LEM:sto1} and~\ref{LEM:sto2}), 
Theorem~\ref{THM:conv2}\,(i) follows from a slight modification
of the proof of 
Theorem~\ref{THM:GWP1} presented in Section \ref{SEC:3}
and thus we omit details.
As for 
Theorem~\ref{THM:conv2}\,(ii), 
we first note that Bourgain's invariant measure
argument \cite{BO94, BO96} is not needed here
in view of  the pathwise global well-posedness
(Theorem \ref{THM:conv2}\,(i)).
The claimed invariance
of the Gibbs measure $\rhoo$
under \eqref{SdNLW1}
follows
from 
the invariance of the frequency-truncated version
of the Gibbs measure $\rhoo$
and a PDE approximation argument; 
see, for example, \cite{GKOT, ORTz}.
As this is standard in the field by now,  we omit details.

Theorem~\ref{THM:conv2}\,(iii.c)
establishes convergence in probability 
at the rate $N^{-\frac 12}$
on an arbitrarily long time interval $[0, T]$, 
thus extending Theorem \ref{THM:conv1}\,(ii)
globally in time in the special case
of the invariant Gibbs dynamics.
The construction of the initial data sets claimed in 
Theorem~\ref{THM:conv2}\,(iii.a) and (iii.b)
is based on the results in  
\cite{SSZZ, DS} (see Lemma \ref{LEM:DS})
and follows from a key proposition (Proposition~\ref{PROP:rand}).
We present details in Subsection \ref{SUBSEC:6.1}.
In view of \eqref{randx}
and the (trivial) higher moment bound
\eqref{class5}, 
Theorem~\ref{THM:conv2}\,(iii.c)
 follows from a
 straightforward modification
 of the proof of Theorem \ref{THM:conv1}\,(ii).
See  Subsection \ref{SUBSEC:6.2} for details.

\begin{remark} \rm
In \cite{Tolo2}, 
Tolomeo  proved ergodicity of the invariant Gibbs measures 
for the defocusing SdNLW on $\T^2$.
A slight variation of the argument  in \cite{Tolo2} 
yields ergodicity of the Gibbs measure $\rhoo$
for \HLS \eqref{SdNLW1}.
 It would be of interest to study ergodic properties of the invariant Gaussian measure 
 $\mu_1 \otimes \mu_0$ for the mean-field SdNLW \eqref{MFx}
 on $\T^2$.
 Note that such a result is known in the parabolic case;
 see \cite{SSZZ}.
\end{remark}

\begin{remark}\rm
With a slight modification of the proof, Theorem \ref{THM:conv2} 
extends to 
the NLW system~\eqref{NLWz1} (with a proper renormalization)
with the Gibbsian initial data in~\eqref{law1}, 
converging
to its mean-field limit \eqref{NLWz3}, 
which is nothing but the linear dynamics:
\begin{align*}
\begin{cases}
    (\dt^2 +  m - \Dl) Z_j = 0\\
    (Z_j, \dt Z_j) |_{t = 0} = \big(\phi^{(0)}_j, \phi^{(1)}_j\big)
\end{cases}
\end{align*}

\noi
with
$\big\{ \big(\phi^{(0)}_j, \phi^{(1)}_j\big)\big\}_{j \in \N}$
 as in \eqref{gauss2}.

\end{remark}

\begin{remark}\rm 
(i) In this paper, we considered the defocusing case, namely, 
with the $-$ sign in front of the nonlinearity
$ \frac 1N \sum_{k = 1}^N u_{N, k}^2 u_{N, j}$
in \eqref{NLW1}.
We note that the local-in-time results
(Theorems \ref{THM:GWP1}\,(i), 
\ref{THM:GWP2}\,(i), and 
\ref{THM:conv1}\,(ii))
hold true even in the focusing case
(i.e.~with the $+$ sign in front of the nonlinearity
$ \frac 1N \sum_{k = 1}^N u_{N, k}^2 u_{N, j}$
in \eqref{NLW1}).

As for Theorem \ref{THM:conv2}\,(iii)
on the convergence of invariant Gibbs dynamics, 
it does not extend to the focusing case, 
since, even in the scalar case ($N = 1$), 
the focusing $\Phi^4_2$-measure
does not exist as a probability measure
(even with a (Wick renormalized) $L^2$-cutoff 
(as in \cite{LRS, BO96})
or a taming by a power of the Wick renormalized $L^2$-norm); 
see \cite{BS, OST2, GOTT}.
See also 
\cite{OST1, OOT1, OOT}
for an overview of the focusing Gibbs measure
(non-)construction and the references therein.
On the other hand, 
in the case of the quadratic nonlinearity, 
the Gibbs measure exists in the scalar case ($N = 1$; see~\cite{OST2}),
which seems to  extend to the $O(N)$ linear sigma model
(with a taming by a power of the Wick renormalized $L^2$-norm,
which is suitable for studying wave dynamics).
Thus, it would be of interest to investigate if an analogue
of Theorem \ref{THM:conv2}\,(iii) holds  in this case.

\medskip

\noi
(ii) In the survey paper \cite[Section 9]{Shen}, 
Shen 
proposed to study the case of infinite volume.
To our knowledge, the only available result is 
due to Delgadino and Smith \cite{DS}, 
who proved  convergence of each marginal of 
the $O(N)$ linear sigma model $\rho_N$ on $\R^2$ to 
a massive Gaussian free field $\mu_1$ on $\R^2$. 
However, there seems to be no dynamical result in infinite volume at this point.
In the wave case, thanks to the finite speed of propagation, 
we may combine the ideas
from \cite{Tolo1} and the current paper
to extend 
Theorem~\ref{THM:conv1}
for general initial data
to \HLS \eqref{NLW4} posed on the plane $\R^2$.
We will address this issue
in a forthcoming work.

\end{remark}

\begin{remark}\rm
Consider 
 the following \HLS on $\T^2$ with a small parameter $\eps > 0$:
\begin{align}
    (\eps^2 \dt^2 + \dt + m - \Dl) u_{N, j}^\eps = - \frac 1N \sum_{k = 1}^N 
    (u_{N, k}^\eps)^2 u_{N, j}^\eps + \sqrt{2} \xi_j,
\qquad j = 1, \dots, N.
\label{eNLW1}
\end{align}

\noi
For each fixed $\eps > 0$, the system \eqref{eNLW1} converges
to the following mean-field SdNLW:
\begin{align}
    (\eps^2\dt^2 + \dt + m - \Dl) u_j^\eps = - \E [ (u_j^\eps)^2 ] \,u_j^\eps + \sqrt{2} \xi_j
\label{eMF1}
\end{align}

\noi
as $N \to \infty$, 
(in the appropriate sense as in 
Theorems \ref{THM:conv1} and \ref{THM:conv2}).
Note that, when $\eps = 0$, 
the system \eqref{eNLW1}
reduces to \PLS \eqref{NLH1}, 
which was shown in \cite{SSZZ}
to converge to the 
mean-field SNLH \eqref{NLH2}
as $N \to \infty$. 
For later convenience, we
address the system \eqref{eNLW1}  
by the abbreviation \HLSSe, 
while 
we refer to \eqref{eMF1}
as the 
mean-field $\text{SdNLW}_\eps$.
In this formal discussion, we ignore the issue of renormalization.

After the appearance of this paper, 
the first two authors with Czernik~\cite{CLL} 
studied convergence issues of \eqref{eNLW1}
and \eqref{eMF1}
as $\eps \to 0$.  
Such 
 convergence is also known as the Smoluchowski-Kramers approximation. 
In the scalar case ($N = 1$), 
the authors in \cite{FHI, Zine} established convergence of  
the scalar SdNLW on $\T^2$ with a small parameter $\eps > 0$:
\begin{align*}
    (\eps^2 \dt^2 + \dt + m - \Dl) u^\eps = 
    - (u^\eps)^3  + \sqrt{2} \xi
\end{align*}

\noi
to SNLH on $\T^2$:
\begin{align*}
    (\dt + m - \Dl) u =-  u^3 + \sqrt{2} \xi
\end{align*}

\noi
as $\eps \to 0$ (under suitable assumptions).
Here, $\xi$ denotes a 
space-time white noise on 
$\R_+ \times \T^2$.
In~\cite{CLL}, this result was extended
to \HLSe \eqref{eNLW1}, converging to \PLS \eqref{NLH1}.
Furthermore, the authors
in \cite{CLL} proved 
convergence of 
the mean-field $\text{SdNLW}_\eps$ \eqref{eMF1}
to 
the mean-field SNLH~\eqref{NLH2}, 
and established the following commuting diagram
of various convergences as $\eps \to 0$ and $N \to \infty$.
See \cite{CLL} for a further discussion.

\begin{figure}[htbp] 
\centering 
    \begin{tikzcd}[cells={nodes={draw}}, row sep=2cm, column sep=4.5cm]
    \textup{\HLSe \eqref{eNLW1}}
    \arrow[r, smalltext=\normalsize 
$\substack{\displaystyle N \to \infty \vspace{3pt} \\ \text{(this paper)}}$
    ] 
    \arrow[d, smalltext=\normalsize 
$\eps \to 0  \text{,~\cite{FHI, Zine, CLL}}$]
    & \textup{Mean-field $\text{SdNLW}_\eps$ \eqref{eMF1}} \arrow[d, smalltext=\normalsize 
$\eps \to 0  \text{,~\cite{CLL}}$
]\\
    \textup{\PLS \eqref{NLH1}}\arrow[r, smalltext=\normalsize 
$N\to \infty  \text{,~\cite{SSZZ}}$
] & \textup{Mean-field SNLH \eqref{NLH2}}
\end{tikzcd}
\caption{Commuting diagram between \HLSe and the mean-field SNLH.}
\label{fig:diagram} 
\vspace{-0.8em}
\end{figure}

\end{remark}

\section{Notations and preliminary lemmas}

In Subsections \ref{SUBSEC:2.1}
and  \ref{SUBSEC:2.2}, 
we introduce basic notations
and recall the product estimates.
In Subsection \ref{SUBSEC:2.3}, 
we then 
recall the definition of the $I$-operator
along with its basic properties
and commutator estimates.
In Subsection \ref{SUBSEC:2.4}, 
we go over the 
definitions of the stochastic convolutions
$\Psi_j$
and $\Phi_j$
in \eqref{psi1}
and \eqref{phi1}, respectively, 
and their Wick powers
along with their regularity properties.

\subsection{Notations}
\label{SUBSEC:2.1}

Throughout this paper, the value of mass $m > 0$
is fixed and thus we often suppress $m$-dependence
in the remaining part of the paper.
In describing regularities of functions and distributions, 
we use $\eps > 0$ to denote an arbitrarily small constant.
Constants in various estimates depend on this $\eps$ but,
for simplicity of notation,  we suppress such $\eps$-dependence.
When it is clear from the context, we often suppress
dependence on other parameters such as $s$ and $t$.

For $A, B > 0$, we use $A\lesssim B$ to mean that
there exists $C>0$ such that $A \leq CB$.
By $A\sim B$, we mean that $A\lesssim B$ and $B \lesssim A$.
We may write  $\les_{\al}$ and  $\sim_{\al}$  to 
emphasize the dependence  on an external parameter $\al$.
We use $A \ll B$ to mean
$A\le c_0 B$
for some small $c_0 > 0$.
We use $C>0$ to denote various constants, which may vary line by line,
and we may write $C_\al$ or $C(\al)$ to  signify  dependence on an external parameter
$\al$.

In dealing with a solution $u$ to a wave equation, 
we use 
$\vec u$ to denote a pair $ ( u, \dt u)$.
With a slight abuse of notation, 
given a vector $(f_1, \dots, f_N)$, 
we often write it as $\{f_j \}_{j = 1}^N$
or $\vec f$.
We apply a similar convention for a vector 
$(f_1, f_2, \dots)$ of infinite length and write it as 
$\{f_j \}_{j = 1}^\infty$
or $\vec f$.
Namely,  we use the vectorial notation 
for two different meanings: $\vec u = (u, \dt u)$
and $
 \vec f = (f_1, \dots, f_N)
 = \{f_j \}_{j = 1}^N $,
but its meaning is clear from the context.

Given 
a Banach space $B$ and 
$N \in \N$,
let the $\A_N B$-norm 
of $ a = \{ a_j \}_{j = 1}^N \in B^{\otimes N}$
be the $\l^2$-average of the $B$-norms of $a_j$, defined in 
 \eqref{AN0}.
When there is no confusion, 
we also write 
$\|a_j \|_{\A_{N}B}$ for $\|a \|_{\A_{N}B}  = \|a_j \|_{\A_{N, j}B}$. 
By Jensen's inequality (or Cauchy-Schwarz's inequality), 
we have 
\begin{align}
 \frac{1}{N} \sum_{j = 1}^N \| a_j \|_B
\le     \|  a \|_{\A_N B} = \|a_j \|_{\A_{N}B}. 
\label{AN1a}
\end{align}

\noi
Similarly, 
we define the $\A_N^{(2)} B$-norm by 
setting
\begin{align}
\|  a \|_{\A_N^{(2)} B}^2 
= \|  a \|_{\A_{N, j} \A_{N, k}  B}^2 
=  \frac{1}{N^2} \sum_{j, k = 1}^N \| a_{j, k} \|_B^2
\label{AN2} 
\end{align}

\noi
for $ a = \{ a_{j, k} \}_{j, k = 1}^N \in B^{\otimes N^2}$. 
When $B = \R$, we simply set
$\A_N = \A_N \R$ and $\A_N^{(2)} = \A_N^{(2)}\R$.

By convention, we endow
$\T^2$ with the normalized Lebesgue measure $ dx_{_{\T^2}} =  (2\pi)^{-2}dx$
such that we do not need to carry factors involving $2\pi$.
For simplicity of notation, 
 we use $dx$ to denote
 the normalized Lebesgue measure  $ dx_{_{\T^2}}$ on $\T^2$
in the following.
 We write $\ft f$ or $\F(f) = \F_x (f)$ to denote the Fourier transform of $f$ on $\T^2$.
Given $M \in \N$, we denote by $\P_M$ the 
(spatial) frequency projector  onto frequencies $\{ n \in \Z^2 : |n| \leq M \}$ and 
we set
$\P_{M}^\perp = \Id - \P_M$.

We use 
 $\D (t)$ 
 to denote 
 the linear damped wave propagator
defined by 
\begin{align}
\D (t) = e^{- \frac{t}{2}} \frac{\sin ( t \jbb{\nb} )}{\jbb{\nb}}, 
\label{D1}
\end{align} 

\noi
where
\begin{align*}
  \jbb{\nb}  
=  \jbb{\nb}_m
  = \sqrt{m - \frac 14 - \Dl}.
\end{align*}

\noi
For $n \in \Z^2$, we also set
\begin{align*}
\jbb{n} 
= \jbb{n}_m= 
 \sqrt{m - \frac 14 + |n|^2}.
\end{align*}

\noi
(When  $0 < m \le \frac 14$, 
the operator $\D(t)$ in \eqref{D1} behaves slightly differently at the frequency $n = 0$.
However, it does not affect our analysis.)
Given $t_0 \in \R_+$, 
consider the following linear damped wave equation
(with a nice function $F$ on $\R_+\times \T^2$):
\begin{align}
\begin{cases}
(\dt^2 + \dt + m - \Dl) u = F\\
(u, \dt u)|_{t = t_0} = (f, g).
\end{cases}
\label{NLWx1}
\end{align}

\noi
Then, a solution $u$ to \eqref{NLWx1}
is 
given by 
\begin{align*}
u(t) = S(t - t_0) (f, g) + \I_{t_0}(F)(t), 
\end{align*}

\noi
where  $S(t)$ is the linear propagator given by 
\begin{align}
    S (t) (f, g) =  \dt \D (t) f + \D (t) (f + g)
\label{D4}
\end{align}

\noi
and $\I_{t_0}$ denotes the damped wave Duhamel integral operator given by 
\begin{align}
\I_{t_0} (f) (t) =  \int_{t_0}^t \D (t - t') f (t') d t' 
= \int_{t_0}^t e^{- \frac{t - t'}{2}} \frac{\sin ( (t - t') \jbb{\nb} )}{\jbb{\nb}} f (t') d t'.
\label{D5}
\end{align}

\noi
When $t_0 = 0$, 
we set $\I = \I_0$
for simplicity.

Given $s \in \R$ and $1 \leq r \leq \infty$, we define the $L^r$-based Sobolev space $W^{s, r} (\T^2)$ 
as the closure of $C^\infty(\T^2)$ under 
 the norm: 
\begin{align*}
    \| f \|_{W^{s, r}} =  \| \jb{\nb}^s f \|_{L^r} 
= \big\| \F^{-1}\big(\jb{\,\cdot\,}^s \ft f\big) \big\|_{L^r}
\end{align*}

\noi
such that $W^{s, r} (\T^2)$
is separable even when $r = \infty$.
Here,  
 $\jb{\,\cdot\,} = (1 + |\cdot|^2)^{\frac 12}$. 
When $r = 2$, we set  $H^s (\T^2) = W^{s, 2} (\T^2)$.
We also define $\H^{s} (\T^2) = H^s (\T^2) \times H^{s - 1} (\T^2)$ via the norm:
\begin{align*}
    \| (f, g) \|_{\H^s} = \| f \|_{H^s} + \| g \|_{H^{s - 1}}.
\end{align*}

 We will often use the short-hand notations 
 such as $C_T H^{s}_x$ and $L^p_\omega H^{s}_x$ for $C( [0,T]; H^{s}(\T^2))$ and $L^p(\Omega; H^{s}(\T^2))$, respectively.
Given an interval $I \subset \R$, 
we also write 
$C_I H_x^s $ for $C (I; H^s (\T^2))$, etc.
Lastly, given $1 \leq p \leq \infty$, we denote by $p'$ its H\"older conjugate (i.e.~$\frac{1}{p} + \frac{1}{p'} = 1$).

\subsection{Product estimates}
\label{SUBSEC:2.2}

We recall the following product estimates for Sobolev spaces;
see~\cite{BOZ} for a proof. See also \cite{GKO}.

\begin{lemma}
\label{LEM:prod}
Let $s > 0$.

\smallskip \noi
\textup{(i) (fractional Leibniz rule).} Let $1 < p_1, p_2, q_1, q_2, r \leq \infty$ be such that $\frac{1}{r} = \frac{1}{p_1} + \frac{1}{q_1} = \frac{1}{p_2} + \frac{1}{q_2}$. Then, we have
\begin{align*}
    \| f g \|_{W^{s, r}} \les \| f \|_{W^{s, p_1}} \| g \|_{L^{q_1}} + \| f \|_{L^{p_2}} \| g \|_{W^{s, q_2}}.
\end{align*}

\smallskip \noi
\textup{(ii)} Let $1 < p \leq \infty$ and $1 < q, r < \infty$ be such that 
$ \frac{1}{p} + \frac{1}{q}\le \frac s2 +  \frac{1}{r}$ and $q, r' \geq p'$. Then, we have
\begin{align*}
    \| f g \|_{W^{-s, r}} \les \| f \|_{W^{-s, p}} \| g \|_{W^{s, q}}.
\end{align*}
\end{lemma}

\subsection{$I$-operator and basic commutator estimates} 
\label{SUBSEC:2.3}

Fix $0 < s < 1$.
Given $M \in \N$, we define a 
smooth, radially symmetric, non-increasing (in $|\xi|$)
multiplier $\mf = \mf_{s, M} \in C^\infty(\R^2; [0, 1])$, satisfying 
\begin{equation*}
\mf_{s, M}(\xi)=
\begin{cases}
1, & \text{if } |\xi|\le M, \\
\big(\frac{M}{|\xi|}\big)^{1-s}, & \text{if } |\xi|\ge2M.
\end{cases}
\end{equation*} 

\noi
We then define the $I$-operator $I = I_{s, M}$
to be the Fourier multiplier operator with the multiplier $\mf_{s, M}$:
\begin{align}
\ft{I_{s, M}f}(n)=\mf_{s, M}(n)\ft{f}(n), \quad n \in \Z^2.
\label{I2}
\end{align}

\noi
As it is customary in the field, 
we  suppress the $s$-dependence
and 
 usually use $I$ (or $I_M$ to denote 
the $M$-dependence in an explicit manner)  in the following. 
Note that 
 $I f \in H^1(\T^2)$ if and only if $f \in H^s(\T^2)$
with the bound:
\begin{align}
\|f\|_{H^s}\les \|I f\|_{ H^1} \les M^{1-s} \|f\|_{H^s}.
\label{I3}
\end{align}

\noi
Moreover, 
 from the Littlewood-Paley theorem, we have
\begin{align}
\| I f \|_{W^{s_0 + s_1, p}} & \les M^{s_1} \| f \|_{W^{s_0, p}}
\label{I4}
\end{align}

\noi
for any $s_0 \in \R$, $0 \leq s_1 \leq 1-s$, and $1 < p < \infty$.

We now recall the commutator estimates.
See  \cite[Lemmas~3.1 and~3.3]{GKOT}
for the proofs.

\begin{lemma} 
\label{LEM:com1}
Let $\frac 23 \le s < 1$. Then, we have 
\begin{align*}
\|   I ( f^2 g ) - ( I f )^2 I g\|_{L^2} \les M^{2-  3s} \| I f \|_{H^1}^2 \| I g \|_{H^1}, 
\end{align*}

\noi
where the implicit constant is independent of $M \in \N$.

\end{lemma}

\begin{lemma}
\label{LEM:com2} 
Let $\frac 23 \le s < 1$.
 Then, given small $\dl > 0$,  there exist small $\sigma_0 = \sigma_0(\dl) >0$ and $p_0 = p_0 (\dl) \gg 1$ such that 
\begin{align*} 
\| I ( f h) - (I f)( I h) \|_{L^2} 
& \les M^{-\frac {s}{2}  +\dl} \| I f \|_{H^1} \| h \|_{W^{-\sigma_0, p_0}}, \\
\| I ( f g h) - (I f)( I g )(I h) \|_{L^2}
&  \les M^{- s + \frac 12  + \dl} \| I f \|_{H^1} \| I g \|_{H^1} \| h \|_{W^{-\sigma_0, p_0}}
\end{align*}

\noi
for any sufficiently large $M \gg 1$, 
where the implicit constant is independent of 
$M \gg 1$.

\end{lemma}

\subsection{Stochastic convolutions
and their Wick powers}
\label{SUBSEC:2.4}

In this subsection, 
by first recalling some basic tools from stochastic analysis
(see, for example, \cite{Simon, Kuo, Nua}), 
we 
go over the definitions 
and basic regularity properties of 
the stochastic convolutions $\Psi_j$  and $\Phi_j$ 
defined in \eqref{psi1} and~\eqref{phi1}, respectively, 
and their Wick powers, 
formally defined in \eqref{Wick0}.

First, 
recall the Hermite polynomials $H_k(x; \s)$ 
of degree $k$
with variance parameter $\s > 0$, 
defined through the generating function:
\begin{equation*}
 e^{tx - \frac{1}{2}\s t^2} = \sum_{k = 0}^\infty \frac{t^k}{k!} H_k(x;\s).
 \end{equation*}
	
\noi
For readers' convenience, we write out the first few Hermite polynomials:
\begin{align*}
\begin{split}
& H_0(x; \s) = 1, 
\quad 
H_1(x; \s) = x, 
\quad
H_2(x; \s) = x^2 - \s,   
\quad
 H_3(x; \s) = x^3 - 3\s x.
\end{split}
\end{align*}

Next, we recall the Wiener chaos estimate.
Let $(H, B, \mu)$ be an abstract Wiener space.
Namely, $\mu$ is a Gaussian measure on a separable Banach space $B$
with $H \subset B$ as its Cameron-Martin space.
Given  a complete orthonormal system $\{e_j \}_{ j \in \N} \subset B^*$ of $H^* = H$, 
we  define a polynomial chaos of order
$k \in \N\cup\{0\}$ to be an element of the form $\prod_{j = 1}^\infty H_{k_j}(\jb{x, e_j})$, 
where $x \in B$, $k_j \ne 0$ for only finitely many $j$'s, $k= \sum_{j = 1}^\infty k_j$, 
$H_{k_j}$ is the Hermite polynomial of degree $k_j$, 
and $\jb{\,\cdot\,, \,\cdot\,} = \vphantom{|}_B \jb{\,\cdot\,, \,\cdot\,}_{B^*}$ denotes the $B$-$B^*$ duality pairing.
We then 
denote the closure  of the span of
polynomial chaoses of order $k$ 
under $L^2(B, \mu)$ by $\mathcal{H}_k$.
The elements in $\H_k$ 
are called homogeneous Wiener chaoses of order $k$.
We also set
\[ \H_{\leq k} = \bigoplus_{j = 0}^k \H_j\]

\noi
 for $k \in \N$.

As a consequence
of the  hypercontractivity of the Ornstein-Uhlenbeck
semigroup  due to Nelson \cite{Nelson2}, 
we have the following Wiener chaos estimate
(\cite[Theorem~I.22]{Simon}).
See also \cite[Proposition~2.4]{TTz}.

\begin{lemma}\label{LEM:hyp}
Let $k \in \N$.
Then, we have
\begin{equation*}
\|X \|_{L^p(\O)} \leq (p-1)^\frac{k}{2} \|X\|_{L^2(\O)}
 \end{equation*}
 
 \noi
 for any $p \geq 2$
 and $X \in \H_{\leq k}$.

\end{lemma}

Next, we provide  a precise meaning of the stochastic convolutions
$\Psi_j$ and $\Phi_j$
and define their Wick powers.
Let $\{ \xi_j\}_{j \in \N}$
be a family of 
independent space-time white noises on 
$\R_+ \times \T^2$
as in~\eqref{NLW1}.
Given $j \in \N$
and $n \in \Z^2$, 
define $B_n^j$
by setting
$B_n^j (t) =  \jb{ \xi_j, \ind_{[0, t]} \cdot e^{i n \cdot x} }_{t, x}$, 
where  
 $\jb{\,\cdot, \cdot\,}_{t, x}$ denotes 
the duality pairing on $ \R_+\times \T^2$.
It is easy
to see that 
$\{ B_n^j \}_{j  \in \N, n \in \Z^2}$ is a family of  independent complex-valued
Brownian motions conditioned  that $B^j_{-n} = \cj{B^j_n}$
for each  $j \in \N$ and $n \in \Z^2$. 
Note that we have
 \[\text{Var}(B^j_n(t)) = \E\Big[
 \jb{\xi_j, \ind_{[0, t]} \cdot e^{in \cdot x}}_{t, x}\cj{\jb{\xi_j, \ind_{[0, t]} \cdot e^{in \cdot x}}_{ t, x}}
 \Big] = \|\ind_{[0, t]} \cdot e^{in \cdot x}\|_{L^2_{t, x}}^2 = t\]

\noi
 for any $n \in \Z^2$.
 We then define  
a family 
$\{  W^j\}_{j \in \N}$
of independent cylindrical Wiener processes on $L^2(\T^2)$
by setting
\begin{align*}
 W^j(t)
 = \sum_{n \in \Z^2} B_n^j (t) e^{in \cdot x}.
\end{align*}

With these notations, we can write 
 $\Psi_j$  in \eqref{psi1}
as 
\begin{align}
\begin{split}
\Psi_j (t) 
& =
\sqrt 2  \int_0^t \D (t - t') \xi_j (dt')\\
& = \sqrt 2  \int_0^t \D (t - t') d W^j (t'),
\end{split}
\label{psi2}
\end{align}

\noi 
where $\D (t)$ is 
 the linear damped wave propagator defined  in \eqref{D1}.
Given  $j, M \in \N$, set $\Psi_{j, M} = \P_M \Psi_j$. 
Then, given $(t, x) \in \R_+ \times \T^2$, 
we see that 
$\Psi_{j, M}(t, x)$ is a mean-zero Gaussian random variable with variance:
\begin{align}
\begin{split}
\sigma_M (t)
    & \deff \E \big[ \Psi_{j, M}^2 (t, x) \big] \\
& \hspace{0.5mm}
= 
\sum_{|n| \le M} 
\bigg(
\frac {1 - e^{- t}}{\jbb{n}^2}   
-
\frac {e^{- t}2 \sin(2 t \jbb{n})}{\jbb{n} (1 + 4\jbb{n}^2)}
- 
\frac {1 - e^{- t} \cos(2 t \jbb{n})}{\jbb{n}^2(1 + 4 \jbb{n}^2)} 
\bigg)    \\
&  \sim_{m, t} \log M 
\too \infty, 
\end{split} 
\label{sig1}
\end{align}

\noi
 as $M \to \infty$.
Note that, 
while  $\sigma_M$ depends on $t\in \R_+$, 
it is independent of $x \in \T^2$.
 
Given $k, j \in \N$, we define
the Wick powers by setting
\begin{align}
\begin{split}
\wick{ \Psi_{k, M} \Psi_{j, M} } \, 
& = 
\begin{cases}
        H_2 ( \Psi_{j, M}; \sigma_M ), &  \text{if } k = j,  \\
\Psi_{k, M} \Psi_{j, M}, & \text{if } k \neq j,
    \end{cases} \\
\wick{ \Psi_{k, M}^2 \Psi_{j, M} } \, 
& = 
\begin{cases}
H_3 ( \Psi_{j, M}; \sigma_M ),  & \text{if } k = j,  \\
H_2 ( \Psi_{k, M}; \sigma_M ) \Psi_{j, M},  & \text{if } k \neq j,
\end{cases}
\end{split}
\label{psi3}
\end{align}

\noi
where $H_\l (x ; \sigma)$ is the  Hermite polynomial of degree $\l$ with variance parameter $\sigma$. Then, the Wick powers,  formally given in \eqref{Wick0}, 
are defined by 
\begin{align}
\begin{split}
    \wick{\Psi_k \Psi_j} \, &=  \lim_{M \to \infty} \wick{\Psi_{k,M} \Psi_{j,M}} \, , \\
    \wick{\Psi_k^2 \Psi_j} \, &=  \lim_{M \to \infty} \wick{\Psi_{k,M}^2 \Psi_{j,M}}.
\end{split}
\label{psi4}
\end{align}

\noi
Proceeding as in \cite{GKO, GKO2}, 
it is easy to see that, given any $\eps > 0$,  the limits in \eqref{psi4}
exist
in $C(\R_+; W^{-\eps, \infty}(\T^2))$
almost surely.
See Lemma \ref{LEM:sto1}.

Next, we turn our attention to  $\Phi_j$  in \eqref{phi1}.
With the notations above, we have 
\begin{align}
\Phi_j (t) = 
S(t) \big(\phi^{(0)}_j, \phi^{(1)}_j\big) + \sqrt {2} \int_0^t \D (t - t') d W^j (t'), 
\label{phi2}
\end{align}

\noi
where 
$S(t)$ is as in \eqref{D4} and $\big\{\big(\phi^{(0)}_j, \phi^{(1)}_j\big)\big\}_{j \in \N}$
is a family of independent $\D'(\T^2)$-valued random variables, satisfying~\eqref{gauss2}, 
which is assumed to be independent of $\{\xi_j\}_{j \in \N}$.
Given $M \in \N$ and $j \in \N$, we set  $\Phi_{j, M} = \P_{M} \Phi_j$. 
Then, 
using \eqref{sig1}, 
it is easy to see that, 
given $(t, x) \in \R_+ \times \T^2$, 
$\Phi_{j, M}(t, x)$ is a mean-zero Gaussian random variable with variance:
\begin{align*}
    \al_M \deff \E \big[ \Phi^2_{j, M} (t, x)\big] = \sum_{\substack{n \in \Z^2 \\ |n| \leq M}} \frac{1}{m + |n|^2} \sim_m \log M
\too \infty,
\end{align*}

\noi
as $M \to \infty$.
Note that $\al_M$ is independent of 
$(t, x) \in \R_+ \times \T^2$.
Then, as in \eqref{psi4}, we define the Wick powers
by setting
\begin{align}
\begin{split}
    \wick{\Phi_k \Phi_j} \, &=  \lim_{M \to \infty} \wick{\Phi_{k,M} \Phi_{j,M}} , \\
    \wick{\Phi_k^2 \Phi_j} \, &=  \lim_{M \to \infty} \wick{\Phi_{k,M}^2 \Phi_{j,M}}, 
\end{split}
\label{def_wick2}
\end{align}

\noi
where the limits exist almost surely in 
 $C(\R_+; W^{-\eps, \infty}(\T^2))$
for any $\eps > 0$.
Here, 
$\wick{\Phi_{k,M} \Phi_{j,M}}$ and $\wick{(\Phi_{k,M})^2 \Phi_{j,M}}$ 
are  as in~\eqref{psi3}, 
but with the variance parameter $\al_M$ instead of $\s_M$.

Before we state lemmas on regularity properties
of the stochastic convolutions and their Wick powers, we need to introduce 
some notations.
Fix $N \in \N$.
Given
$\eps> 0 $
and an interval $J \subset \R_+$, 
we
define the space $\ZZ_N^\eps(J)$
 of enhanced data sets
$\pmb Z_N$ of the form
\begin{align}
\pmb Z_N = \big\{Z_j^{(1)},  Z_{k, j}^{(2)}, Z_{k, j}^{(3)}\big\}_{j, k = 1}^N
\label{X0}
\end{align}

\noi
by setting
\begin{align}
\begin{split}
\| \pmb Z_N \|_{\ZZ_N^\eps(J)}
&
=  \|Z_j^{(1)}\|_{\A_N C_J W_x^{- \eps, \infty}}
+  \|Z_{k, k}^{(2)}\|_{\A_N C_J W_x^{- \eps, \infty}}\\
&  \quad
+  \|Z_{k, j}^{(2)}\|_{\A^{(2)}_N C_J W_x^{- \eps, \infty}}
 +   \|Z_{k, j}^{(3)}\|_{\A^{(2)}_N C_J W_x^{- \eps, \infty}}, 
\end{split}
\label{X1}
\end{align}

\noi
where 
$\A_N B$ 
and 
$\A^{(2)}_N B$ 
are 
the $\l^2$-averages defined
 in~\eqref{AN0} and \eqref{AN2}.
When $J = [0, T]$, 
with a slight abuse of notation, we simply set
\begin{align}
\ZZ_N^\eps(T) = \ZZ_N^\eps([0, T]).
\label{X1a}
\end{align}

\noi
For our application, we take
\begin{align}
Z_j^{(1)} = \Psi_j,\qquad   Z_{k, j}^{(2)} = \, \wick{\Psi_k \Psi_j} \, , 
\qquad \text{and}\qquad 
Z_{k, j}^{(3)}
=  \, \wick{\Psi_k^2 \Psi_j} 
\label{X2}
\end{align}

\noi
in studying 
\eqref{NLW3}, 
while we set 
\begin{align}
Z_j^{(1)} = \Phi_j,\qquad   Z_{k, j}^{(2)} = \, \wick{\Phi_k \Phi_j} \,, 
\qquad \text{and}\qquad 
Z_{k, j}^{(3)}
=  \, \wick{\Phi_k^2 \Phi_j}
\label{X3}
\end{align}

\noi
in studying 
\eqref{NLW6}.

The first lemma states the basic regularity properties
of these stochastic terms.

\begin{lemma}
\label{LEM:sto1}

\textup{(i)}
Given 
$j, k, M \in \N$, set 
\begin{align*}
Z_{M}^{(1)} = \Psi_{j, M},\qquad   Z_{M}^{(2)} = \, \wick{\Psi_{k, M} \Psi_{j, M}} \, , 
\qquad \text{and}\qquad 
Z_{M}^{(3)}
=  \, \wick{\Psi_{k, M}^2 \Psi_{j, M}}, 
\end{align*}

\noi
where 
$\Psi_{j, M} = \P_M \Psi_{j}$, 
and
$ \wick{\Psi_{k, M} \Psi_{j, M}} $
and $ \wick{\Psi_{k, M}^2 \Psi_{j, M}}$
are as in 
\eqref{psi3}.
Let $\l = 1, 2, 3$.
Then, 
given any  $T,\eps>0$ and finite $p \geq 1$, 
$  Z^{(\l)}_M$
converges to some limit $Z^{(\l)}$ in 
\[L^p(\O;C([0,T];W^{-\eps,\infty}(\T^2)))\]

\noi
and also almost surely in $C(\R_+;W^{-\eps,\infty}(\T^2))$, 
endowed with the compact-open topology in time, 
as $M \to \infty$, 
where $Z^{(\l)}$ is given by 
\begin{align*}
Z^{(1)} = \Psi_{j},\qquad   Z^{(2)} = \, \wick{\Psi_{k} \Psi_{j}} \, , 
\qquad \text{and}\qquad 
Z^{(3)}
=  \, \wick{\Psi_{k}^2 \Psi_{j}}.
\end{align*}

\noi
Moreover, there exist
  $C, c > 0$ such that 
\begin{align}
\PP \Big( \| Z^{(\l)} \|_{C_{J} W_x^{- \eps, \infty}} > \ld \Big) 
\leq C \exp\big( - c\ld^{\frac{2}{\l}}\big)
\label{tail1}
\end{align}

\noi
for any $\ld > 0$ and any interval $J \subset \R_+$ with $|J| = 1$, 
uniformly in $J$ and $j, k \in \N$.

\smallskip

\noi
\textup{(ii)}
Given $N \in \N$, 
let $\pmb Z_N$ be as in \eqref{X0}, 
where 
$Z_j^{(1)}$,  $Z_{k, j}^{(2)}$, and $Z_{k, j}^{(3)}$
are as in \eqref{X2}.
Then,  given $\eps > 0$, there exist
  $C, c > 0$ such that 
\begin{align}
\PP \Big( \| \pmb Z_N \|_{\ZZ^\eps_N(J)} > \ld \Big) 
\le C 
  \exp\Big( - c\frac{\ld^\frac{2}{3}}{(1 + |J|)^\frac 23 }\Big)
\label{tail2}
\end{align}

\noi
for any 
interval $J \subset \R_+$
and 
$\ld > 0$,  
uniformly in $N \in \N$, 
where the $\ZZ_N^\eps(J)$-norm is as in \eqref{X1}.

\smallskip

\noi
\textup{(iii)}
The claims in Parts (i) and (ii) hold
even if we replace $\Psi_j$ 
and its  Wick powers
by $\Phi_j$
defined in \eqref{phi2} 
with the corresponding Wick powers.

\end{lemma}

\begin{proof}
(i) 
By noting that $Z^{(\l)}_M, \, Z^{(\l)} \in \H_\l$, 
the claims in Part (i) 
follow from 
a slight modification of 
the proof of 
\cite[Lemma~2.3]{GKOT}, 
based on 
the Wiener chaos estimate (Lemma \ref{LEM:hyp})
and 
the Garsia-Rodemich-Rumsey inequality
(\cite[Lemma~2.2]{GKOT}).
Since we work with the damped wave equation, 
there is no growth in time
in computing the $p$th moments of 
the stochastic terms,
which is the reason that the right-hand side of \eqref{tail1}
is independent of time;
compare this with \cite[(2.7)]{GKOT}.
See also  \cite[Proposition~2.1]{GKO}
and 
\cite[Lemma~3.1]{GKO2}, 
where the latter utilizes
  \cite[Proposition~2.7 and Lemma~A.1]{OOTz}.

\medskip

\noi
(ii)
For simplicity of the presentation, 
we only consider the case $J = [0, T]$.
We first consider the contribution from the last term
on the right-hand side of~\eqref{X1}. 
Write $[0, T] = \bigcup_{\l = 0}^{[T]} J_\l$, 
where $J_\l = [\l, \l+1]$ for $\l = 0, \dots, [T] - 1$
and $J_{[T]} = \big[[T], T\big]$.
Here, $[T]$ denotes the integer part of $T$.

From 
 \eqref{tail1}, 
we have 
\begin{align*}
\sup_{\l = 0, \dots, [T]} 
\sup_{j, k \in \N}
\|Z_{k, j}^{(3)}  \|_{L^p_\o C_{J_\l} W_x^{- \eps, \infty}}
\les p ^\frac 32.
\end{align*}

\noi
Then, by Minkowski's integral inequality, we obtain
\begin{align}
\begin{split}
\Big\|\|Z_{k, j}^{(3)}  \|_{\A^{(2)}_N C_T W_x^{- \eps, \infty}}\Big\|_{L^p(\O)}
& \le 
\sum_{\l = 0}^{[T]}
\Big\|\|Z_{k, j}^{(3)}  \|_{L^p_\o C_{J_\l} W_x^{- \eps, \infty}}\Big\|_{\A_N^{(2)}}\\
& \les  p ^\frac 32 ( 1+ T)
\end{split}
\label{tail5}
\end{align}

\noi
for any finite $p \ge 2$, 
where the implicit constant is independent of $N \in \N$.
By a crude estimate, 
we see that 
the other terms on the right-hand side of \eqref{X1}
also satisfy the same bound.
Hence, 
the  tail estimate~\eqref{tail2} 
follows
from~\eqref{tail5} and Chebyshev's inequality
(see the proof of \cite[Lemma 3]{BOP1}\footnote{Lemma 2.1 in the arXiv version.}); see also \cite[Lemma~4.5]{Tz10}.

\medskip

\noi
(iii) A straightforward modification yields the corresponding results for $\Phi_j$, 
and thus we omit details.
\end{proof}

\begin{remark} \rm
\label{REM:sto1a}

(i) As a consequence of  \eqref{tail1}, 
we have the following moment bound:
\begin{align*}
 \Big\| \| Z ^{(\l)}\|_{C_J W_x^{- \eps, \infty}} \Big\|_{L^p(\O)} \les p^\frac{\l}{2}
\end{align*}

\noi
for any finite $p \ge 1$
and any interval $J \subset \R_+$ with $|J| = 1$, 
uniformly in $J$ and $j, k \in \N$.

\smallskip

\noi
(ii)
In Section~\ref{SEC:conv}, 
we  need to deal with the product $\Psi_k \P_{M}^\perp \Psi_j$ with $k \neq j$, where 
$\P^\perp_{M} = \Id - \P_{M}$.
It follows from a slight modification of the proof of Lemma \ref{LEM:sto1}
that 
there exists small $\ta > 0$ such that 
\begin{align}
 \E \big[ \| \Psi_k \P_M^\perp \Psi_j \|_{C_T W_x^{- \eps, \infty}}^p \big] \les_{T} 
 p 
 M^{- \ta  p}
\label{mom2}
\end{align}

\noi
\noi
for any finite $p \ge 1$, 
uniformly in $j, k, M \in \N$.

\end{remark}

Lastly, we state a variant of \cite[Lemma~2.4]{GKOT}
and its consequence (see \cite[Subsection 3.2]{GKOT})
on the 
 logarithmically divergent nature of the stochastic convolution~$\Psi_j$
 in \eqref{psi2}.

\begin{lemma} 
\label{LEM:sto2}

Given $0 < s < 1$ and $M \in \N$, let $I = I_{s, M}$
be the $I$-operator defined in~\eqref{I2}.
Given $j \in \N$ and  
$(t, x) \in \R_+\times \T^2$, 
 $I \Psi_j (t, x)$ is a mean-zero Gaussian random variable, 
 satisfying
\begin{align}
 \E\big[  (I\Psi_j (t, x))^2 \big]
\les _s  \log \jb{M}, 
\label{mom3a}
\end{align}

\noi
uniformly in 
$j \in \N$ and 
$(t, x) \in \R_+\times \T^2$.
Furthermore, given $N \in \N$, there exists
a random variable $R_N$
and  $\ta = \ta(s) \gg 1$, independent of $N \in \N$, 
such that 
\begin{align}
\sup_{N \in \N} \E[ R_N ]  < \infty
\label{mom4}
\end{align}

\noi
and 
\begin{align}
\| I \Psi_j  \|_{L^r_{[t_1, t_2], x}\A_N} 
 \le
r e^{\frac 1r \ta T \log \jb{M}} R_N^\frac 1r
\label{mom4a} 
\end{align}

\noi
for any 
$1 \le r < \infty$, 
$0 \le t_1  \le t_2 \le T$, 
and $M \in \N$, 
uniformly in $N \in \N$, 
where
$L^r_{[t_1, t_2], x}\A_N
= L^r_{[t_1, t_2]}L^r_x \A_N$.

\end{lemma}

\begin{proof}
The bound \eqref{mom3a}
follows
from a straightforward modification of the proof of 
\cite[Lemma~2.4]{GKOT} and thus we omit details.

Given $N \in \N$, 
define $R_N = R_N(\ta)$ by setting 
\begin{align}
R_N = \sum_{M = 1}^\infty \sum_{\l = 1}^\infty
e^{-\ta \l \log \jb{M}}
\int_0^\l \int_{\T^2}e^{\|I \Psi_j (t, x) \|_{\A_N}} dx dt, 
\label{mom3}
\end{align}

\noi
where $\ta = \ta(s) \gg 1$ (to be chosen later)  and $\A_N = \A_N \R$ is as in \eqref{AN0}.
Then, from \eqref{mom3}, we have 
\begin{align}
\begin{split}
\| I \Psi_j  \|_{L^r_{[t_1, t_2], x}\A_N}^r 
& = \int_{t_1}^{t_2}\int_{\T^2}
\| I \Psi_j(t, x)\|_{\A_N}^r 
 dx dt 
 \leq r!  \int_0^T 
\int_{\T^2}
e^{\| I \Psi_j( t, x)\|_{\A_N}}
 dx dt \\
& \le
r! e^{\ta T \log \jb{M}} R_N 
\end{split}
\label{mom4b}
\end{align}

\noi
for any 
$1 \le r < \infty$, 
$0 \le t_1  \le t_2 \le T$, 
and $M \in \N$, 
uniformly in $N \in \N$.
By noting $r! \le r^r$, 
the bound \eqref{mom4a} follows from \eqref{mom4b}.

It remains to prove \eqref{mom4}.
Note that the function $e^{(1+|x|)^\frac 12}$ is convex.
From \eqref{AN0}
and Jensen's inequality, we have 
\begin{align*}
e^{\| I \Psi_j (t, x) \|_{\A_N}}
& \le  \exp\bigg( \Big\{1+  \frac 1N\sum_{j = 1}^N \big( I \Psi_j (t, x)\big)^2 \Big\}^\frac 12 \bigg)\\
& \le 
\frac 1N \sum_{j = 1}^N e^{1 + | I \Psi_j (t, x)|}.
\end{align*}

\noi
Then, from \eqref{mom3},  \eqref{mom3a}, 
and the fact that 
 $I \Psi_j (t, x)$ is a mean-zero Gaussian random variable, 
satisfying \eqref{mom3a}, 
we obtain
\begin{align*}
\E[R_N] 
& 
\les 
\frac 1N \sum_{j = 1}^N
 \sum_{M = 1}^\infty \sum_{\l = 1}^\infty
e^{-\ta \l \log \jb{M}}
\int_0^\l \int_{\T^2}
\E\Big[e^{|I \Psi_j (t, x)|}\Big] dx dt\\
& \les 
\frac 1N \sum_{j = 1}^N
 \sum_{M = 1}^\infty \sum_{\l = 1}^\infty
e^{-\ta \l \log \jb{M}}\cdot 
\l e^{C_s \log \jb{M}} \\
& \les 1,
\end{align*}

\noi
uniformly in $N \in \N$, 
where the last step follows
from taking  
 sufficiently large
$\ta = \ta(s)\gg1 $.
This proves \eqref{mom4}.
\end{proof}

\section{Well-posedness of the hyperbolic $O (N)$ linear sigma model}

\label{SEC:3}
In this section, we prove pathwise global 
well-posedness of 
\HLS\eqref{NLW4}
(Theorem \ref{THM:GWP1}). 
In fact, by the first order expansion \eqref{exp}, we show global well-posedness of the SdNLW 
system~\eqref{NLW3} for 
$\vv_N = \{v_{N, j}\}_{j = 1}^N$, which we write out below for readers' convenience:
\begin{align}
\begin{split}
(\dt^2 + \dt + m - \Dl) v_{N, j} 
& = - \frac{1}{N} \sum_{k = 1}^N  \Big( 
v_{N, k}^2 v_{N, j}
+ 2 \Psi_k v_{N, k} v_{N, j}
  + v_{N, k}^2 \Psi_j \\
& 
\hphantom{XXXXX}
 + \wick{\Psi_k^2} v_{N, j} + 2 v_{N, k} \wick{ \Psi_k \Psi_j } 
+ \wick{\Psi_k^2 \Psi_j} \Big)
\end{split}
\label{VN1}
\end{align}

\noi
for $j = 1, \dots, N$.

In Subsection~\ref{SUBSEC:3.1}, we prove local well-posedness of \eqref{VN1} via a standard 
contraction  argument. 
In Subsection~\ref{SUBSEC:3.2}, we then establish some preliminary uniform
(in $N$) estimates for the modified energy 
$\EN (t) $
defined in  \eqref{EN1a}.
Finally, in Subsection~\ref{SUBSEC:3.3}, 
we prove global well-posedness of~\eqref{VN1}, 
with uniform (in $N$) bounds, 
by adapting
the hybrid $I$-method introduced in \cite{GKOT}
to the current vector-valued setting.
We note that a slight modification of the argument presented in this section also yields
pathwise global well-posedness for~\eqref{SdNLW1}
(Theorem~\ref{THM:conv2}\,(i)).
We omit details.

\subsection{Local well-posedness}
\label{SUBSEC:3.1}

In this subsection, we prove local well-posedness
of \eqref{VN1}.
In order to separate the deterministic and probabilistic components
of the argument, we
consider the following system
of SdNLW
with a given {\it deterministic} 
enhanced data set
$\pmb Z_N = \big\{Z_j^{(1)},  Z_{k, j}^{(2)}, Z_{k, j}^{(3)}\big\}_{j, k = 1}^N$
as in 
\eqref{X0}:
\begin{align}
\begin{split}
(\dt^2 + \dt + m - \Dl) v_{N, j} 
& = - \frac{1}{N} \sum_{k = 1}^N  \Big( 
v_{N, k}^2 v_{N, j}
+ 2 Z_{k}^{(1)} v_{N, k} v_{N, j} + v_{N, k}^2 Z_{j}^{(1)}\\
&\hphantom{XXXXX}
 + 
 Z_{k, k}^{(2)} v_{N, j} + 2 v_{N, k} Z_{k, j}^{(2)}
+  Z_{k, j}^{(3)} \Big)
\end{split}
\label{VN2}
\end{align}

\noi
for $j = 1, \dots, N$.
Under \eqref{X2}, 
the system \eqref{VN2} reduces to \eqref{VN1}, 
while it reduces to~\eqref{NLW6}
under \eqref{X3}.

Given an interval $J \subset \R_+$, 
define the $X^s_N(J)$-norm by setting
\begin{align}
\begin{split}
\|  \vec \vv_N\|_{X^s_N(J)}
& 
= 
\| (\vv_N, \dt \vv_N)\|_{\A_N C_J  \H^s_x}\\
& = 
\| ( v_{N, j}, \dt   v_{N, j})\|_{\A_{N, j} C_J  \H^s_x}.
\end{split}
\label{XN1}
\end{align}

\noi
When $J = [0, \tau]$, 
with a slight abuse of notation, 
we set 
$X^s_N(\tau) = X^s_N([0, \tau])$.

\begin{proposition}
\label{PROP:LWP1}
Let $m > 0$ and $\frac 12 \le s < 1$
and $ T\ge 1$.
Given $N \in \N$, 
suppose that 
\begin{align}
\pmb Z_N = \big\{Z_j^{(1)},  Z_{k, j}^{(2)}, Z_{k, j}^{(3)}\big\}_{j, k = 1}^N
\in \ZZ^\eps_N(T),
\label{enh1x}
\end{align}
and 
$\big(\vv_N^{(0)}, \vv_N^{(1)}\big) = \big\{\big(v^{(0)}_{N, j}, v^{(1)}_{N, j}\big) \big\}_{j = 1}^N \in \A_N \H^s(\T^2)$, 
where 
$\ZZ^\eps_N(T)$ 
and
$\A_N B$ 
are  as in \eqref{X1}
and  \eqref{AN0}, respectively.
Then, given any $t_0 \in [0, T)$, 
there exists a 
unique solution $\vec \vv_N = (\vv_N, \dt \vv_N)
= \{ (v_{N, j}, \dt v_{N, j})\}_{j = 1}^N$
to \eqref{VN2}
with
$\vec \vv_N|_{t = t_0} 
=  \big(\vv_N^{(0)}, \vv_N^{(1)}\big)$
on the time interval 
$[t_0, t_0 + \tau]\cap [0, T]$, 
where the local existence time
\[\tau = \tau\Big(\big\| \big(v^{(0)}_{N, j}, v^{(1)}_{N, j}\big) \big\|_{\A_{N} \H^s_x}, 
\| \pmb Z_N \|_{\ZZ_N^\eps(T)}\Big) >0\]
does not depend on $N \in \N$ in an explicit manner.
Moreover, the solution $\vec \vv_N$
is unique
in the entire class $X^s_N([t_0, (t_0 + \tau)\wedge T])$, where 
$a\wedge b = \min (a, b)$.

\end{proposition}

Theorem \ref{THM:GWP1}\,(i)
follows from Proposition \ref{PROP:LWP1}, 
\eqref{X2}, and Lemma \ref{LEM:sto1}.
We omit details.

\begin{remark}
\label{REM:LWP1} \rm

(i)
As pointed out in Section \ref{SEC:1}, 
by using the Strichartz estimates as in \cite{GKO}, 
it is possible to prove local well-posedness
of \eqref{VN2} for $s > \frac 14$.
However, the uniqueness guaranteed
by such an argument holds
only in the intersection of 
$X^s_N(\tau\wedge (T - t_0))$
with a suitable $\l^2$-averaged version
of Strichartz spaces.
A similar remark holds for 
the local well-posedness argument 
of the mean-field SdNLW \eqref{MF2} (see Proposition \ref{PROP:LWP2})
for which 
 unconditional uniqueness
of  solutions
plays a key role in Section \ref{SEC:6}.
Since we prefer to have a uniform treatment
for both \HLS~\eqref{NLW4}
and 
the mean-field SdNLW \eqref{MF2}, 
we restrict our attention to $\frac 12 \le s < 1$
in  
Proposition \ref{PROP:LWP1}
(and also in Proposition \ref{PROP:LWP2}), 
 guaranteeing
unconditional uniqueness of solutions.

\smallskip

\noi
(ii) Let $T \ge 1$.
Given $N \in \N$, 
let $\vv_N$ be the local-in-time solution 
to \eqref{VN2} constructed in Proposition \ref{PROP:LWP1}.
Then, 
by denoting by $T_*$ its maximal existence time, 
the following blowup alternative holds:
\begin{align}
T_* = T\qquad \text{or}\qquad 
\lim_{t \to T_*-}
\| (\vv_N, \dt \vv_N)(t)\|_{\A_{N} \H^s_x}
= \infty.
\label{blow1}
\end{align}


\end{remark}

We now present a proof of 
Proposition \ref{PROP:LWP1}.

\begin{proof}[Proof of Proposition \ref{PROP:LWP1}] 

The claimed local well-posedness of \eqref{VN2}
follows
from a straightforward adaptation of the proof of 
\cite[Proposition 4.1]{GKOT}.
For readers' convenience, however, we present some details, 
in particular to show 
that 
 the local existence time $\tau > 0$ does not depend
 on $N \in \N$ in an explicit manner.

Without loss of generality, we assume that $t_0 = 0$.
Fix $N \in \N$.
By writing  \eqref{VN1} in the Duhamel formulation, 
we have
\begin{align}
v_{N, j} (t)
=  S (t) \big(v^{(0)}_{N, j}, v^{(1)}_{N, j}\big) 
-  \sum_{\l = 1}^6 \I(A_{j, \l})(t), 
\quad j = 1, \dots, N, 
\label{LWP1}
\end{align}

\noi
where $S(t)$ and $\I = \I_0$ are as in \eqref{D4} and \eqref{D5}, respectively, 
and 
$\{ A_{j, \l} = A_{j, \l}(N)\}_{j = 1}^N$, $ \l = 1, \dots, 6$, are given by 
\begin{align*}
 A_{j, 1}  & =  \frac{1}{N} \sum_{k = 1}^N  
  v_{N, k}^2 v_{N, j}, 
  \qquad 
A_{j, 2}   =  \frac{1}{N} \sum_{k = 1}^N 
 2 Z_{k}^{(1)} v_{N, k} v_{N, j}, \\
A_{j, 3} & =  \frac{1}{N} \sum_{k = 1}^N 
 v_{N, k}^2 Z_{j}^{(1)} , \qquad 
 A_{j, 4}  =  \frac{1}{N} \sum_{k = 1}^N 
 Z_{k, k}^{(2)} v_{N, j},\\
A_{j, 5}  & =  \frac{1}{N} \sum_{k = 1}^N 
  2 v_{N, k} Z_{k, j}^{(2)},
  \qquad 
A_{j, 6}  =    \frac{1}{N} \sum_{k = 1}^N 
Z_{k, j}^{(3)}.
\end{align*}

Fix $0 < \tau \le 1$
and small $\eps > 0$
such that $  s \le 1-\eps$.
We point out that for all the estimates appearing 
in the following,  the implicit constants are independent of $N \in \N$, $\tau>0$, 
 the initial data 
 $\big(\vv_N^{(0)}, \vv_N^{(1)}\big) = \big\{\big(v^{(0)}_{N, j}, v^{(1)}_{N, j}\big) \big\}_{j = 1}^N$, 
 and the enhanced data set $\pmb Z_N$ in \eqref{enh1x}.

By 
Sobolev's and  H\"older's inequalities 
with~\eqref{XN1}, 
we have 
\begin{align}
\begin{split}
\|\vec \I(A_{j, 1})\|_{X^s_{N, j}(\tau)}
&\les \tau 
\frac 1N \sum_{k = 1}^N 
\| v_{N, k} ^2 v_{N, j} \|_{\A_{N, j} C_\tau  H_x^{s-1}} \\
&\les \tau 
\frac 1N \sum_{k = 1}^N 
\| v_{N, k}  \|_{ C_\tau  L^{\frac{6}{2 - s}}_x}^2
\|  v_{N, j} \|_{\A_{N, j} C_\tau   L^{\frac{6}{2 - s}}_x} \\
&\le \tau
\|  \vec \vv_N\|_{X^s_N(\tau)}^3 , 
\end{split}
\label{LWP2}
\end{align}

\noi
provided that $\frac 12 \le s$ ($< 1$).
Here, $\vec  \I(A_{j, 1})$ is given by 
 \begin{align}
  \vec  \I(A_{j, 1})
 = \big( \I(A_{j, 1}), \dt \I(A_{j, 1})\big)
\label{LWP2a}
 \end{align}

\noi
and we use analogous notations in the following.
By \eqref{XN1}, 
Sobolev's inequality, 
Lemma~\ref{LEM:prod}\,(ii), 
the fractional Leibniz rule 
(Lemma~\ref{LEM:prod}\,(i)), 
Sobolev's inequality, 
and~Cauchy-Schwarz's inequality (in $k$) 
with~\eqref{X1}, 
we have 
\begin{align}
\begin{split}
\|\vec \I(A_{j, 2})\|_{X^s_{N, j}(\tau)}
& 
\les \tau
\frac 1N \sum_{k = 1}^N 
\|  Z_{k}^{(1)} v_{N, k} v_{N, j} \|_{\A_{N, j} C_\tau W_x^{-\eps, \frac 2{2-s-\eps}}}\\
& 
\les \tau
\frac 1N \sum_{k = 1}^N 
\|  Z_{k}^{(1)}\|_{C_T W^{-\eps, \infty}_x}
\| v_{N, k} v_{N, j}\|_{\A_{N, j}C_\tau W_x^{\eps, \frac 2{2-s-\eps}}}\\
& 
\les \tau
\frac 1N \sum_{k = 1}^N 
\|  Z_{k}^{(1)}\|_{C_T W^{-\eps, \infty}_x}
\| v_{N, k} \|_{C_\tau H^{\frac s2 + \frac 32 \eps}_x}
\| v_{N, j}\|_{\A_{N, j}C_\tau H^{\frac s2  + \frac 32 \eps}_x}\\
& 
\le
 \tau
\| \pmb Z_N \|_{\ZZ_N^\eps(T)}
\|  \vec \vv_N\|_{X^s_N(\tau)}^2, 
\end{split}
\label{LWP3}
\end{align}

\noi
provided that $3 \eps \le s \le 1-\eps$.
A similar computation yields
\begin{align*}
\begin{split}
\|\vec \I(A_{j, 3})\|_{X^s_{N, j}(\tau)}
\les
 \tau
\| \pmb Z_N \|_{\ZZ_N^\eps(T)}
\|  \vec \vv_N\|_{X^s_N(\tau)}^2.
\end{split}
\end{align*}

\noi
By Sobolev's inequality, 
Lemma~\ref{LEM:prod}\,(ii),
and  \eqref{AN1a}
with~\eqref{X1}, we have 
\begin{align}
\begin{split}
\|\vec \I(A_{j, 4})\|_{X^s_{N, j}(\tau)}
& \les  \tau
 \frac 1N \sum_{k = 1}^N
\| Z_{k, k}^{(2)} v_{N, j}\|_{\A_{N, j} C_\tau  W_x^{-\eps, \frac 2{2-s-\eps}}}\\
& \les  \tau
 \frac 1N \sum_{k = 1}^N
\| Z_{k, k}^{(2)}\|_{C_T W^{-\eps, \infty}_x}  
\|v_{N, j}\|_{\A_{N, j} C_\tau
 W_x^{\eps, \frac 2{2-s-\eps}}}\\
 & \les 
 \tau
\| \pmb Z_N \|_{\ZZ_N^\eps(T)}
\|  \vec \vv_N\|_{X^s_N(\tau)}, 
\end{split}
\label{LWP5}
\end{align}

\noi
provided that $\eps \le s  \le  1-\eps$. 
Similarly, by Cauchy-Schwarz's inequality (in $k$),  we have 
\begin{align*}
\begin{split}
\|\vec \I(A_{j, 5})\|_{X^s_{N, j}(\tau)}
& \les  \tau
  \frac 1N \sum_{k = 1}^N  
\|v_{N, k}Z_{k, j}^{(2)}
\|_{\A_{N, j} C_\tau H_x^{s-1}}\\
& \les 
 \tau
\| \pmb Z_N \|_{\ZZ_N^\eps(T)}
\|  \vec \vv_N\|_{X^s_N(\tau)}.
\end{split}
\end{align*}

\noi
Lastly, 
it follows from 
\eqref{XN1}, 
\eqref{AN1a}, 
and \eqref{AN2}
 with 
\eqref{X1} that 
\begin{align}
\begin{split}
\|\vec  \I(A_{j, 6})\|_{X^s_{N, j}(\tau)}
& \les  \tau
 \bigg\| \frac 1N \sum_{k = 1}^N Z_{k, j}^{(3)}
 \bigg\|_{\A_{N, j} C_T W_x^{s-1, \infty}}
\le 
 \tau\|Z_{k, j}^{(3)} \|_{\A_{N, j}^{(2)} C_T W_x^{- \eps, \infty}}\\
& \le 
 \tau
\| \pmb Z_N \|_{\ZZ_N^\eps(T)}, 
\end{split}
\label{LWP7}
\end{align}

\noi 
provided that $s - 1 \le - \eps $.

Denote the right-hand side of \eqref{LWP1}
by $\G_{N, j}(\vec \vv_{N})$, 
and set 
\begin{align*}
 \vec {\pmb \G}_N(\vec \vv_{N}) = 
\big\{(\G_{N, j}(\vec \vv_{N}), \dt \G_{N, j}(\vec \vv_{N}))\big\}_{j = 1}^N.
\end{align*}

\noi
Then, putting \eqref{LWP1}-\eqref{LWP7} together, 
we obtain
\begin{align}
\begin{split}
\|\vec {\pmb \G}_{N}(\vec \vv_{N}) \|_{X^s_N(\tau)}
& \le C_1 \big\| \big(v^{(0)}_{N, j}, v^{(1)}_{N, j}\big) \big\|_{\A_{N, j} \H^s_x}
+ C_2 \tau \|  \vec \vv_N\|_{X^s_N(\tau)}^3
\\
& \quad + C_3 \tau
\| \pmb Z_N \|_{\ZZ_N^\eps(T)}
\Big( 1  + \|  \vec \vv_N\|_{X^s_N(\tau)}\Big)^2
\label{LWP8}
\end{split}
\end{align}

\noi
for any $\vec \vv_N \in X^s_N(\tau)$.
A similar computation yields the following difference estimate:
\begin{align}
\begin{split}
 \|\vec {\pmb \G}_{N} & (\vec \vv_{N}) - \vec {\pmb \G}_{N}(\vec \ww_{N})\|_{X^s_N(\tau)}\\
& 
\le 
C_4  \tau
\Big( 1+ \| \pmb Z_N \|_{\ZZ_N^\eps(T)}\Big)\\
&  \quad \times \Big( 1  + \| \vec \vv_{N}\|_{X^s_N(\tau)}
+ \| \vec \ww_{N}\|_{X^s_N(\tau)}
\Big)^2
\|  \vec \vv_{N} - \vec \ww_{N}\|_{X^s_N(\tau)}
\label{LWP9}
\end{split}
\end{align}

\noi
for any $\vec \vv_N,\vec  \ww_N \in X^s_N(\tau)$.
We point out that  the constants $C_1, C_2, C_3, C_4 > 0$ 
in \eqref{LWP8} and~\eqref{LWP9} are independent of 
 $N \in \N$, $\tau>0$, 
 the initial data 
 $\big(\vv_N^{(0)}, \vv_N^{(1)}\big) = \big\{\big(v^{(0)}_{N, j}, v^{(1)}_{N, j}\big) \big\}_{j = 1}^N$, 
 and the enhanced data set $\pmb Z_N$ in \eqref{enh1x}.
Therefore, 
by choosing
\[\tau = \tau\Big(\big\| \big(v^{(0)}_{N, j}, v^{(1)}_{N, j}\big) \big\|_{\A_{N} \H^s}, 
\| \pmb Z_N \|_{\ZZ_N^\eps(T)})\Big) >0\]

\noi
 sufficiently small, 
we conclude that $\vec {\pmb \G}_N$ is a contraction in the (closed) ball 
$B_R \subset X_N^s(\tau)$ of radius
$R = 2C_1 \big\| \big(v^{(0)}_{N, j}, v^{(1)}_{N, j}\big) \big\|_{\A_{N} \H^s} + 1$.
At this point, the uniqueness holds only in the ball $B_R$
but by a standard continuity argument, 
we can extend the uniqueness to hold
in the entire class $ X_N^s(\tau)$.
We omit details. 
This concludes the proof of Proposition \ref{PROP:LWP1}.
\end{proof}

\subsection{Preliminary estimates} 
\label{SUBSEC:3.2}

In this subsection, we establish preliminary  uniform (in $N$) 
estimates for proving  global well-posedness of \HLS \eqref{NLW4}
by pathwise analysis.

Fix a target time $T \gg1$.
In view of 
the blowup alternative \eqref{blow1}
and 
\eqref{I3}, 
almost sure global well-posedness of \eqref{VN1} on $[0, T]$
follows once we show
the following bound:
\begin{align}
\sup_{0 \le t \le T}
\| (I v_{N, j} , \dt I v_{N, j} ) (t;\o)\|_{\A_{N} \H^1_x} \le C(\o, T) 
\label{VN3a}
\end{align}

\noi
for some almost surely finite constant $C(\o, T) > 0$, 
where 
$\{I v_{N, j}\}_{j = 1}^N$ is a solution to the following $I$-SdNLW system:
\begin{align}
\begin{split}
& (\dt^2 + \dt + m - \Dl) I v_{N, j} \\
& \quad = - \frac{1}{N} \sum_{k = 1}^N  \Big( 
I (v_{N, k}^2 v_{N, j})
+ 2 I (\Psi_k v_{N, k} v_{N, j})
  + I (v_{N, k}^2 \Psi_j )\\
& 
\hphantom{XXXXXX}
 + I (\wick{\Psi_k^2} v_{N, j}) + 2 I(v_{N, k} \wick{ \Psi_k \Psi_j } )
+ I(\wick{\Psi_k^2 \Psi_j}) \Big).
\end{split}
\label{VN3}
\end{align}

\noi
Here,  $I= I_{s, M}$ denotes the $I$-operator  defined in \eqref{I2}
 for some $0 < s < 1$ and $M \gg 1$ (to be chosen later).
Define
the  modified energy $\EN (t)$ by setting
\begin{align}
\EN (t) 
 =  \EN (I \vv_N, \dt I \vv_N) (t), 
\label{EN1a}
\end{align}

\noi
where  $E_N$ is the energy defined in \eqref{E1}.
Then, \eqref{VN3a} follows once we prove
\begin{align*}
\sup_{0 \le t \le T}
\EN(t; \o) \le C'(\o, T) 
\end{align*}

\noi
for some almost surely finite constant $C'(\o, T)> 0$.

Let  $t_2\ge t_1 \ge 0$.
From \eqref{EN1a} and \eqref{VN3},  we have
\begin{align}
\begin{split}
\EN (t_2) - \EN (t_1) 
& = \sum_{\l = 1}^4 B_\l
- \int_{t_1}^{t_2} \int_{\T^2} \frac 1N \sum_{j=1}^{N} ( \dt  I v_{N,j} )^2 d x d t  \\
& \le  \sum_{\l = 1}^4 B_\l, 
\end{split}
\label{EN2}
\end{align}

\noi
where $B_\l = B_\l(N)$, $\l = 1, \dots, 4$, 
are given by 
\begin{align}
\begin{split}
B_1 & = 
 \frac{1}{N^2} \sum_{j, k =1}^{N}
\int_{t_1}^{t_2} \int_{\T^2}
 \dt I v_{N,j}  \Big( ( I v_{N,k} )^2 I v_{N,j} 
- I (  v_{N,k} ^2 v_{N,j} ) \Big) d x d t  \\
B_2 & =  - 
 \frac{1}{N^2} \sum_{j, k =1}^{N}
\int_{t_1}^{t_2} \int_{\T^2}  \dt I v_{N,j}
 \Big(  2 I ( \Psi_k v_{N,k} v_{N,j} ) 
 + I (  v_{N,k} ^2 \Psi_j )\Big) d x d t, \\
B_3 & =  - \frac{1}{N^2} \sum_{j, k =1}^{N}
\int_{t_1}^{t_2} \int_{\T^2}  \dt I v_{N,j}   \Big( I ( \wick{ \Psi_k^2 } v_{N,j} ) 
+ 2 I ( v_{N,k} \wick{ \Psi_k \Psi_j } ) \Big) d x d t, \\
B_4 & =  - 
 \frac{1}{N^2} \sum_{j, k =1}^{N}
\int_{t_1}^{t_2} \int_{\T^2} \dt I v_{N,j} 
\, I ( \wick{ \Psi_k^2 \Psi_j } ) d x d t.
\end{split}
\label{EN3}
\end{align}

\noi
In \eqref{EN2} and \eqref{EN3}, we suppressed
$t$-dependence for notational simplicity;
 we apply the same convention in the following.
The first term $B_1$ is the main commutator term, appearing
in the standard application of the $I$-method, 
while the terms $B_2$, $B_3$, and $B_4$
represent the contributions from the perturbative terms 
in~\eqref{VN3}
which are to be controlled 
by a Gronwall-type argument.
As seen in \cite{GKOT}, 
the key contribution from the perturbative terms
comes from $B_2$
(or rather $B_2'$ defined in \eqref{EN4} below).

In the following, 
we establish uniform (in $N \in \N$)
bounds on $B_\l$, $\l = 1, \dots,4$. 

\begin{lemma}
\label{LEM:com3} 
Let $\frac 23 \le s < 1$. Then, we have 
\begin{align} 
    | B_1 | \les M^{2 - 3s} \int_{t_1}^{t_2} \EN^2 (t) d t
\label{EN5}
\end{align} 

\noi
for any $t_2 \ge t_1 \ge 0$, 
where the implicit constant is independent of $N, M \in \N$.

\end{lemma}

\begin{proof} 
The bound \eqref{EN5}
follows from 
Cauchy-Schwarz's inequality in $x$, 
Lemma \ref{LEM:com1}, 
and 
Cauchy-Schwarz's inequality in $j$ and $k$
with  \eqref{EN1a}.
\end{proof}

By formally distributing the $I$-operator 
to each factor of $B_2$ and $B_3$
in \eqref{EN3}, 
we obtain the following terms: 
\begin{align} 
\begin{split}
B_2' & =  
- \frac{1}{N^2}
\sum_{j, k =1}^N
\int_{t_1}^{t_2} \int_{\T^2}   \dt I v_{N,j} \\
&  \hphantom{XXXXX} 
\times
\Big( 2 (I \Psi_k)( I v_{N,k})( I v_{N,j}) + ( I v_{N,k})^2 I \Psi_j \Big) d x d t,  \\ 
B_3' & = 
- \frac{1}{N^2}\sum_{j, k =1}^N
 \int_{t_1}^{t_2} \int_{\T^2} \dt I v_{N,j} \\
 &  \hphantom{XXXXX} 
\times
  \Big( \big(I ( \wick{ \Psi_k^2 } ) \big)I v_{N,j} + 2 Iv_{N,k} \big(I ( \wick{\Psi_k \Psi_j})\big) \Big) d x d t . 
\end{split}
\label{EN4}
\end{align}

\noi
We first establish  the following commutator estimates.

\begin{lemma}
\label{LEM:com4} 
For  $\l = 2, 3$, 
let $B_\l$ and $B_\l'$ be as in \eqref{EN3} and \eqref{EN4}, respectively.
Let $\frac 23 \le s < 1$.
 Then, given $T \gg 1$ and small $\dl > 0$, we have
\begin{align} 
| B_2 - B_2' | 
& \les
 M^{- s + \frac 12 + \dl} V_N^{(1)} (T, \dl) \int_{t_1}^{t_2} 
\EN ^{\frac 32} (t) d t, 
\label{EN4b}\\
| B_3 - B_3' | 
& \les M^{- \frac {s}{2} + \dl} V_N^{(1)}  (T, \dl) 
\int_{t_1}^{t_2} \EN(t) d t 
\label{EN4c}
\end{align}

\noi
for any 
$0 \le t_1 \le t_2 \le T$ and 
any sufficiently large $M \gg 1$, 
where the implicit constants are  independent of 
$N \in \N$ and  $M,  T \gg 1$.
Here, 
$ V_N^{(1)} ( T, \dl )$
is defined by 
\begin{align} 
\begin{split}
V_N^{(1)} (T, \dl)
& =  \| \Psi_j \|_{\A_N C_T W_x^{-\sigma_0, p_0}} 
+ \| \wick{ \Psi_j^2 } \|_{\A_N C_T W_x^{-\sigma_0, p_0}} \\
& \quad 
 + \| \wick{ \Psi_k \Psi_j } \|_{\A_N^{(2)}C_T W_x^{-\sigma_0, p_0}},
\end{split}
\label{EN4d} 
\end{align} 

\noi
where 
$\A_N B$ and $\A_N^{(2)} B$ are 
the $\l^2$-averages defined  in \eqref{AN0} and \eqref{AN2}, 
while
$\sigma_0 = \sigma_0 (\dl) >0$ and $p_0 = p_0 (\dl) \gg 1$ are as  in Lemma~\ref{LEM:com2}.

\end{lemma}

\begin{proof}
The first bound \eqref{EN4b} follows 
from 
Lemma \ref{LEM:com2}
and Cauchy-Schwarz's inequality, 
while the 
second  bound \eqref{EN4c} follows 
from 
Lemma \ref{LEM:com2}, 
Cauchy-Schwarz's inequality,
and \eqref{AN1a}.
 \end{proof}

Next, we estimate
$B_4$ and $B_3'$  in \eqref{EN3} and \eqref{EN4}, respectively.

\begin{lemma}
\label{LEM:com5} 
Let $B_4$ and $B_3'$ be as in \eqref{EN3} and \eqref{EN4}, respectively.
Let $\frac 23 \le s < 1$.
 Then, given small $\g > 0$, we have 
\begin{align} 
\label{EN6}
| B_4 | & \les M^{\g} V_N^{(2)} (T, \g) \int_{t_1}^{t_2} \EN^{\frac 12}(t) d t, \\
\label{EN7}
| B_3' | & \les M^{\g} V_N^{(2)} (T, \g) \int_{t_1}^{t_2} \EN^{\frac 34}(t) d t
\end{align}

\noi
for any 
$0 \le t_1 \le t_2 \le T$ and 
$M\in \N$,
where the implicit constants are  independent of 
$N, M \in \N$ and  $ T \gg 1$.
Here, 
$V_N^{(2)} ( T, \g )$
is defined by 
\begin{align} 
\begin{split}
V_N^{(2)} (T, \g)
& 
= \| \wick { \Psi_j ^2} \|_{\A_N C_T W_x^{-\g, 4}} 
+ 
 \| \wick{ \Psi_k \Psi_j } \|_{\A_N^{(2)}C_T W_x^{-\g, 4}}\\
& \quad  + \| \wick{\Psi_k^2 \Psi_j} \|_{\A_N^{(2)}C_T W_x^{-\g, 4}}, 
\end{split}
\label{EN8} 
\end{align} 

\noi
where 
$\A_N B$ and $\A_N^{(2)} B$ are 
the $\l^2$-averages defined  
 in \eqref{AN0} and \eqref{AN2}.

\end{lemma}

\begin{proof} 
The first bound \eqref{EN6} on $B_4$ follows 
from 
Cauchy-Schwarz's inequality
in $x$ and then in~$j$
and~\eqref{AN1a}.

As for the second bound \eqref{EN7} on $B_3'$, 
we first write 
\begin{align}
B_3' = B_{31}' + B_{32}',
\label{EN8a} 
\end{align}

\noi
where 
$B_{31}'$ and $B_{32}'$ are given by 
\begin{align*} 
B_{31}' & = 
- \frac{1}{N^2}\sum_{j, k =1}^N
 \int_{t_1}^{t_2} \int_{\T^2} \dt I v_{N,j} \cdot 
I ( \wick{ \Psi_k^2 }) \cdot I v_{N,j} d x d t,  \\
B_{32}' & = 
- \frac{2}{N^2}\sum_{j, k =1}^N
 \int_{t_1}^{t_2} \int_{\T^2} \dt I v_{N,j} 
\cdot 
 Iv_{N,k} \cdot I ( \wick{\Psi_k \Psi_j})  d x d t. 
\end{align*} 

\noi
By Cauchy-Schwarz's inequality in $j$, 
 H\"older's inequality in $x$, 
\eqref{EN1a}, 
\eqref{I4},
and \eqref{AN1a}, we have 
\begin{align}
\begin{split}
| B_{31}' | 
& \leq 
\frac 1{N^2} \sum_{k=1}^N 
\int_{t_1}^{t_2} 
\|  \dt I v_{N, j}\|_{\l_j^2L_x^2} 
\|  I v_{N, j} \|_{L_x^4\l^2_j} 
\| I ( \wick { \Psi_k^2} ) \|_{L_x^{4}}
d t \\
& \les
M^{\g} 
 \|  \wick { \Psi_k^2}  \|_{\A_N C_T W_x^{-\g, 4}}
\int_{t_1}^{t_2}\EN^{\frac 34}(t)
d t , 
\end{split}
\label{EN8b}
\end{align} 

\noi
where the implicit constant is independent of 
$N, M \in \N$ and  $ T \gg 1$.
Similarly, we have
\begin{align}
\begin{split}
| B_{32}' | 
& \les 
\frac 1{N^2}
\int_{t_1}^{t_2} 
\|  \dt I v_{N, j}\|_{\l_j^2L_x^2} 
\|  I v_{N, k} \|_{L_x^4\l^2_k} 
\| I ( \wick { \Psi_k \Psi_j} ) \|_{\l^2_{j, k} L_x^{4}}
d t\\
& \les
M^{\g} 
 \|  \wick { \Psi_k \Psi_j }  \|_{\A_N^{(2)} C_T W_x^{-\g, 4}}
\int_{t_1}^{t_2}\EN^{\frac 34}(t)
d t, 
\end{split}
\label{EN8c}
\end{align} 

\noi
 where the implicit constant is independent of 
$N, M \in \N$ and  $T \gg 1$.
Hence, the second bound~\eqref{EN7}
follows from \eqref{EN8a}, \eqref{EN8b}, and \eqref{EN8c}
with \eqref{EN8}.
\end{proof}

Lastly, we 
estimate $B_2'$
in \eqref{EN4}, 
where the logarithmically divergent nature
of $ I \Psi_j $ plays an important role.

\begin{lemma} 
\label{LEM:com6}  

Let $B_2'$ be as in \eqref{EN4}.
Then, given $s < 1$, 
there exists $c > 0$ such that 
\begin{align} 
\label{EN9}
| B_2' | \les 
\bigg\{\int_{t_1}^{t_2} 
\Big(1 + 
 \EN^{1 + c\eta} (t)+ \frac {\eta}{(t - t_1)^{\frac 12}} \Big)  d t\bigg\}
\| I \Psi_j \|_{L^{\eta^{-1}}_{[t_1, t_2], x} \A_N } 
\end{align} 

\noi 
for any 
$t_2 \geq t_1 \geq 0$
and small $\eta > 0$, 
uniformly in 
 $N, M \in \N$,
where
$L^r_{[t_1, t_2], x}\A_N
= L^r_{[t_1, t_2]}L^r_x \A_N$.

\end{lemma}

\begin{proof} 
The bound \eqref{EN9} follows
from a straightforward modification
of the proof of 
\cite[Lemma~3.4]{GKOT}
adapted to the current vector-valued setting.
We first note that 
by interpolation, 
Minkowski's integral inequality, 
and Sobolev's inequality with \eqref{EN1a}, we have
 \begin{align}
\begin{split}
\|I v_{N, j}\|_{L^{\frac 4 {1- 2\eta}}_x\A_N}
& \le
\|I v_{N, j}\|_{L^4_x\A_N}^{1 - 4\eta}
\|I v_{N, j}\|_{L^8_x\A_N}^{4\eta}
 \le
\|I v_{N, j}\|_{L^4_x\A_N}^{1 - 4\eta}
\|I v_{N, j}\|_{\A_NL^8_x}^{4\eta}\\
& \les 
\|I v_{N, j}\|_{L^4_x\A_N}^{1 - 4\eta}
\|I v_{N, j}\|_{\A_N H^1_x}^{4\eta}
\\
&   \les 
\EN^\frac {1-4\eta} 4
\EN^{2\eta}
= 
\EN^{\frac{1+4\eta}{4}} , 
\end{split}
\label{EN9z}
 \end{align}

\noi
uniformly in  $N, M \in \N$ and small $\eta > 0$.

Write $B_2'$ in \eqref{EN4}
as 
\begin{align*}
B_2' = B_{21}' + B_{22}',
\end{align*}

\noi
where 
$B_{21}'$ and $B_{22}'$ are given by 
\begin{align} 
\begin{split}
B_{21}' & =  
- \frac{2}{N^2}
\sum_{j, k =1}^N
\int_{t_1}^{t_2} \int_{\T^2}   \dt I v_{N,j} \cdot
 (I \Psi_k)( I v_{N,k})( I v_{N,j}) d x d t, \\
B_{22}' & =  
- \frac{1}{N^2}
\sum_{j, k =1}^N
\int_{t_1}^{t_2} \int_{\T^2}   \dt I v_{N,j} \cdot
 ( I v_{N,k})^2 \cdot I \Psi_j  d x d t.
\end{split}
\label{EN9b}
\end{align}

\noi
By Cauchy-Schwarz's inequality (in $j$ and $k$), 
H\"older's inequality (in $x$), 
and \eqref{EN9z}, we have 
 \begin{align*}
|B_{22}'|
&\le \int_{t_1}^{t_2}
\|\dt Iv_{N, j}\|_{L^2_x\A_N} 
\|Iv_{N, k}\|_{L^{\frac 4 {1- 2\eta}}_x\A_N}^2
 \|I\Psi_j \|_{L_x^{\eta^{-1}}\A_N} dt \\
& \les \int_{t_1}^{t_2} \EN^{1 +2 \eta} (t) \|I \Psi_j (t)\|_{L_x^{\eta^{-1}}\A_N} dt \\
& \les \bigg(\int_{t_1}^{t_2} \EN^{\frac{1 + 2\eta}{1-\eta}}(t)dt\bigg)^{1-\eta}
\|I \Psi_j \|_{L^{\eta^{-1}}_{[t_1, t_2], x} \A_N}\\
& \lesssim 
\bigg\{\int_{t_1}^{t_2}
\Big( \EN^{\frac{1+2\eta}{1-2\eta}}(t) + \frac{\eta}{(t-t_1)^\frac12}\Big) dt
\bigg\}\|I \Psi_j \|_{L^{\eta^{-1}}_{[t_1, t_2], x} \A_N}\\
& \les \bigg\{\int_{t_1}^{t_2} 
\Big(1 + 
 \EN^{1 + c\eta} (t)+ \frac {\eta}{(t - t_1)^{\frac 12}}  \Big) d t\bigg\}
\| I \Psi_j \|_{L^{\eta^{-1}}_{[t_1, t_2], x} \A_N } , 
 \end{align*}
 
 \noi
uniformly in $N, M \in \N$ and
 small $ \eta > 0$, 
where the penultimate step follows from \cite[the math display before~(3.21)]{GKOT}
as a consequence of H\"older's and Young's inequalities.
The contribution from $B_{21}'$ in \eqref{EN9b}
can be treated in an analogous manner, 
thus yielding \eqref{EN9}.
\end{proof}

\subsection{Global well-posedness} 
\label{SUBSEC:3.3}

In this subsection, we prove
pathwise  global well-posedness of \HLS~\eqref{NLW4}
(Theorem~\ref{THM:GWP1}\,(ii)).
In view of the first order expansion \eqref{exp}, 
Theorem~\ref{THM:GWP1}\,(ii)
follows once we prove the following proposition.

\begin{proposition}
\label{PROP:GWP1}
Let $m > 0$ and $ \frac 45 < s < 1$.
Given $N \in \N$, 
let $\big(\vv_{N}^{(0)},  \vv_{N}^{(1)}\big) = \big\{\big(v_{N, j}^{(0)},  v_{N, j}^{(1)}\big)\big\}_{j = 1}^N \in \big(\H^s (\T^2)\big)^{\otimes N}$. Then, there exists a unique solution 
$(\vv_N, \dt \vv_N) = \{( v_{N, j}, \dt  v_{N, j} )\}_{j = 1}^N$
in the class
$\big(  C (\R_+; \H^s (\T^2) )\big)^{\otimes N}$
 to~\eqref{VN1} with initial data
$\big(\vv_{N}^{(0)},  \vv_{N}^{(1)}\big)$, 
almost surely.

\end{proposition}

The rest of this subsection is devoted to a proof of Proposition \ref{PROP:GWP1}, 
with a particular emphasis on uniformity in $N \in \N$;
see, for example, Remark \ref{REM:unif1}.
Fix $  \frac 45 < s < 1$.
Given small $\dl, \g > 0$ and a target time $T \gg1$, 
it follows from 
 \eqref{EN2}, \eqref{EN3}, \eqref{EN4}, Lemmas \ref{LEM:com3}, \ref{LEM:com4}, \ref{LEM:com5}, and \ref{LEM:com6}
 that 
\begin{align} 
\begin{split}
& \EN(t_2)  - \EN(t_1) \\
& \quad \les M^{2 - 3s} \int_{t_1}^{t_2} \EN^2 (t) d t \\ 
&\quad \quad 
+ \sum_{\l = 1}^2
M^{-  \frac {1 - \l(1-s)}2 +  \dl} 
V_N^{(1)} (T, \dl) \int_{t_1}^{t_2} \EN^\frac {\l + 1}2 (t)dt\\
&\quad \quad 
+ M^{\g} V_N^{(2)} (T, \g) \int_{t_1}^{t_2} \big(1 + \EN^\frac 34 (t)\big) d t 
\\
&\quad \quad 
+ \bigg\{\int_{t_1}^{t_2}  
\Big(1 + \EN^{1 + c\eta} (t) + \frac {\eta}{(t - t_1)^{\frac 12}} \Big)d t\bigg\}
\| I \Psi_j \|_{L^{\eta^{-1}}_{[t_1, t_2], x} \A_N } 
\end{split}
\label{EN10}
\end{align} 

\noi 
for any $0\leq t_1 \leq t_2 \leq T$
and small $\eta > 0$, 
uniformly in $N \in \N$ and $M \gg1$, 
where $V_N^{(1)} (T, \dl)$ and 
$V_N^{(2)} (T, \g)$
are as in \eqref{EN4d} and \eqref{EN8}, respectively.

Before we state a crucial lemma (Lemma \ref{LEM:GWP1a})
that allows us to proceed with an iterative argument, 
we introduce two auxiliary random variables.
Given $N \in \N$, let $R_N$ be as in Lemma~\ref{LEM:sto2}
which controls $I \Psi_j$ on the last term in \eqref{EN10}.
The next random variable $V_N$ controls
the other 
contributions from the stochastic terms in \eqref{EN10}.
By choosing $\eps = \min(\s_0, \g)$, 
it follows from~\eqref{EN4d} and \eqref{EN8}
that 
\begin{align*}
V_N^{(1)} (T, \dl)
+ V_N^{(2)} (T, \g)
\les 
\| \pmb{\Psi}_N\|_{\ZZ_N^\eps(T)}, 
\end{align*}

\noi
where 
$\ZZ_N^\eps(T)$ is as 
in \eqref{X1}
and $\pmb \Psi_N$ is given by 
\begin{align*}
\pmb \Psi_N = \big\{
\Psi_j,  \, \wick{\Psi_k \Psi_j} \, , 
 \, \wick{\Psi_k^2 \Psi_j} \big\}_{j, k = 1}^N.
\end{align*}

\noi
Define a random variable $V_N$ by 
\begin{align}
e^{V_N^\frac{1}{3}}
 = \sum_{\l = 0}^\infty e^{-\ta \l} e^{\| \pmb \Psi_N\|_{\ZZ^\eps(\l)}^\frac{1}{3}}
\label{EN12}
\end{align}

\noi
for some  $\ta > 0$.
Then, by applying 
Lemma \ref{LEM:sto1}\,(ii)
and choosing  $\ta  > 0$ sufficiently large, we have 
\begin{align}
\sup_{N \in \N} \E\Big[ 
e^{V_N^\frac{1}{3}} \Big] 
= 
\sup_{N \in \N}
\sum_{\l = 0}^\infty e^{-\ta \l} \E\Big[e^{\| \pmb \Psi_N\|_{\ZZ^\eps(\l)}^\frac{1}{3}}\Big]
\les \sum_{\l = 0}^\infty e^{-\ta \l} e^{c \l} < \infty.
\label{EN13}
\end{align}

\noi
From \eqref{EN12}, we also have
\begin{align}
 \| \pmb \Psi_N\|_{\ZZ_N^\eps(T)}
\le \Big( V_N^\frac{1}{3} + \ta (1 + T) \Big)^3
\les V_N + T^3.
\label{EN14}
\end{align}

\noi
The left-hand side of \eqref{EN14}
plays a role of $M_T$ 
in \cite[(3.24)]{GKOT}.

The following lemma
is an analogue of
\cite[Proposition 3.5]{GKOT}
in the current vector-valued setting.

\begin{lemma}
\label{LEM:GWP1a}
Let $\frac 23 < s < 1$
and $T \gg 1$.
Given $N \in \N$, 
 let $R_N$ and $V_N$ be as in Lemma~\ref{LEM:sto2}
 and~\eqref{EN12}, respectively, 
 which are almost surely finite.
Fix $\o \in \O$
such that $R_N(\o), V_N(\o) < \infty$.
Then, there exist $\al = \al(s) $,  $\be = \be(s)> 0$, 
and  $M_* = M_*(s) \in \N$, 
satisfying
\begin{align}
\be < \al \le 3s - 2, 
\label{EN15a}
\end{align}

\noi
 such that 
if \begin{align}
\EN(t_0) \le M^\beta
\label{EN15}
\end{align}

\noi
for some $0 \le t_0 < T$
and   $M \ge M_*$,
 then
there exists
small 
\[\tau_N = \tau_N (s, M, T, R_N(\o), V_N(\o)) = \tau_N (s, M, T, \o) >0 \] 

\noi
such that 
\begin{align*}
\EN(t) \le M^\alpha
\end{align*} 

\noi
for any $t $
satisfying
$t_0 \le t \le \min(T, t_0 + \tau_N)$. 

\end{lemma}

\begin{remark}\label{REM:unif1}\rm

In Lemma \ref{LEM:GWP1a}, 
the dependence of $\tau_N$ on $N$ 
only appears through
$R_N$ and $V_N$.
In view of 
the uniform (in $N$) bounds
\eqref{mom4}
and
\eqref{EN13}
with Chebyshev's inequality, 
we see that 
Lemma \ref{LEM:GWP1a}
holds ``uniformly in $N \in \N$''
in the sense that the statistical properties
of $\tau_N$ can be 
estimated 
uniformly in $N \in \N$.

\end{remark}

\begin{proof}
[Proof of Lemma \ref{LEM:GWP1a}]
Once we have the bound \eqref{EN10}, 
this proposition follows from repeating
the proof of 
\cite[Proposition~3.5]{GKOT}
based on a 
double Gronwall-type inequality
and 
we only indicate the basic idea.
As in \cite[(3.32) and (3.38)]{GKOT}, 
define $F_N$ by 
\begin{align}
\begin{split}
 F_N(t) & = \max_{t_0 \le \tau \le t}  \EN(\tau) - \EN(t_0) + \max(\EN(t_0), M^\be), \\
 G_N(t) & = F_N(t) - 2 C_0 R_N (t - t_0)^\frac{1}{2}
\end{split}
\label{EN16a}
\end{align}

\noi
for some suitable $C_0 > 0$.
Then, after some manipulation on 
\eqref{EN10} with \eqref{EN16a}, 
we arrive at 
\begin{align} 
\begin{split}
G_N(t)  -  G_N(t_0) 
\les (1 + V_N + R_N  + T)
\int_{t_0}^{t}  G_N(t') (\log G_N(t') + T^2)  dt'
\end{split}
\label{EN16b}
\end{align}

\noi
for any $t_0 \le t \le \min( t_1, t_0 + t_*(C_0, R_N))$
such that 
$2 C_0 R (t - t_0)^\frac{1}{2} \sim 1$, 
guaranteeing
$G_N(t) \sim F_N(t)$;
see \cite[(3.39)]{GKOT}. 
We emphasize that the inequality \eqref{EN16b} holds uniformly in $N \in \N$, $M, T \gg 1$.
From the differential inequality \eqref{EN16b}, 
we obtain a double exponential bound on $G_N(t)$,
which, together with \eqref{EN15}, 
implies the conclusion of this proposition.
Lastly, we note that 
the condition \eqref{EN15a} follows from 
\cite[(3.36)]{GKOT}
and that 
the condition
$M \ge M_* = M_*(s)$
comes from 
\cite[(3.45) and (3.46)]{GKOT}.
\end{proof}

We conclude this section by presenting
a proof of
 Proposition~\ref{PROP:GWP1}.

\begin{proof}[Proof of Proposition~\ref{PROP:GWP1}] 
This proposition follows from 
a straightforward modification of the proof of 
 \cite[Theorem 1.2]{GKOT}.
We, however, present some details
in order to demonstrate  that  the following iterative argument
provides
a uniform (in $N$) bound on the $\A_N\H^s_x$-norm 
of a solution  $\vec \vv_N (t) = (\vv_N, \dt \vv_N)(t)$.
See also Remark \ref{REM:growth1}.

Let $\frac 45 < s < 1$
and fix  a target time $T \gg1$.
Moreover, 
given $N \in \N$, 
fix 
 $\om \in \Omega$ 
such that 
\begin{align}
R_N(\o), V_N(\o) < \infty.
\label{EN17a}
\end{align}

\noi
Then, let the parameters $\al, \be, M_*$ and $\tau_N$
be as in Lemma~\ref{LEM:GWP1a}.

Fix  $M_0 \gg1$ (to be determined later)
and define an increasing sequence $\{M_k\}_{k \in \N}$ by setting 
\begin{align*} 
    M_k = M_0^{\sigma^k} 
\end{align*}

\noi 
for some $\sigma  \gg 1$.
Under the assumption
\begin{align}
2 (1-s) < \be, 
\label{EN18}
\end{align}

\noi
we can take 
$\sigma = \sigma(s, \al, \be) \gg 1$
 sufficiently large such that 
\begin{align} 
    M_{k+1}^{2(1-s)} M_k^{\al} + M_k^{2\al} \ll M_{k+1}^{\be}, 
\label{EN19}
\end{align} 

\noi 
We point out that the constraint $s > \frac 45$
appears 
from  \eqref{EN15a}
and \eqref{EN18}.

Given $N, M \in \N$
and $t \ge 0$, set
\begin{align}
\ENM(t) = \EN (I_M \vv_N, \dt I_M \vv_N) (t)
\label{EN19a}
\end{align}

\noi
where the right-hand side is as in  \eqref{E1}.
Suppose that 
\begin{align} 
E_{N, M_k}(t)  \leq M_{k}^{\al}
\label{EN20}
\end{align}

\noi
for some $k \in \N \cup \{0\}$ and $t \geq 0$. Then, by  \eqref{I3}, Minkowski's integral inequality, 
\eqref{I4}, Sobolev's inequality (with $s\ge \frac 12$), \eqref{EN20}, and \eqref{EN19}, we have 
\begin{align}
\begin{split}
 E_{N, M_{k+1}}(t)
& \les  \| ( I_{M_{k+1}} \vv_N, \dt I_{M_{k+1}} \vv_N ) \|^2_{\A_N  \H_x^{1}} 
+ \| I_{M_{k+1}}\vv_N \|^4_{L_x^4 \A_N}  \\ 
& \les M_{k+1}^{2(1-s)} \| ( \vv_N, \dt \vv_N ) \|^2_{\A_N \H_x^s} 
+ \| I_{M_{k+1}} \vv_N \|^4_{\A_N L_x^4}  \\ 
& \les M_{k+1}^{2(1-s)} \| (I_{M_{k}} \vv_N, \dt I_{M_k} \vv_N ) \|^2_{\A_N \mathcal{H}_x^{1}} +
 \| \vv_N \|^4_{\A_N H_x^{\frac 12}}  \\ 
& \les M_{k+1}^{2(1-s)} M_{k}^{\al} 
+ \| I_{M_{k}} \vv_N \|^4_{\A_N H_x^{1}}  \\ 
& \les M_{k+1}^{2(1-s)} M_{k}^{\al} + M_k^{2\al}  \\ 
& \ll M_{k+1}^{\be}, 
\end{split}
\label{EN21}
\end{align}

\noi
where, at the fourth inequality, we use the fact that $s\ge \frac 12$
(in applying \eqref{I3}).

Given $\big(\vv_N^{(0)}, \vv_N^{(1)}\big) \in \A_N \H_x^{s}(\T^2)$
and $\o \in \O$, satisfying \eqref{EN17a}, 
let 
$ \vec \vv_N(\o)
=  (\vv_N(\o), \dt \vv_N(\o))$
be the local-in-time solution to \eqref{VN1}
with $\vec \vv_N(\o)|_{t = 0} = 
\big(\vv_N^{(0)}, \vv_N^{(1)}\big)$.
In view of 
the blowup alternative \eqref{blow1}, 
the desired global well-posedness of \eqref{VN1}
on the time interval $[0, T]$
(for the fixed $\o \in \O$) follows
once we prove 
\begin{align}
\sup_{0 \le t \le T} \| \vec  \vv_N(t; \o)\|_{\A_N  \H^s_x}
\le C(\o, T) < \infty.
\label{EN21a}
\end{align}

\noi
In the following, we prove \eqref{EN21a}
by implementing an iterative argument
based on Lemma~\ref{LEM:GWP1a}.

Choose 
 $M_0 
 = M_0\big(\vv_N^{(0)}, \vv_N^{(1)}, s, T\big)
  \gg 1$
such that 
\begin{align*} 
 E_{N, M_0}(0) \leq M_{0}^\be.
\end{align*}

\noi
Then, it follows from Lemma  \ref{LEM:GWP1a}
that 
\begin{align*}
\sup_{0 \le t \le \tau_N} 
 E_{N, M_0}(t) \leq M_{0}^\al, 
\end{align*}

\noi
satisfying \eqref{EN20} at time $\tau_N$. 
From 
 \eqref{EN19a}, 
\eqref{E1}, 
and \eqref{I3}, 
this in particular implies
\begin{align*}
\sup_{0 \le t \le \tau_N} \| \vec  \vv_N(t; \o)\|_{ \A_N  \H^s_x}^2
\les M_{0}^\al.
\end{align*}

\noi
Moreover, 
 from \eqref{EN21}, 
we obtain
\begin{align*}
 E_{N, M_1}(\tau_N) \leq M_{1}^\be. 
\end{align*}

\noi
By applying 
Lemma \ref{LEM:GWP1a} once again, 
we have 
\begin{align*}
\sup_{\tau_N \le t \le 2 \tau_N} 
 E_{N, M_1}(t) \leq M_{1}^\al, 
\end{align*}

\noi
satisfying \eqref{EN20} at time $2\tau_N$, 
from which we obtain 
\begin{align*}
\sup_{\tau_N \le t \le 2\tau_N} \| \vec  \vv_N(t; \o)\|_{\A_N  \H^s_x}^2
\les M_{1}^\al
\end{align*}

\noi
and 
\begin{align*}
 E_{N, M_2}(2\tau_N) \leq M_{2}^\be. 
\end{align*}

\noi
By  iterating this argument $\big[\frac{T}{\tau_N}\big] + 1$ times, we
obtain the following bound:
\begin{align*}
\sup_{\l \tau_N \le t \le (\l+1)\tau_N} \| \vec  \vv_N(t; \o)\|_{ \A_N  \H^s_x}^2
\les M_{\l}^\al
\end{align*}

\noi
for $\l = 0, \dots, \big[\frac T {\tau_N}\big] + 1$, 
yielding \eqref{EN21a}.
This proves  existence of the unique solution $\vv_N$ to~\eqref{VN1}
on the time interval $[0, T]$
(for the fixed $\o \in \O$, satisfying~\eqref{EN17a}).
Since the choice of $T \gg1 $
was arbitrary, 
we conclude almost sure global well-posedness of \eqref{VN1}.
\end{proof}

\begin{remark}\label{REM:growth1}
\rm

As pointed out in 
\cite[Remark 3.7]{GKOT}, 
the iterative argument presented above
yields
the following uniform (in $N$)
double exponential bound:
\begin{align*}
\| \vec \vv_N(t)\|_{\A_N\H^s_x} 
\le
C \exp\Big( c \log 
\big(2 + \| \vec \vv_N(0)\|_{\A_N\H^s_x}\big)
\cdot e^{C(\o)  t^2}\Big)
\end{align*}

\noi
for any $t \geq 0$.
See \cite[Remark 3.7]{GKOT} for details.

\end{remark}

\section{Well-posedness of the mean-field SdNLW} 
\label{SEC:4}
In this section, we present a proof of  Theorem~\ref{THM:GWP2}
on local  and global well-posedness of the mean-field SdNLW \eqref{MF2}. By the first order expansion \eqref{expj}, 
we study the following  equation
(see \eqref{NLW5})
satisfied by the residual part $v = u - \Psi$, where  we drop the superscript $j$ 
for notational simplicity:
\begin{align} 
\begin{cases}
(\dt^2 + \dt + m - \Delta) v 
= 
 - \E [ v^2 ] v
 - 2 \E [ \Psi v ] v - \E [ v^2 ] \Psi
- 2 \E [ v \Psi ] \Psi \\
(v, \dt v)|_{t = 0} = (v^{(0)}, v^{(1)}). 
\end{cases}
\label{WN1} 
\end{align}

\noi
As pointed out in Section \ref{SEC:1}, 
 there are no terms involving $\E[\, \wick{\Psi^2}\,]$
since $\E[ \,\wick{\Psi^2}\,] = 0$.
We note that there are certain similarities between \eqref{WN1} and \eqref{VN1}
after replacing the empirical average 
$\frac 1N \sum_{1 \leq k \leq N}$ in~\eqref{VN1}
with the expectation $\E$, 
leading to computations similar to those
presented in Section \ref{SEC:3}.
There are, however,  
fundamental differences
between
 the mean-field SdNLW~\eqref{WN1} and the SdNLW system
 \eqref{VN1}, in particular on  the treatment of the following terms: 
\[ \E [v\Psi ]\Psi
\qquad \text{and}\qquad  
\frac 1N \sum_{k = 1}^N v_{N,k} \wick{\Psi_k \Psi_j}\,,  \]

\noi
where the former term does {\it not} have an explicit renormalization.
%
%
In order to handle this term, 
we 
write 
\begin{align}
    \E [v\Psi ] \Psi(\o) = \E [ v'\Psi'  ] \Psi(\o)
= \int_\O
v' (\o') \Psi' (\o') \Psi(\o)
 \PP(d\o')    , 
\label{WN3}
\end{align}

\noi
where  $(v',  \Psi')$ is an independent copy of $(v, \Psi)$.
In particular, the product 
$\Psi' \Psi$ makes sense {\it without} renormalization; see \eqref{WN3a} below.
See also the proofs of Lemmas \ref{LEM:CM1}, \ref{LEM:CM2}, \ref{LEM:CM3}, and \ref{LEM:CM4}, 
where we apply a similar idea.

In Subsection~\ref{SUBSEC:4.1}, we prove local well-posedness of \eqref{WN1}. 
In Subsection~\ref{SUBSEC:4.2}, we briefly discuss 
global well-posedness of \eqref{WN1}.

\subsection{Local well-posedness} 
\label{SUBSEC:4.1}

In this subsection, 
we prove the following proposition on local well-posedness
of the mean-field SdNLW \eqref{WN1}
(Theorem \ref{THM:GWP2}\,(i))
whose proof follows closely that of Proposition~\ref{PROP:LWP1}
but with one twist due to the term $\E[ v\Psi]\Psi$
as discussed above.


\begin{proposition}
\label{PROP:LWP2}

Let  $m > 0$ and  $\frac 12 \le s < 1$. 
Given  $2 \le p < \infty$, 
let  $(v^{(0)}, v^{(1)}) \in L^p (\Omega; \H^s (\T^2) )$.
Then, given any $t_0 \in \R_+$, 
there exists a unique solution $(v, \dt v) \in L^p(\O; C([t_0, t_0+ \tau]; \H^s(\T^2)))$
to~\eqref{WN1}
with $(v, \dt v)|_{ t= t_0} = (v^{(0)}, v^{(1)})$, 
where the local existence time $\tau > 0$ depends
on 
$\|(v^{(0)}, v^{(1)})\|_{L^p_\o \H^s_x}$ and~$\Psi$.
\end{proposition}

\begin{remark}
\label{REM:LWP2} \rm

(i)
As pointed out in Remark \ref{REM:LWP1}\,(i), 
it is possible to lower the regularity threshold
by using the Strichartz estimates.
We, however, 
restrict our attention to the range $\frac 12 \le s < 1$,
which guarantees
unconditional uniqueness
of solutions
(namely, uniqueness
 in the entire class 
$L^p(\O; C(\tau; \H^s(\T^2)))$), 
since it
plays a key role in Section \ref{SEC:6}.

\smallskip

\noi
(ii) Let $T \ge 1$.
Let $(v, \dt v)$ be the local-in-time solution 
to \eqref{WN1} constructed in Proposition \ref{PROP:LWP2}.
Then, 
by denoting by $T_*$ its maximal existence time, 
the following blowup alternative holds:
\begin{align}
T_* = T\qquad \text{or}\qquad 
\lim_{t \to T_*-}
\| (v, \dt v)(t)\|_{L^p_\o \H^s_x}
= \infty.
\label{blow2}
\end{align}


\end{remark}

\begin{proof}[Proof of Proposition \ref{PROP:LWP2}] 
Without loss of generality, we assume that $t_0 = 0$.
By writing  \eqref{WN1} in the Duhamel formulation, 
we have 
\begin{align}
\begin{split}
v(t) 
& = S (t) (v^{(0)}, v^{(1)}) 
- \I \big(\E [ v^2 ] v\big)(t) 
- 2 \I \big(\E [ \Psi v ] v\big)(t)\\
& \quad 
 -  \I \big(\E [ v^2 ] \Psi \big)(t)
 - 2 \I \big(\E [ v \Psi] \Psi\big)(t),  \\
\end{split}
\label{WN2}
\end{align} 

\noi
where $S (t)$ and  $\mathcal{I}$
are as in \eqref{D4} and 
 \eqref{D5}, respectively.

Let $2 \le p < \infty$.
Fix $0 < \tau \le 1$
and small $\eps > 0$
such that $s \le 1-\eps$.
From Lemma \ref{LEM:sto1}, we have 
\begin{align}
K_p =  \| \Psi  \|_{L^p_\o C_1 W_x^{-\eps, \infty}}
+ 
\|\Psi' \Psi  \|_{L^p_\o C_1 W_x^{-\eps, \infty}} < \infty, 
\label{WN3a}
\end{align}

\noi
where $\Psi'$ is an independent copy of $\Psi$ as in \eqref{WN3}.

We first consider the last term on the right-hand side
of \eqref{WN2},
involving $\E [ v\Psi  ] \Psi$.
Recall the notation \eqref{LWP2a}.
We  apply \eqref{WN3}
and  proceed as in \eqref{LWP5}
with 
 Sobolev's inequality
and 
Lemma~\ref{LEM:prod}\,(ii).
Then, from 
the independence of $\Psi'$ and $\Psi$
with $\Law (v') =  
\Law (v)$
and~\eqref{WN3a}, 
we have 
\begin{align}
\begin{split}
& \big\| \vec \I\big(\E [v\Psi] \Psi \big)\big\|_{L_\om^p C_\tau \H_x^{s}}
\les \tau 
\big\| \|v'(\o') \Psi'(\o') \Psi(\o)\|_{C_\tau W_x^{-\eps, \frac 2{2-s-\eps}} } \big\|_{L^p_\o L^1_{\o'}} \\
& \quad 
\les \tau 
\Big\|\|\Psi' (\o')\Psi (\o) \|_{C_1 W_x^{-\eps, \infty}} 
\| v'  (\o') \|_{C_\tau W_x^{\eps, \frac 2{2-s-\eps}}} 
\Big\|_{L^p_\o L^1_{\o'}}
\\
& \quad 
\le \tau 
\|\Psi' (\o')\Psi (\o) \|_{L^p_{\o, \o'}C_1 W_x^{-\eps, \infty}}
\| v'   \|_{L^2_{\o'} C_\tau W_x^{\eps, \frac 2{2-s-\eps}}}
\\
& \quad 
=  \tau 
\|\Psi' (\o)\Psi (\o) \|_{L^p_{\o}C_1 W_x^{-\eps, \infty}}
\| v   \|_{L^2_{\o} C_\tau W_x^{\eps, \frac 2{2-s-\eps}}}
\\
& \quad 
\les  \tau 
K_p
\| v   \|_{L^p_\o C_\tau H_x^s}, 
\end{split}
\label{WN4}
\end{align}

\noi 
provided that $\eps \le s \le 1 - \eps$, 
where, at the third inequality, we used the fact that $p \ge 2$.

Let us now handle the remaining terms.
Proceeding as in \eqref{LWP2}, we have 
\begin{align}
\begin{split}
\big\| \vec \I\big(\E [v^2] v \big)\big\|_{L_\om^p C_\tau \H_x^{s}}
& \les 
\tau 
\|\E [v^2] v \|_{L_\om^p C_\tau H_x^{s-1}}\\
& \les 
\tau 
\Big\| \E \big[ \| v \|^2_{C_\tau L_x^\frac 6{2-s}} \big]
 \| v \|_{C_\tau L_x^\frac 6{2-s}} \Big\|_{L_\o^p} \\
&\les
\tau  \| v \|^3_{L_\o^p C_{\tau} H_x^s}, 
\end{split}
\label{WN5}
\end{align} 

\noi
provided that $\frac 12 \le s $ ($< 1$).
Proceeding as in \eqref{LWP3} with Lemma \ref{LEM:prod}\,(ii)
(but in a slightly different order due to the presence of the expectation), 
we have 
\begin{align}
\begin{split}
\big\| \vec \I\big(\E [\Psi v] v \big)\big\|_{L_\om^p C_\tau \H_x^{s}}
&\les \tau \|\E [\Psi v] v \|_{L_\om^p C_\tau W_x^{-\eps, \frac {2}{2-s-\eps}}}\\ 
&\les \tau \Big\| \E \big[ \| \Psi v \|_{C_\tau  W_x^{-\eps, \frac 4{2- s- \eps}}} \big]
\| v \|_{C_1 W_x^{\eps, \frac 4{2- s- \eps}}} \Big\|_{L_\o^p} \\
&\les \tau 
 \| \Psi \|_{L_\om^p C_1 W_x^{-\eps, \infty}}
\| v \|_{L_\om^p C_\tau H_x^{\frac s2 + \frac 32\eps }}^2  \\
&\les \tau 
K_p
\| v \|_{L_\o^p C_\tau H_x^s}^2, 
\end{split}
\label{WN6}
\end{align} 

\noi
provided that $3\eps  \le s $ ($< 1$).
Similarly, 
we have 
\begin{align}
\begin{split}
\big\| \vec \I\big(\E [v^2] \Psi \big)\big\|_{L_\om^p C_\tau \H_x^{s}}
&\les \tau \|\E [v^2] \Psi \|_{L_\om^p C_\tau W_x^{-\eps, \frac {2}{2-s-\eps}}}\\ 
&\les \tau \Big\| \E \big[ \| v^2 \|_{C_\tau  W_x^{\eps, \frac 2{2- s- \eps}}} \big]
\| \Psi \|_{C_1 W_x^{-\eps, \infty}} \Big\|_{L_\o^p} \\
&\les \tau 
 \| \Psi \|_{L_\om^p C_1 W_x^{-\eps, \infty}}
\| v \|_{L_\om^p C_\tau H_x^{\frac s2 + \frac 32\eps }}^2  \\
&\les \tau 
K_p
\| v \|_{L_\o^p C_\tau H_x^s}^2, 
\end{split}
\label{WN7}
\end{align} 

\noi
provided that $3\eps  \le s $ ($< 1$).

Denote the right-hand side of \eqref{WN2}
by $\G(\vec v) = \G(v, \dt v)$, 
and set 
$\vec \G(\vec v)
= \big(\G(\vec v), \dt \G(\vec v)\big)$.
Then, putting \eqref{WN2}, \eqref{WN4}, \eqref{WN5}, 
\eqref{WN6}, and \eqref{WN7}  together, 
we obtain
\begin{align*}
\|\vec \G(\vec v) \|_{L^p_\o C_\tau \H^s_x}
& \le C_1 \| (v^{(0)}, v^{(1)}) \|_{L^p_\o \H^s_x}
+ C_2
\tau  \| \vec v \|^3_{L_\o^p C_{\tau} \H_x^s}
\\
& \quad + C_3 \tau
K_p 
\Big( 1  + \| \vec v   \|_{L^p_\o C_\tau \H_x^s}\Big)
\| \vec v   \|_{L^p_\o C_\tau \H_x^s} 
\end{align*}

\noi
for any $\vec v \in L^p(\O; C([0, \tau]; \H^s(\T^2)))$,
where $K_p$
is as in \eqref{WN3a}.
A similar computation yields the following difference estimate:
\begin{align*}
& \|\vec\G( \vec v) - \vec \G (\vec w)\|_{L^p_\o C_\tau \H^s_x}\\
& \quad 
\les 
  \tau
\Big(  \| \vec v\|_{L^p_\o C_\tau \H^s_x}^2
+ \| \vec w\|_{L^p_\o C_\tau \H^s_x}^2
\Big)\|  \vec v  - \vec w\|_{L^p_\o C_\tau \H^s_x}\\
& \quad \quad + 
\tau K_p 
\Big( 1  + \| \vec v\|_{L^p_\o C_\tau \H^s_x}
+ \| \vec w\|_{L^p_\o C_\tau \H^s_x}
\Big)
 \|  \vec v  - \vec w\|_{L^p_\o C_\tau \H^s_x}
\end{align*}

\noi
\noi
for any $\vec v, \vec w \in L^p(\O; C([0, \tau]; \H^s(\T^2)))$.
Therefore, 
by choosing
\[
\tau = \tau \big(\| (v^{(0)}, v^{(1)}) \|_{L^p_\o \H^s_x}, 
K_p\big) >0\] sufficiently small, 
we conclude that $\vec \G$ is a contraction in the 
(closed) ball 
\[B_R \subset L^p(\O; C([0, \tau]; \H^s(\T^2)))\] of radius
$R = 2C_1
\| (v^{(0)}, v^{(1)}) \|_{L^p_\o \H^s_x} + 1$.
At this point, the uniqueness holds only in the ball $B_R$
but by a standard continuity argument, 
we can extend the uniqueness to hold
in the entire class $L^p(\O; C([0, \tau]; \H^s(\T^2)))$.
We omit details. 
This concludes the proof of Proposition \ref{PROP:LWP2}.
\end{proof}

\subsection{Global well-posedness} 
\label{SUBSEC:4.2}

In this subsection, 
we prove global well-posedness of the mean-field SdNLW
\eqref{MF2} (Theorem \ref{THM:GWP2}\,(ii)).
As in Subsection \ref{SUBSEC:3.3}, 
our main strategy  is to adapt  the 
hybrid $I$-method
introduced in \cite{GKOT}
to the current mean-field setting.
We note that, while the analysis in \cite{GKOT}
and Section \ref{SEC:3}
is pathwise, 
our analysis presented below
is carried out under the $L^2(\O)$-norm (i.e.~not pathwise).

Fix a target time $T \gg1$.
In view of 
the blowup alternative \eqref{blow2} (with $p = 2$)
and 
\eqref{I3}, 
existence of the unique solution $(v, \dt v)$ to  \eqref{WN1} on 
the time interval $[0, T]$
follows once we show the following bound:
\begin{align*}
\sup_{0 \le t \le T}
\| (I v , \dt I v ) (t)\|_{ L^2_\o \H^1_x} \le C(T) 
\end{align*}

\noi
for some  finite constant $C(T) > 0$, 
where $(I v , \dt I v )$ is a solution to the following
$I$-mean-field SdNLW:
\begin{align} 
\begin{split}
& (\dt^2 + \dt + m - \Delta) I v \\
& \quad = 
 - I\big( \E [ v^2 ] v\big)
 - 2 I \big(\E [ \Psi v ] v\big) -I \big( \E [ v^2 ] \Psi\big)
- 2 I \big(\E [ v \Psi  ] \Psi\big).
\end{split}
\label{WN1a} 
\end{align}

\noi
For this purpose, define
the modified energy 
$\EE (I v, \dt I v) $:
\begin{align}
\begin{split}
\EE (I v, \dt I v) 
& = 
\frac 12 \E \bigg[ \int_{\T^2} 
|\nb I v|^2  + m (Iv)^2 +  (\dt Iv)^2 d x \bigg] + \frac 14 \int_{\T^2} \big(\E [(Iv)^2]\big)^2 d x\\
& = 
\frac 12  \int_{\T^2} 
\|\nb I v\|_{L^2_\o}^2  + m \|Iv\|_{L^2_\o}^2 +  \|\dt Iv\|_{L^2_\o}^2 d x  + \frac 14 \int_{\T^2} 
\|Iv\|_{L^2_\o}^4 d x.
\end{split}
\label{WW1}
\end{align}

\noi
where 
$\EE(v, \dt v)$ is as in  \eqref{MFE1}.
For notational simplicity, we set 
\begin{align*}
\EE(t) = \EE (I v, \dt I v) (t).
\end{align*}

\noi
As in Subsection \ref{SUBSEC:3.3}, 
Theorem \ref{THM:GWP2}\,(ii) on global well-posedness
of the mean-field SdNLW \eqref{MF2} follows
from the following proposition;
compare \eqref{WW2a} with \eqref{EN10}, 
 where the last terms are slightly different.

\begin{proposition}\label{PROP:CM0}
\textup{(i)}
Let $\frac 23 \le s < 1$.
Then, given small $\dl, \g > 0$ and $T \gg1$, 
we have 
\begin{align} 
\begin{split}
 \EE(t_2)  - \EE(t_1) 
&  \les M^{2 - 3s} \int_{t_1}^{t_2} \EE^2 (t) d t \\ 
&\quad 
+  M^{- s + \frac 12 + \dl} 
  \int_{t_1}^{t_2} 
\EE ^{\frac 32} (t) d t
\cdot \| \Psi \|_{L_\om^2 C_T W_x^{- \sigma_0, p_0}} \\
& \quad
+  M^{- \frac {s}{2} + \dl} 
\int_{t_1}^{t_2} \EE(t) d t 
\cdot \| \Psi' \Psi \|_{L_\om^2 C_T W_x^{- \sigma_0, p_0}} \\
&\quad 
+ M^{\g} \int_{t_1}^{t_2} \EE^{\frac 34}(t) d t
\cdot 
\| \Psi' \Psi  \|_{L_\om^2 C_T W_x^{- \g, \infty}} 
\\
&\quad 
+ \int_{t_1}^{t_2} \EE (t) dt \cdot \| I \Psi \|_{L_{[t_1, t_2], x}^\infty L_\o^2}
\end{split}
\label{WW2a}
\end{align}

\noi 
for any $0\leq t_1 \leq t_2 \leq T$, 
uniformly in  $M, T\gg1$, 
where
$\sigma_0 = \sigma_0 (\dl) >0$ and $p_0 = p_0 (\dl) \gg 1$ are as  in Lemma~\ref{LEM:com2}.
Here, $\Psi'$ denotes an independent copy of $\Psi$.

\medskip

\noi
\textup{(ii)}
Let $ \frac 45 < s < 1$.
Then, given $(v^{(0)}, v^{(1)}) \in L^2(\O; \H^s(\T^2))$ 
 and $T > 0$, 
there exists a unique global solution $(v, \dt v) 
\in L^2 (\O; C([0, T]; \H^s (\T^2)))$
to 
\eqref{WN1} with $(v, \dt v) |_{t = 0} = (v^{(0)}, v^{(1)})$.

\end{proposition}

\begin{proof} 
(i)
Let  $t_2\ge t_1 \ge 0$.
Then, from \eqref{WW1} and \eqref{WN1a}, 
we have
\begin{align*}
\EE (t_2) - \EE (t_1) 
& = \sum_{\l = 1}^3 \B_\l
- 
\E \bigg[ \int_{t_1}^{t_2} \int_{\T^2} (\dt I v)^2 d x d t \bigg] \\ 
& \le  \sum_{\l = 1}^3 \B_\l, 
\end{align*}

\noi
where $\B_\l$, $\l = 1, \dots, 3$, 
are given by 
\begin{align}
\begin{split}
\B_1 
& = \E \bigg[ \int_{t_1}^{t_2} \int_{\T^2} (\dt I v) \Big( \E [ (I v)^2 ] I v - I \big( \E [v^2] v \big) \Big) d x d t \bigg] , \\
\B_2 
& = 
 - \E \bigg[ \int_{t_1}^{t_2} \int_{\T^2} (\dt I v) \Big( 
2  I \big( \E [\Psi v] v \big)
+  I \big( \E [v^2] \Psi \big) \Big) d x d t \bigg],  \\
\B_3 & = - 2 \E \bigg[ \int_{t_1}^{t_2} \int_{\T^2} (\dt I v) I \big( \E [v \Psi ] \Psi \big) d x d t \bigg].
\end{split} 
\label{WW4}
\end{align}

\noi
Then, 
the energy growth bound \eqref{WW2a}
follows from 
Lemmas \ref{LEM:CM1}, \ref{LEM:CM2}, 
\ref{LEM:CM3}, and \ref{LEM:CM4}
presented below.

\medskip

\noi
(ii)
By applying \eqref{mom3a} (for $\Psi$), 
we can crudely bound
the last term on the right-hand side of~\eqref{WW2a}
by 
\begin{align}
 \les \log M  \int_{t_1}^{t_2} \EE (t) dt,  
 \label{WW4a}
\end{align}

\noi
uniformly in $M \gg 1$, 
which is 
controlled by the analogue of \cite[(3.31)]{GKOT}
in the current mean-field setting
(namely, with the additional $L^2_\o$-norm in the modified energy; see \eqref{WW1}).
Then, 
by replacing \cite[(3.29) and (3.31)]{GKOT}
with \eqref{WW2a} and~\eqref{WW4a}, 
we can proceed as in the proof of \cite[Proposition 3.5]{GKOT}
and apply the double Gronwall argument, 
 leading to
an analogue of 
Lemma~\ref{LEM:GWP1a}
in the current mean-field setting.
We note that 
 the additional $L^2_\o$-norm in the modified energy $\EE(t)$
(as compared to that in \cite{GKOT})
 does not play any role here.
See also \cite[Lemma~5.6]{CLL} (with $\eps = 1$) for 
the details of such an argument in the mean-field setting.
This  in turn allows us to proceed
with an iterative argument as in the 
proof of Proposition~\ref{PROP:GWP1}
presented in Subsection~\ref{SUBSEC:3.3}.
We omit details. 
\end{proof}

The remaining part of this section 
is devoted
 to estimating 
 $\B_\l$, $\l = 1, \dots, 3$, 
 which 
closely 
follows the argument presented in 
Subsection \ref{SUBSEC:3.2}
but
with 
modification coming from  the viewpoint~\eqref{WN3};
more precisely, as in \eqref{WN3}, 
we regard the outer expectation 
as an integration 
 over $\o \in \O$, 
while
the inner expectation 
is over $\o' \in \O$.
Moreover, 
as in 
Subsection \ref{SUBSEC:3.2}, 
we often suppress
$t$-dependence for notational simplicity
in the following.

The following lemma establishes 
the  commutator estimate
on the cubic term in $v$.

\begin{lemma}
\label{LEM:CM1}
Let $\frac 23 \le s < 1$. Then, we have 
\begin{align} 
    | \B_1 | \les M^{2 - 3s} \int_{t_1}^{t_2} \EE^2 (t) d t
\label{WW5}
\end{align} 

\noi
for any $t_2 \ge t_1 \ge 0$, 
where the implicit constant is independent of $M \in \N$.

\end{lemma}

\begin{proof}
By Cauchy-Schwarz's and Minkowski's integral inequalities, we have
\begin{align}
\begin{split}
& \Big\|(\dt I v) \Big( \E [ (I v)^2 ] I v - I \big( \E [v^2] v \big)\Big) \Big\|_{L^1_{\o, x}}\\
& \quad 
\le 
\|\dt I v \|_{L^2_{\o, x}}
\big\| \E [ (I v)^2 ] I v - I \big( \E [v^2] v \big)\big\|_{L^2_{\o, x}}\\
& \quad 
\les \EE^\frac 12
\Big\| \|(I v(\o'))^2 I v(\o) - I \big( v^2(\o')  v(\o) \big) \|_{L^2_x}\Big\|_{L^2_{\o} L^1_{\o'}}.
\end{split}
\label{WW5a}
\end{align}

\noi
Then, the bound \eqref{WW5} 
follows
from \eqref{WW5a},
Lemma \ref{LEM:com1}, and \eqref{WW1}.
\end{proof}

By formally distributing the $I$-operator to each factor of $\B_2$ 
and $\B_3$ 
in \eqref{WW4}, 
we obtain the following term:
\begin{align} 
\begin{split}
\B_2' &  = -   \E \bigg[ \int_{t_1}^{t_2} \int_{\T^2} (\dt I v)
 \Big(  2 \E [ (I \Psi)( I v) ] I v  + \E [ (I v)^2] I \Psi  \Big) d x d t \bigg], \\
\B_3'  & =  - 2 \E_\o \bigg[ \int_{t_1}^{t_2} \int_{\T^2} (\dt I v(\o))\,
 \E_{\o'} \Big[ (I v'(\o')) \big(I  (\Psi' (\o') \Psi(\o)) \big)\Big] d x d t \bigg], 
\end{split} 
\label{WW6}
\end{align}  

\noi
\noi
where  $(\Psi', v')$ is an independent copy of $(\Psi, v)$
as in \eqref{WN3}.
We first establish  the following commutator estimates.

\begin{lemma}
\label{LEM:CM2}
For $\l = 2, 3$, 
let $\B_\l$ and $\B_\l'$ be as in \eqref{WW4} and \eqref{WW6}, respectively.
Let $\frac 23 \le s < 1$.
 Then, given $T \gg 1$ and small $\dl > 0$, we have
\begin{align} 
| \B_2 - \B_2' | 
& \les
 M^{- s + \frac 12 + \dl} 
  \int_{t_1}^{t_2} 
\EE ^{\frac 32} (t) d t
\cdot \| \Psi \|_{L_\om^2 C_T W_x^{- \sigma_0, p_0}} , 
\label{WW7}\\
| \B_3 - \B_3' | 
& \les M^{- \frac {s}{2} + \dl} 
\int_{t_1}^{t_2} \EE(t) d t 
\cdot \| \Psi' \Psi \|_{L_\om^2 C_T W_x^{- \sigma_0, p_0}} 
\label{WW7a}
\end{align}

\noi
for any 
$0 \le t_1 \le t_2 \le T$, 
uniformly in 
  $M,  T \gg 1$, 
where
$\sigma_0 = \sigma_0 (\dl) >0$ and $p_0 = p_0 (\dl) \gg 1$ are as  in Lemma~\ref{LEM:com2}.
Here, $\Psi'$ denotes an independent copy of $\Psi$.

\end{lemma}

\begin{proof} 
Proceeding as in \eqref{WW6}, we have 
\begin{align}
\begin{split}
& \big\|\E [ (I \Psi)( I v) ] I v
-  I \big( \E [\Psi v] v \big)\big\|_{L^2_{\o, x}}\\
& \quad
\le \Big\|
\|
 (I \Psi (\o'))( I v(\o')) ( I v(\o))
-  I \big( \Psi(\o')  v(\o')  v(\o) \big) \|_{L^2_x}\Big\|_{L^2_{\o} L^1_{\o'}}.
\end{split}
\label{WW8}
\end{align}

\noi
Then, 
the  bound \eqref{WW7} follows 
from
Cauchy-Schwarz's inequality (in $\o$ and $x$), 
\eqref{WW8}, 
and Lemma~\ref{LEM:com2}
with \eqref{WW1}.
The second bound \eqref{WW7a}
follows in a similar manner.
\end{proof}

We conclude this section by estimating 
$\B_3'$
and $\B_2'$ in \eqref{WW6}.

\begin{lemma}
\label{LEM:CM3}
Let  $\B_3'$ be as in \eqref{WW6}.
Let $\frac 23 \le s < 1$.
 Then, given small $\g > 0$, we have 
\begin{align} 
| \B_3' | & \les M^{\g} \int_{t_1}^{t_2} \EE^{\frac 34}(t) d t
\cdot 
\| \Psi' \Psi  \|_{L_\om^2 C_T W_x^{- \g, \infty}} 
\label{WW9}
\end{align}

\noi
for any 
$0 \le t_1 \le t_2 \le T$, 
uniformly in 
  $M, T \gg 1$.
Here, $\Psi'$ denotes an independent copy of $\Psi$.

\end{lemma}

\begin{proof} 

From the independence of $\Psi$ and $\Psi'$, we have 
\begin{align} 
\begin{split}
& \big\| (\dt I v(\o))
 (I v'(\o')) \big(I  (\Psi' (\o') \Psi(\o)) \big)\big\|_{L^1_{\o, \o'}} \\
& \quad 
\le  
\|\dt I v\|_{L^2_\o}
\|  I v\|_{L^2_\o}
 \big\| I  (\Psi' (\o') \Psi(\o)) \big\|_{L^2_{\o, \o'}} \\
& \quad 
= \|\dt I v\|_{L^2_\o}
 \| I v\|_{L^2_\o}
\| I  (\Psi' \Psi )\|_{L^2_{\o}}. 
\end{split}
\label{WW10}
\end{align}

\noi
Then, the bound \eqref{WW9}
follows
from proceeding as in the proof of \eqref{EN7}
in Lemma \ref{LEM:com5}
with~\eqref{WW10}.
\end{proof}

\begin{lemma}
\label{LEM:CM4}
Let $\B_2'$ be as in \eqref{WW6}.
Then, given $s < 1$, 
there exists $c > 0$ such that 
\begin{align} 
| \B_2' | \les 
\int_{t_1}^{t_2} 
 \EE (t) d t \cdot
\| I \Psi \|_{ L^{\infty}_{[t_1, t_2], x} L^2_\o } 
\label{WW11r}
\end{align} 

\noi 
for any 
$t_2 \geq t_1 \geq 0$,  
uniformly in 
$M \in \N$.

\end{lemma}

\begin{proof} 
From \eqref{WW1},  we have 
 \begin{align}
\begin{split}
\|I v\|_{L^{4}_xL^2_\o }
 \les 
\EE^{\frac{1}{4}} .
\end{split}
\label{WW12r}
 \end{align}

\noi
Moreover, by H\"older's inequality, we have 
\begin{align}
\begin{split}
& \Big\| (\dt I v (\o))
 \Big(  2  (I \Psi (\o') )( I v(\o')) ( I v(\o))  +  (I v(\o'))^2 (I \Psi (\o)) \Big) 
 \Big\|_{L^1_{\o, \o'}L^1_x}\\
& \quad \les \|\dt I v \|_{L^2_x L^2_\o}
\| I v \|_{L^4_x L^2_\o}^2
\| I \Psi \|_{L^{\infty}_xL^2_\o}.
\end{split}
\label{WW13r}
\end{align}

\noi
Then, the bound \eqref{WW11r} follows from \eqref{WW12r} and \eqref{WW13r}.
\end{proof}

\section{Mean-field limit of 
the hyperbolic $O (N)$ linear sigma model}

\label{SEC:conv}

In this section, we present a proof of  Theorem~\ref{THM:conv1}
on convergence of \HLS 
 \eqref{NLW4}
to the mean-field SdNLW \eqref{MF2}. 
In Subsection~\ref{SUBSEC:5.1}, 
we present several law-of-large-numbers type lemmas.
In Subsection~\ref{SUBSEC:5.2}, 
we then prove Theorem~\ref{THM:conv1}\,(i). 
Lastly, 
under a higher moment assumption, 
we establish 
a convergence rate
 of~\eqref{NLW4}
 to the mean-field limit
 (Theorem~\ref{THM:conv1}\,(ii)) 
in Subsection~\ref{SUBSEC:5.3}.

\subsection{Law of large numbers} 
\label{SUBSEC:5.1}

In this subsection, we present three 
 law-of-large-numbers type lemmas.
While the first lemma (Lemma \ref{LEM:LLN1})
provides a convergence rate of the purely stochastic terms,
the second and third
lemmas (Lemmas \ref{LEM:LLN2} and \ref{LEM:LLN3})
do not provide any rate
under the second moment assumption.
Under a higher moment assumption, 
they can be upgraded
to versions with a decay rate;
see Lemmas~\ref{LEM:LLN4} and~\ref{LEM:LLN5}.

\begin{lemma}
\label{LEM:LLN1}

Given $\l = 2, 3$ and $j, k \in \N$, 
let $Z^{(\l)}_{k, j}$
be as in~\eqref{X2} or~\eqref{X3}.
Namely, 
\begin{align}
  Z_{k, j}^{(2)} = \, \wick{\Psi_k \Psi_j} \,  
\qquad \text{and}\qquad 
Z_{k, j}^{(3)}
=  \, \wick{\Psi_k^2 \Psi_j} 
\label{LN1}
\end{align}

\noi
or 
\begin{align}
   Z_{k, j}^{(2)} = \, \wick{\Phi_k \Phi_j} \, 
\qquad \text{and}\qquad 
Z_{k, j}^{(3)}
=  \, \wick{\Phi_k^2 \Phi_j}\,.
\label{LN2}
\end{align}

\noi
Then, 
given $\eps > 0$, we have 
\begin{align}
 \bigg\| \frac{1}{N} \sum_{k = 1}^N 
Z_{k, k}^{(2)}\, \bigg\|_{L^p_\o L_T^2 W_x^{- \eps, \infty}}
& \les_T \frac{p}{N^{\frac 12}} , 
\label{LN3} \\
\sup_{j \in \N}
 \bigg\| \frac 1N \sum_{k = 1}^N 
 Z_{k, j}^{(3)}\, \bigg\|_{L^p_\o L_T^2 W_x^{- \eps, \infty}}
& \les_T \frac{p^{\frac 32} }{N^{\frac 12}} , 
\label{LN4} \\
 \bigg\| \frac 1N \sum_{k = 1}^N 
 Z_{k, j}^{(3)}\, \bigg\|_{L^p_\o \A_{N, j}L_T^2 W_x^{- \eps, \infty}}
& \les_T \frac{p^{\frac 32} }{N^{\frac 12}} , 
\label{LN5}
\end{align}

\noi
uniformly in finite $p \ge 1$ and $N \in \N$, 
where $\A_N B$ is as in \eqref{AN0}.
As a consequence,
for each $T > 0$, we have 
\begin{align}
\begin{split}
 \bigg\| \frac{1}{N} \sum_{k = 1}^N 
Z_{k, k}^{(2)}\, \bigg\|_{ L_T^2 W_x^{- \eps, \infty}}
& \too 0 , \\
\text{for each fixed $j \in \N:$}
\quad 
  \bigg\| \frac 1N \sum_{k = 1}^N 
 Z_{k, j}^{(3)}\, \bigg\|_{ L_T^2 W_x^{- \eps, \infty}}
& \too 0 , \\
  \bigg\| \frac 1N \sum_{k = 1}^N 
 Z_{k, j}^{(3)}\, \bigg\|_{ \A_{N, j}L_T^2 W_x^{- \eps, \infty}}
& \too 0
\end{split}
\label{LN5a}
\end{align}

\noi
in probability
at the rate $N^{-\frac 12 }$
as $N \to \infty$
in the sense of Definition \ref{DEF:conv1}.

\end{lemma}

\begin{proof}
We first note that, 
as explained in  Remark \ref{REM:conv2}\,(ii), 
 the convergence \eqref{LN5a} in probability
at the rate $N^{-\frac 12}$
follows from \eqref{LN3}, \eqref{LN4}, and \eqref{LN5}
with  Chebyshev's inequality.
Hence, 
 we focus on proving \eqref{LN3}, \eqref{LN4}, and \eqref{LN5}
 in the following.

We assume 
\eqref{LN1} since the same proof holds under~\eqref{LN2}.
We first look at \eqref{LN3}. 
By Sobolev's embedding (with finite $r \gg1$ such that $\eps r > 4$),
 Minkowski's integral inequality
 (for finite $p \ge r \ge 2$),  the Wiener chaos estimate (Lemma \ref{LEM:hyp}), 
the independence
of $\{Z_{k, k}^{(2)}\}_{k = 1}^N$, 
and  Lemma~\ref{LEM:sto1}\,(i), we have
\begin{align}
\begin{split}
& \bigg\| \frac{1}{N} \sum_{k = 1}^N 
 Z_{k, k}^{(2)}\, \bigg\|_{L^p_\o L_T^2 W_x^{- \eps, \infty}} 
\les  \bigg\| \frac{1}{N} \sum_{k = 1}^N \jb{\nb}^{- \frac{\eps}{2}} 
Z_{k, k}^{(2)}\, 
\bigg\|_{L^p_\o L_T^2 L_x^r}\\
&\quad \leq  \bigg\| \frac{1}{N} \sum_{k = 1}^N \jb{\nb}^{- \frac{\eps}{2}} 
Z_{k, k}^{(2)}(t, x) \bigg\|_{L_T^2 L_x^r L^p_\o} 
\le p
 \bigg\| \frac{1}{N} \sum_{k = 1}^N \jb{\nb}^{- \frac{\eps}{2}} 
Z_{k, k}^{(2)}(t, x) \bigg\|_{L_T^2 L_x^r L^2_\o} \\
& \quad =  
\frac{p}{N^\frac 12 } 
\|  \jb{\nb}^{- \frac{\eps}{2}} 
Z_{k, k}^{(2)}(t, x) \|_{L_T^2 L_x^r \A_N L^2_\o} 
\les_T   
\frac{p }{N^\frac 12 } 
\|  \jb{\nb}^{- \frac{\eps}{2}} 
Z_{k, k}^{(2)}(t, x) \|_{\A_N L^2_\o C_T L_x^r } \\
&\quad \les_T \frac{p}{N^{\frac 12}},
\end{split}
\label{LN6}
\end{align}

\noi
yielding \eqref{LN3}.

We note that \eqref{LN5} follows from \eqref{LN4}
and the definition \eqref{AN0} of the $\A_NB$-norm.
Thus, it suffices  to prove \eqref{LN4}.
Note from 
\eqref{psi3} that 
\begin{align*}
\E\big[ 
\wick{\Psi_k^2 \Psi_j} \, \wick{\Psi_{k'}^2 \Psi_{j}}
\big] = 0
\end{align*}

\noi
unless $k = k'$.
Then, the bound \eqref{LN4} follows
from proceeding 
 as in \eqref{LN6}.
\end{proof}

In the following two lemmas, 
we assume that 
$v_k \in L^{2} (\Omega; C([0, T]; H^{s} (\T^2)))$, $k \in \N$,
which is the only integrability
guaranteed by Theorem \ref{THM:GWP2}
for global-in-time solutions to the mean-field SdNLW~\eqref{MF2}.
See Subsection \ref{SUBSEC:5.3}
for a discussion under
a higher moment assumption.

\begin{lemma}
\label{LEM:LLN2}

Let $\frac 12 \le s < 1$.
Given $k \in \N$, 
let  $\Psi_k$ be the stochastic convolution defined in~\eqref{psi2} 
and let $v_k \in L^{2} (\Omega; C([0, T]; H^{s} (\T^2)))$ be a solution of the mean-field SdNLW \eqref{NLW5}.
Moreover, we assume that $\big(v_k(0), \dt v_k(0)\big)
\in L^{2} (\Omega; \H^s (\T^2))$, $k \in \N$, 
are independent and identically distributed over $k \in \N$.
Then, given $T > 0$
and small $\eps > 0$, 
we have 
\begin{align}
\bigg\| \frac{1}{N} \sum_{k = 1}^N 
\Big( \Psi_k v_k - \E [\Psi_k v_k] \Big) \bigg\|_{L_T^2
W^{-\eps, \frac 4{2- s - \eps}}_x} 
& \longrightarrow 0, 
\label{LNy2} \\
\bigg\| \frac{1}{N} \sum_{k = 1}^N
\Big( v_k^2 - \E [v_k^2] \Big) \bigg\|_{L_T^2 (L^\frac 3{2-s}_x
\cap W_x^{\eps, \frac 2{2-s-\eps}})
} & \longrightarrow 0, 
\label{LNy3} 
\end{align}

\noi
almost surely, as
 $N \to \infty$.

\end{lemma}

\begin{proof} 
We first note that 
$(\Psi_k, v_k)$, $k \in \N$,  are independent 
and identically distributed
since
$\Psi_k$ are independent
 and identically distributed 
and so are
the initial data
$(v_k(0), \dt v_k(0))$.
In particular, we have
\begin{align}
R\deff 
\sup_{k \in \N} 
\| v_k \|_{L^2_\o C_T H^s_x}
= 
\| v_1 \|_{L^2_\o C_T H^s_x}
 < \infty.
\label{LNy1}
\end{align}

Proceeding as in \eqref{WN6} with 
  Lemma~\ref{LEM:prod}\,(ii), 
Cauchy-Schwarz's inequality,  Lemma~\ref{LEM:sto1}\,(i), 
Sobolev's inequality, 
and~\eqref{LNy1}, we have
\begin{align*}
 \E \Big[ \| \Psi_k v_k \|_{L_T^2 W^{-\eps, \frac 4{2- s - \eps}}_x}  \Big] 
 &\les T^\frac 12  \| \Psi_k \|_{L^2_\o C_T W_x^{- \eps, \infty}}
  \| v_k \|_{L^2_\o C_T  W_x^{\eps, \frac4 {2-s-\eps}}}\\
&\les T^\frac 12 
  \| v_k \|_{L^2_\o C_T  H^s_x}\\
&< \infty, 
\end{align*}

\noi
uniformly in $k \in \N$.
Hence, in view of the  independence of $\{\Psi_k v_k\}_{k \in \N}$, 
the almost sure convergence
\eqref{LNy2} follows from the strong law of large numbers 
in the Banach space setting
(\cite[Corollary~7.10 on p.\,189]{LT}).

Next, we consider  \eqref{LNy3}. 
Proceeding as in \eqref{WN5} and \eqref{WN7} with 
the fractional Leibniz rule (Lemma~\ref{LEM:prod}\,(i)), 
Sobolev's inequality, Cauchy-Schwarz's inequality, 
and  \eqref{LNy1}, we have
\begin{align}
\E \Big[ \| v_k^2 \|_{L_T^2 (L^\frac 3{2-s}_x
\cap W_x^{\eps, \frac 2{2-s-\eps}})} \Big] 
 \les T^\frac 12 \| v_k \|^2_{L^2_\o C_T H_x^{s}} < \infty, 
\label{LNy4}
\end{align}

\noi
uniformly in $k \in \N$.
Hence, in view of the  independence of $\{v_k^2\}_{k \in \N}$, 
the almost sure convergence~\eqref{LNy3} follows from the strong law of large numbers (\cite[Corollary~7.10 on p.\,189]{LT}).
\end{proof}

\begin{remark}\label{REM:high0}
\rm 

From \eqref{LNy4}, we see that 
the $L^2(\O)$-integrability assumption
for $v_k$
is indeed necessary 
for the (almost sure) convergence \eqref{LNy3};
see \cite[Corollary~7.10 on p.\,189]{LT}.

\smallskip

\end{remark}

\begin{lemma}
\label{LEM:LLN3}

Let $\frac 12  \le s < 1$
and $\{(\Psi_j, v_j)\}_{j \in \N}$ be as in Lemma \ref{LEM:LLN2}.
Then, given $T  > 0$, 
we have, for each $j \in \N$, 
\begin{align}
\bigg\| \frac 1N \sum_{k = 1}^N \Big( v_k\! \wick{\Psi_k \Psi_j} 
- \E [ v_k\Psi_k ] \Psi_j \Big) \bigg\|_{L_T^2 H_x^{s-1}} \longrightarrow 0
\label{LNx1} 
\end{align}

\noi
in probability as $N \to \infty$, 
and
\begin{align}
&\bigg\| \frac 1N \sum_{k = 1}^N \Big( v_k \!\wick{\Psi_k \Psi_j} - \E [v_k\Psi_k ] \Psi_j \Big) 
\bigg\|_{\A_{N, j}L_T^2 H_x^{s-1}} \longrightarrow 0
\label{LNx1a}
\end{align}

\noi
in probability as $N \to \infty$.

\end{lemma}

\begin{proof}
We only prove \eqref{LNx1a} in the following, since  \eqref{LNx1} 
follows in a similar manner.

We first study the contribution to \eqref{LNx1a} from  the case $k = j$.
Let $1 < p < 2$.
Then, 
proceeding as in  \eqref{LWP5}
with H\"older's inequality, 
 Lemma~\ref{LEM:sto1}\,(i), and \eqref{LNy1}, we have 
\begin{align}
\begin{split}
\Big\| \| v_j \! \wick{\Psi_j^2} \|_{L_T^2 H_x^{s-1}} \Big\|_{L^p_\o}
&\les T^\frac 12  \| v_j \|_{L^2_\o C_T  H_x^{s}}
 \| \wick{\Psi_j^2} \|_{L^\frac{2p}{2-p}_\o C_T  W_x^{- \eps, \infty}}\\
 & \les_p 
  T^\frac 12, 
\end{split}
\label{LNx2}
\end{align}

\noi
uniformly in $j \in \N$, 
provided that $0 < \eps \le \min(1-s,  s)$.
Let $(\Psi', v')$ be an independent copy of $(\Psi_j, v_j)$.
Then, given $1 < p < 2$, 
it follows from 
Minkowski's integral inequality, 
the independence of $(\Psi', v')$ from $\Psi_j$, 
and proceeding as in \eqref{LNx2} 
with \eqref{WN3a} that 
\begin{align}
\begin{split}
 \Big\| \| \E [ v_j \Psi_j ] \Psi_j \|_{L_T^2 H_x^{s-1}} \Big\|_{L^p_\o}
& = \Big\| \| \E_{\o'} [ v' (\o')  \Psi' (\o')] \Psi_j  (\o)\|_{L_T^2 H_x^{s-1}}\Big\|_{L^p_\o}\\
& \le
 \| v'  \Psi'   \Psi_j  \|_{L^p_\o L_T^2 H_x^{s-1}} \\
&\les T^\frac 12  \| v_j \|_{L^2_\o C_T  H_x^s}
 \|\Psi' \Psi_j  \|_{L^\frac {2p}{2-p}_\o C_T  W_x^{- \eps, \infty}}\\
 & \les 
  T^\frac 12, 
\end{split}
\label{LNx2a}
\end{align}

\noi
uniformly in $j \in \N$.
Hence, 
given $L \gg1$, it follows
from Chebyshev's inequality with \eqref{LNx2} and~\eqref{LNx2a}
that 
\begin{align}
 \sum_{N = 1}^\infty \PP 
\Bigg(\bigg\| \frac 1N  \Big(v_j \wick{\Psi_j^2} - \E [v_j  \Psi_j ] \Psi_j\Big)
\bigg\|_{\A_{N, j}L_T^2 H_x^{s-1}} > \frac 1 L \Bigg)
\les_T \sum_{N = 1}^\infty \frac {L^p}{N^p}
 < \infty, 
\label{LNx2b}
\end{align}

\noi
since $p > 1$.
Therefore, by the Borel-Cantelli lemma, we conclude that 
\begin{align*}
\bigg\| \frac 1N  \Big(v_j \wick{\Psi_j^2} - \E [ v_j \Psi_j] \Psi_j\Big)
\bigg\|_{\A_{N, j}L_T^2 H_x^{s-1}} 
\too 0, 
\end{align*}

\noi
almost surely (and thus in probability) as $N \to \infty$, 
which shows \eqref{LNx1a} 
for the case $k =  j$.

Next, we study the contribution to \eqref{LNx1a} from the case $k \ne j$.
By taking the Fourier transform in $x$ (suppressing the $t$-dependence), we have
\begin{align*}
& \F_x \Big(\big( v_k \Psi_k  - \E [v_k\Psi_k ] \big) \Psi_j\Big)(n)\\
& \quad = \sum_{n = n_1 + n_2 + n_3}
\Big(\ft v_k(n_1)
\ft \Psi_k(n_2)
- \E[ \ft v_k(n_1)
\ft \Psi_k(n_2)]
\Big) \ft \Psi_j(n_3).
\end{align*}

\noi
Then, given $M \gg 1$ (to be chosen later), we write
\begin{align}
\frac 1N \sum_{\substack{k = 1\\k \ne j}}^N \big( v_k \Psi_k  - \E [v_k\Psi_k ] \big) \Psi_j
= \1_j + \II_j
\label{LZ2}
\end{align}

\noi
where 
$\1_j$ denotes the contribution
from the case 
\begin{align}
\max(|n|, |n_1|,  |n_2|,  |n_3|) > M.
\label{LZ3}
\end{align}

\noi
Namely,  $\II_j$ denotes the contribution from the case:
$\max(|n|, |n_1|,  |n_2|,  |n_3|) \le M$, given by 
\begin{align}
\II_j = 
\frac 1N \sum_{\substack{k = 1\\k \ne j}}^N 
\P_M \Big(
\big( \P_M v_k \Psi_{k, M}  - \E [\P_M v_k\Psi_{k, M} ] \big) \Psi_{j, M}\Big), 
\label{LZ4}
\end{align}

\noi
where 
$\Psi_{j, M} = \P_M \Psi_j$.

We first treat the first term $\1_j$.
From a slight modification of \eqref{LWP5}, we have
\begin{align}
\begin{split}
\|v_k \Psi_k \Psi_j\|_{L^2_T H^{s-1+ \eps }_x}
& \les  T^\frac 12
\| v_k \Psi_k \Psi_j\|_{C_T  W_x^{-2\eps, \frac 2{2-s-3\eps}}}\\
& \les  T^\frac 12
\|v_k\|_{C_T
 W_x^{2\eps, \frac 2{2-s-3\eps}}}
\| \Psi_k \Psi_j \|_{C_T W^{-2\eps, \infty}_x}\\
& \les  T^\frac 12
\|v_k\|_{C_T  H_x^{s-\eps}}
\| \Psi_k \Psi_j \|_{C_T W^{-2\eps, \infty}_x}.
\end{split}
\label{LZ7}
\end{align} 

\noi
Note that, thanks to \eqref{LZ3}, we can gain $M^{-\eps}$
by losing $\eps$-regularity.
Thus, from 
\eqref{LZ2}, \eqref{LZ3}, 
 Lemma~\ref{LEM:sto1}\,(ii)
(see also \eqref{mom2} in Remark \ref{REM:sto1a}\,(ii)), 
 \eqref{LZ7},
Cauchy-Schwarz's inequality, 
and~\eqref{LNy1}, 
we have
\begin{align*}
\E \Big[
\|\1_j \|_{\A_{N, j}L^2_T H^{s-1}_x}\Big]
& \les_T M^{-\eps} 
R, 
\end{align*}

\noi
uniformly in $N \in \N$.
Then, given small $\dl > 0$, 
it follows from Markov's inequality
that 
\begin{align}
\PP\Big(\|\1_j \|_{\A_{N, j}L^2_T H^{s-1}_x}
> \dl \Big)
\le
C_T \frac{ 
R}
{\dl M^{\eps} }
< \frac \dl 3,  
\label{LZ9}
\end{align}

\noi
uniformly in $N \in \N$, 
by taking 
$M = M(\dl, 
T, 
R)\gg1 $
sufficiently large.

Next, we consider the term $\II_j$ in \eqref{LZ4}, 
where all the frequencies are bounded by $M$.
We first write $\II_j$ as 
\begin{align}
\begin{split}
\II_j & = \II^{(1)}_j + \II^{(2)}_j\\
& \hspace{-0.3mm} \deff \frac 1N \sum_{k = 1}^N 
\P_M \Big(
\big( \P_M v_k \Psi_{k, M}  - \E [\P_M v_k\Psi_{k, M} ] \big) \Psi_{j, M}\Big)\\
& \quad \, - 
\frac 1N
\P_M \Big(
\big( \P_M v_j \Psi_{j, M}  - \E [\P_M v_j\Psi_{j, M} ] \big) \Psi_{j, M}\Big).
\end{split}
\label{LZZ1a}
\end{align}

Let us first consider $\II^{(2)}_j$.
By a slight modification of \eqref{LNx2}, 
for $1 < p < 2$, we have 
\begin{align}
\begin{split}
\Big\| \| \P_M v_j \Psi_{j, M}^2 \|_{L_T^2 H_x^{s-1}} \Big\|_{L^p_\o}
&\les T^\frac 12  M^{2\eps}\| v_j \|_{L^2_\o C_T  H_x^{s}}
 \|\Psi_j \|^2_{L^\frac{4p}{2-p}_\o C_T  W_x^{- \eps, \infty}}\\
 & \les_p 
  T^\frac 12 M^{2\eps}, 
\end{split}
\label{LZZ1b}
\end{align}

\noi
uniformly in $j \in \N$.
Then, from \eqref{LZZ1b} and \eqref{LNx2a}
(which implies
\eqref{LNx2b} with an extra factor $M^{2\eps p }$ on the right-hand side)
with   the Borel-Cantelli lemma, we conclude that
\begin{align*}
\big\| \II^{(2)}_j \big\|_{\A_{N, j}L_T^2 H_x^{s-1}} 
\too 0, 
\end{align*}

\noi
almost surely (and thus in probability) as $N \to \infty$.
Namely, given small $\dl > 0$, 
there exists $N_1  = N_1(\dl) \gg 1$ such that
\begin{align}
\PP\Big(\big\|\II^{(2)}_j \big\|_{\A_{N, j}L_T^2 H_x^{s-1}} > \dl  \Big) < \frac \dl 3
\label{LZZ1c}
\end{align}

\noi
for any $N \ge N_1$.

Next, we treat $\II^{(1)}_j$ in \eqref{LZZ1a}.
From~\eqref{LNy1}
and Lemma~\ref{LEM:sto1}\,(i), we have 
\begin{align*}
\E\Big[
\|\P_M v_k \Psi_{k, M} 
\|_{L^2_T L^2_x}\Big]
& \les
M^\eps T^\frac 12 
\| v_1 \|_{L^2_\o C_T L^2_x}
\|\Psi_1\|_{L^2_\o  C_T W^{-\eps, \infty}_x}\\
& < \infty, 
\end{align*}

\noi
uniformly in $k \in \N$.
Then, 
noting
that $\P_M v_k \Psi_{k, M}  - \E [\P_M v_k\Psi_{k, M}]$, $k \in \N$, 
are independent and identically distributed, 
it follows from \cite[Corollary 7.10 on p.\,189]{LT}
that 
\begin{align}
Y_{N, M} \deff
\bigg\|\frac 1N \sum_{k = 1}^N 
\big( \P_M v_k \Psi_{k, M}  - \E [\P_M v_k\Psi_{k, M} ] \big)
\bigg\|_{L^2_T L^2_x}
\too 0, 
\label{LZZ2}
\end{align}

\noi
almost surely, as $N \to \infty$, 
which in particular implies that the convergence \eqref{LZZ2}
also holds in probability.
Thus, given small $\dl > 0$ and $K \gg1$, there exists $N_2 = N_2 (\dl, M, K) \gg 1$ such that 
\begin{align}
\PP\Big(Y_{N, M} > \dl K^{-1} \Big) < \frac \dl 6
\label{LZZ2a}
\end{align}

\noi
for any $N \ge N_2$.
By a crude estimate, we have
\begin{align}
\big\|\II_j^{(1)} \big\|_{\A_{N, j}L^2_T H^{s-1}_x}
& \le C_0  
M^\eps
Y_{N, M} 
\|\Psi_{j}\|_{\A_{N, j}C_T W^{-\eps, \infty}_x}.
\label{LZZ3}
\end{align}

\noi
Given small $\dl > 0$, $T > 0$, and $M \gg 1$, 
it follows from 
 Lemma \ref{LEM:sto1}\,(ii)
that 
there exists $K = K(\dl, M, T) \gg1$ such that 
\begin{align}
\PP\Big(C_0  
M^\eps
\|\Psi_{j}\|_{\A_{N, j}C_T W^{-\eps, \infty}_x} > K\Big)
< \frac \dl 6, 
\label{LZZ3a}
\end{align}

\noi
uniformly in $N \in \N$.
Hence, given small $\dl > 0$, 
it follows from \eqref{LZZ3}
with \eqref{LZZ2a} and \eqref{LZZ3a}
that 
we have 
\begin{align}
\begin{split}
\PP & \Big( \big\|\II_j^{(1)} \big\|_{\A_{N, j}L^2_T H^{s-1}_x}> \dl \Big) \\
& \le 
\PP\Big(Y_{N, M} > \dl K^{-1} \Big)\\
& \quad  + 
\PP\Big(\{Y_{N, M} \le  \dl K^{-1}\}
\cap 
\{C_0  
M^\eps
\|\Psi_{j}\|_{\A_{N, j}C_T W^{-\eps, \infty}_x} > K\}\Big)\\
& < \frac \dl 6 + 
\PP\Big(C_0  
M^\eps
\|\Psi_{j}\|_{\A_{N, j}C_T W^{-\eps, \infty}_x} > K\Big)\\
& < \frac \dl 3
\end{split}
\label{LZZ4}
\end{align}

\noi
for any $N \ge N_2$.

Therefore,  from \eqref{LZ2}, \eqref{LZ9}, \eqref{LZZ1a}, \eqref{LZZ1c},  and \eqref{LZZ4}, 
we conclude
 \eqref{LNx1a} 
for the case $k \ne  j$.
\end{proof}

\subsection{Proof of Theorem~\ref{THM:conv1}\,(i)}
\label{SUBSEC:5.2}

In this subsection, 
we present a proof of Theorem~\ref{THM:conv1}\,(i).

Let $\frac 45 < s < 1$.
Let 
$\big\{ \big(u_{N, j}^{(0)}, u_{N, j}^{(1)}\big)\big\}_{j = 1}^N
\in \big(\H^s (\T^2)\big)^{\otimes N}$, $N \in \N$, 
and 
$\big\{\big(u_{j}^{(0)}, u_{ j}^{(1)}\big)\big\}_{j \in \N}
\in \big(\H^s (\T^2)\big)^{\otimes \N}$
be  sequences of random functions
as in
Theorem~\ref{THM:conv1}\,(i), 
where the former converges in probability to the latter in the sense
of \eqref{IV1} and \eqref{IV2}.
In particular, we assume that 
$\big(u_{j}^{(0)}, u_{ j}^{(1)}\big)
\in L^2(\O; \H^s (\T^2))$, $j \in \N$, 
 are independent and identically distributed.

Let 
$\{ (u_{N, j}, \dt  u_{N, j})\}_{j = 1}^N$
and $\{(u_j, \dt u_j)\}_{j \in \N}$
be the global-in-time solutions to \eqref{NLW4} 
and \eqref{MF2}
with initial data $\big\{ \big(u_{N, j}^{(0)}, u_{N, j}^{(1)}\big)\big\}_{j = 1}^N$
and $\big\{\big(u_{j}^{(0)}, u_{ j}^{(1)}\big)\big\}_{j \in \N}$
 constructed in Theorems \ref{THM:GWP1}
and \ref{THM:GWP2}, respectively.
Then, set 
$v_{N, j} = u_{N, j} - \Psi_{j}$
and 
$v_j = u_j - \Psi_j$
as in \eqref{exp}
and \eqref{expj}
such that 
$\{(v_{N, j}, \dt v_{N, j})\}_{j = 1}^N
\in \big(  C (\R_+; \H^s (\T^2) )\big)^{\otimes N}$
almost surely
and 
$(v_j, \dt v_j) \in 
L^2(\O;  C ([0, T]; \H^s (\T^2) ))$
for any $T > 0$.
Then, 
by writing \eqref{NLW3} and \eqref{NLW5} in the Duhamel formulation, we have 
\begin{align}
\begin{split}
v_{N, j} 
& = S(t) \big( u_{N, j}^{(0)}, u_{N, j}^{(1)} \big) 
- \frac{1}{N} \sum_{k = 1}^N  \I\Big( 
v_{N, k}^2 v_{N, j}
+ 2 \Psi_k v_{N, k} v_{N, j}
  + v_{N, k}^2 \Psi_j \\
& 
\hphantom{XXXXXXXXXXX}
 + \wick{\Psi_k^2} v_{N, j} + 2 v_{N, k} \wick{ \Psi_k \Psi_j } 
+ \wick{\Psi_k^2 \Psi_j} \Big)
\end{split}
\label{CN1}
\end{align}

\noi
and
\begin{align}
\begin{split}
v_j 
& = S (t) \big(u_{j}^{(0)}, u_{ j}^{(1)} \big)\\
& \quad - \I\Big(
 \E [ v_j^2 ] v_j
+ 2 \E [ \Psi_j v_j ] v_j 
+ \E [ v_j^2 ] \Psi_j 
+ 2 \E [v_j  \Psi_j  ] \Psi_j \Big), 
\end{split}
\label{CN2}
\end{align}

\noi
where $S (t)$ and $\I$ are as in  \eqref{D4} and  \eqref{D5}, 
respectively.
We set 
\begin{align}
V_{N, j} =  v_{N, j} - v_j 
\label{CN3}
\end{align}

\noi
and $\vec V_{N, j} = (V_{N, j}, \dt V_{N, j})$.
By taking the difference of \eqref{CN1} and \eqref{CN2}, we have
\begin{align}
\begin{split}
V_{N, j} (t)
& = S(t)
( \vec V_{N, j} (0)) 
 - \sum_{\l = 1}^6 \I(\A_{j, \l})(t)\\
& = S(t)
\big( u_{N, j}^{(0)} - u_{j}^{(0)}, u_{N, j}^{(1)} 
- u_{j}^{(1)}\big) 
 - \sum_{\l = 1}^6 \I(\A_{j, \l})(t), 
\end{split}
\label{CN4}
\end{align}

\noi
where
$\{ \A_{j, \l} = A_{j, \l}(N)\}_{j = 1}^N$, $ \l = 1, \dots, 6$, are given by 
\begin{align}
\begin{split}
\A_{j, 1} & =  
\frac{1}{N} \sum_{k = 1}^N  
v_{N, k}^2 v_{N, j}
-  \E [ v_j^2 ] v_j, \\
\A_{j, 2} & =  
 \frac{1}{N} \sum_{k = 1}^N 
2 \Psi_k v_{N, k} v_{N, j}
- 2 \E [ \Psi_j v_j ] v_j ,\\
\A_{j, 3} & =   \frac{1}{N} \sum_{k = 1}^N 
v_{N, k}^2 \Psi_j - 
\E [ v_j^2 ] \Psi_j , \\
\A_{j, 4} & =   \frac{1}{N} \sum_{k = 1}^N 
\wick{\Psi_k^2} v_{N, j},  \\
\A_{j, 5} & =   \frac{1}{N} \sum_{k = 1}^N 
 2 v_{N, k} \wick{ \Psi_k \Psi_j } 
 -  2 \E [ v_j \Psi_j  ] \Psi_j, \\
\A_{j, 6} & =   \frac{1}{N} \sum_{k = 1}^N 
 \wick{\Psi_k^2 \Psi_j} .
\end{split}
\label{CN5}
\end{align}

Fix a target time $T \gg 1$
and we work on the interval $[0, T]$.
We point out the following;
while 
we take $\frac 45 < s < 1$
to prove  global-in-time convergence
(namely, on $[0, T]$ for any given $T\gg1$),
 the convergence argument presented below works for 
$\frac 12 \le s < 1$, 
{\it assuming} that 
$v_j$ exists on $[0, T]$
(more precisely, $(v_j, \dt v_j) \in 
L^2(\O;  C ([0, T]; \H^s (\T^2) ))$).

Since 
the parameters $s, \eps, T$ are fixed, we often drop the dependence
on these parameters.
In the remaining part of this subsection, 
we fix small $\dl > 0$.

\medskip

\noi
$\bullet$ {\bf Step 1:}
In this first step, we  establish the  convergence \eqref{IV4}
for each $T \gg 1$:
\begin{align}
\| \vec V_{N, j} \|_{X^s_N(T)}
= \| \vec v_{N, j}  - \vec v_j \|_{X^s_N(T)}
\too 0  
\label{CN5a}
\end{align}

\noi
in probability as $N \to \infty$, 
where $X^s_N(T)$ is as in \eqref{XN1}.

Before we start estimating the terms on the right-hand side of  \eqref{CN4}, 
let us first  prepare 
data sets indexed by $N \gg1 $. 
As we noted in the proof of Lemma \ref{LEM:LLN2}, 
the solutions $\vec v_j = (v_j, \dt v_j) \in C([0, T]; \H^s(\T^2))$
 to~\eqref{NLW5}, $j \in \N$,  are 
 independent 
and identically distributed.
In particular, we have 
\begin{align}
\sup_{j \in \N} 
\| \vec v_j \|_{L^2_\o C_T \H^s_x}
= 
\| \vec v_1 \|_{L^2_\o C_T \H^s_x}
 < \infty.
\label{BD1}
\end{align}

\noi
Given $N \in \N$ and $R \ge 1$, define
the set $\O_{1, N} (R) \subset \O$
by setting
\begin{align}
\O_{1, N}(R) = \Big\{ \o \in \O: 
\| \vec v_j \|_{X_N^s(T)}\le R\Big\}.
\label{BD2}
\end{align}

\noi
Then, it follows from  Chebyshev's  inequality, 
\eqref{XN1}, 
and \eqref{BD1} that 
there exists $R = R(\dl) \gg1 $, independent of $N \in \N$,  such that 
\begin{align}
\PP\big( (\O_{1, N}(R))^c \big) < \frac \dl 8.
\label{BD4}
\end{align}

\noi
Given small $\eta_0 > 0$,  
define $\O_{2, N}(\eta_0)\subset \O$ by setting 
\begin{align}
\O_{2, N}(\eta_0) = \big\{ \o \in \O: 
\|\vec  V_{N, j} (0)\|_{\A_N \H^s_x} 
= \big\|\big( u_{N, j}^{(0)} - u_{j}^{(0)}, u_{N, j}^{(1)} 
- u_{j}^{(1)}\big) \big\|_{\A_N \H^s_x} \le 
\eta_0\Big\}.
\label{BD5}
\end{align}

\noi
Then, from 
the assumption \eqref{IV2} on the initial data, 
there exists  $N_1 = N_1(\eta_0, \dl) \in \N$ such that
\begin{align}
\PP\big( (\O_{2, N}(\eta_0))^c \big) < \frac \dl 8.
\label{BD5a}
\end{align}

\noi
for any $N \ge N_1$.

Next, we discuss controls on the stochastic terms.
Given $K \ge 1 $, define the set $\O_{3, N}(K) \subset \O$ by setting 
\begin{align}
\O_{3, N}(K) = \Big\{ \o \in \O: \| \Psi_j \|_{\A_{N} C_T W_x^{-\eps, \infty}}
+ \| \wick{\Psi_k\Psi_j} \|_{\A^{(2)}_NC_T W^{-\eps, \infty}_x}  \le K\Big\}.
\label{BD6}
\end{align}

\noi
Then, it follows from Lemma \ref{LEM:sto1}\,(ii) 
and  Chebyshev's  inequality
that 
there exists $K  = K(\dl) \gg1 $, independent of $N \in \N$,  such that 
\begin{align}
\PP\big( (\O_{3, N}(K))^c \big) < \frac \dl 8 .
\label{BD7}
\end{align}

\noi
Given small $\eta > 0$, 
define the set 
$\O_{4, N}(\eta) \subset \O$ by setting 
\begin{align}
\O_{4, N}(\eta)
= \Big\{ \o \in \O:  
\|(\vv, \pmb \Psi) \|_{\Y_N} \le \eta
 \Big\}, 
\label{BDX}
\end{align}

\noi
where the $\Y_N$-norm is defined by 
\begin{align}
\begin{split}
\|(\vv, \pmb \Psi) \|_{\Y_N} &  =  \bigg\| \frac{1}{N} \sum_{k = 1}^N 
\wick{\Psi_k^2} \bigg\|_{ L_T^2 W_x^{- \eps, \infty}}
+  \bigg\| \frac 1N \sum_{k = 1}^N 
\wick{\Psi_k^2 \Psi_j} \bigg\|_{ \A_{N, j}L_T^2 W_x^{- \eps, \infty}}\\
& \quad + \bigg\| \frac{1}{N} \sum_{k = 1}^N 
\Big( \Psi_k v_k - \E [\Psi_k v_k] \Big) \bigg\|_{L_T^2
W_x^{-\eps, \frac 4{2- s - \eps}}}\\
& \quad 
+ 
\bigg\| \frac{1}{N} \sum_{k = 1}^N
\Big( v_k^2 - \E [v_k^2] \Big) \bigg\|_{L_T^2 (L^\frac 3{2-s}_x
\cap W_x^{\eps, \frac 2{2-s-\eps}})} \\
& \quad
+ \bigg\| \frac 1N \sum_{k = 1}^N \Big( v_k \!\wick{\Psi_k \Psi_j} - \E [v_k\Psi_k ] \Psi_j \Big) 
\bigg\|_{\A_{N, j}L_T^2 H_x^{s-1}}.
 \end{split}
\label{BD8} 
\end{align}

\noi
Then, 
it follows from Lemmas \ref{LEM:LLN1}, 
 \ref{LEM:LLN2},  and \ref{LEM:LLN3}
 with Chebyshev's inequality
that there exists $N_2 = N_2(\eta, \dl) \in \N$ such that 
\begin{align}
\PP\big( (\O_{4, N}(\eta))^c \big) < \frac \dl 8
\label{BD9}
\end{align}

\noi
for $N \ge N_2$.

\medskip

We are now ready to estimate the terms on the right-hand side of \eqref{CN4}.
Let $J \subset [0, T]$ be an interval
with $|J| \le 1$.
With a slight abuse of notation, 
we define
 $X^s_N(J)$ by setting 
\begin{align}
\begin{split}
\|  \vec v_j\|_{X^s_N(J)}
& 
= 
\| ( v_j, \dt  v_j)\|_{\A_N C_J  \H^s_x}\\
& = 
\| ( v_{j}, \dt   v_{j})\|_{\A_{N, j} C_J  \H^s_x}, 
\end{split}
\label{CN5b}
\end{align}

\noi
where  the $\A_NB$-norm is as in \eqref{AN0}.
In the following, we work on 
$\O_N (R, K, \eta_0, \eta)$ defined by 
\begin{align}
\O_N  (R, K, \eta_0, \eta)= \O_{1, N}(R)
\cap \O_{2, N}(\eta_0)
 \cap \O_{3, N}(K) \cap \O_{4, N}(\eta), 
\label{BD10}
\end{align}

\noi
where 
$\O_{1, N}(R)$, $\O_{2, N}(\eta_0)$, $\O_{3, N}(K)$, and $\O_{4, N}(\eta)$
are as in \eqref{BD2}, \eqref{BD5}, \eqref{BD6}, and \eqref{BDX}, 
respectively.
Before proceeding further, we note that,  given small $\dl > 0$, 
by choosing 
 $R = R(\dl)$
 and $K = K(\dl)$ sufficiently large, 
 it follows from 
\eqref{BD10}
with  \eqref{BD4}, \eqref{BD5a}, \eqref{BD7}, and \eqref{BD9}
that 
\begin{align}
\PP\big((\O_N  (R, K, \eta_0, \eta))^c\big) < \frac \dl 2
\label{BD11}
\end{align}

\noi
for any $N \ge \max (N_1, N_2)$, 
where $N_1(\eta_0, \dl)$
and  $N_2(\eta, \dl)$
are as in \eqref{BD5a}
and \eqref{BD9}, respectively.

From \eqref{CN5} with \eqref{CN3}, we have 
\begin{align}
\begin{split}
\A_{j, 1} &=  \frac{1}{N} \sum_{k = 1}^N V_{N, k} v_{N, k} v_{N, j} 
+  \frac{1}{N} \sum_{k = 1}^N  v_k V_{N, k} v_{N, j} \\
& \quad +  \frac{1}{N} \sum_{k = 1}^N  v_k v_k V_{N, j} 
 + \frac{1}{N} \sum_{k = 1}^N \Big( v_k^2 - \E [ v_k^2 ] \Big)  v_j.
\end{split}
\label{CN6}
\end{align}

\noi
Using \eqref{CN3}, 
we replace each occurrence of $v_{N, \l}$ in \eqref{CN6} by $V_{N, \l} + v_\l$, $\l = j, k$.
Then, 
by proceeding as in \eqref{LWP2}
(while keeping the $k$-summation inside for the  last term
in \eqref{CN6})
with~\eqref{BD2} and \eqref{BDX}, 
we have 
\begin{align}
\begin{split}
\| \vec \I (\A_{j , 1})\|_{X^s_{N}(J)} 
&\les |J|
\Big( \| \vec V_{N, k}\|_{X^s_{N}(J)}^2 
+ \| \vec v_{k}\|_{X^s_{N}(J)}^2 \Big)
\|\vec  V_{N, j}\|_{X^s_{N}(J)}
\\
&\quad + |J|^\frac 12 
\bigg\| \frac{1}{N} \sum_{k = 1}^N 
\Big( v_k^2 - \E [ v_k^2 ] \Big) \bigg\|_{L_J^2 L^\frac 3{2- s}_x} 
\| \vec v_j \|_{X^s_{N}(J)}\\
&\le |J|
\Big( \| \vec V_{N, k}\|_{X^s_{N}(J)}^2 
+ R^2  \Big)
\|\vec  V_{N, j}\|_{X^s_{N}(J)}
+|J|^\frac 12 \eta  R, 
\end{split}
\label{CN7}
\end{align}

\noi
uniformly in $N\in \N$,
provided that $\frac 12 \le s$ ($< 1$).

From \eqref{CN5} with \eqref{CN3}, we have 
\begin{align}
\begin{split}
\A_{j, 2}  
& =  
 \frac{1}{N} \sum_{k = 1}^N 
2 \Psi_k V_{N, k} v_{N, j}
+  \frac{1}{N} \sum_{k = 1}^N 
2 \Psi_k v_{k} V_{N, j}\\
& \quad 
+ 
 \frac{2}{N} \sum_{k = 1}^N 
\Big( \Psi_k v_{k}
-  \E [ \Psi_k v_k ]\Big) v_j .
\end{split}
\label{CN7a}
\end{align}

\noi
Then, 
by proceeding as in \eqref{LWP3} (see also \eqref{WN6})
with Lemma \ref{LEM:prod}\,(ii), 
Sobolev's inequality, 
\eqref{BD2}, 
\eqref{BD6}, 
and  \eqref{BDX}, 
we have 
\begin{align}
\begin{split}
 \| & \vec \I(\A_{j, 2})\|_{X^s_{N}(J)}\\
& 
\les 
|J|
\| \Psi_k \|_{\A_{N} C_J W_x^{-\eps, \infty}}
\Big(\| \vec  V_{N, k} \|_{X^s_N(J)} 
+ \| \vec  v_k \|_{X^s_N(J)} \Big)
\| \vec  V_{N, j} \|_{X^s_{N}(J)} \\
& \quad + 
|J|^\frac 12  \bigg\| \frac{1}{N} \sum_{k = 1}^N 
\Big( \Psi_k v_{k}
-  \E [ \Psi_k v_k ]\Big)\bigg\|_{L^2_J W^{-\eps, \frac 4{2- s - \eps}}_x}
\| v_j \|_{\A_{N, j}C_J W^{\eps, \frac 4{2- s - \eps}}_x}\\
& 
\le 
|J|
K 
\Big(\| \vec  V_{N, k} \|_{X^s_N(J)} 
+ R \Big)
\| \vec  V_{N, j} \|_{X^s_{N}(J)}  + 
 |J|^\frac 12 \eta  R, 
\end{split}
\label{CN8}
\end{align}

\noi
uniformly in $N\in \N$, 
provided that $3 \eps \le s$ ($< 1$).
Similarly, by 
writing
\begin{align}
\A_{j, 3} & =  
 \frac{1}{N} \sum_{k = 1}^N 
 V_{N, k} v_{N, k} \Psi_j
 +  \frac{1}{N} \sum_{k = 1}^N 
v_k  V_{N, k} \Psi_j
+ 
 \frac{1}{N} \sum_{k = 1}^N 
\Big(v_{k}^2 
- \E [ v_k^2 ] \Big) \Psi_j 
\label{CN8a}
\end{align}

\noi
and proceeding as in \eqref{CN8} (see also \eqref{WN7}) with 
\eqref{BD2}, 
\eqref{BD6}, 
and 
\eqref{BDX}, 
we have 
\begin{align}
\begin{split}
 \| & \vec \I(\A_{j, 3})\|_{X^s_{N}(J)}\\
& 
\les 
|J|
\Big(\| \vec  V_{N, k} \|_{X^s_N(J)} 
+ \| \vec  v_k \|_{X^s_N(J)} \Big)
\| \vec  V_{N, k} \|_{X^s_{N}(J)} 
\| \Psi_j \|_{\A_{N} C_J W_x^{-\eps, \infty}}
\\
& \quad + 
|J|^\frac 12
\bigg\| \frac{1}{N} \sum_{k = 1}^N 
\Big( v_k^2 - \E [ v_k^2 ] \Big) \bigg\|_{L_J^2 
 W_x^{\eps, \frac 2{2-s-\eps}}}
\| \Psi_j \|_{\A_{N} C_J W^{-\eps, \infty}_x}\\
& 
\le
|J|
K 
\Big(\| \vec  V_{N, k} \|_{X^s_N(J)} 
+ R \Big)
\| \vec  V_{N, j} \|_{X^s_{N}(J)}  + 
 |J|^\frac 12 \eta  K, 
\end{split}
\label{CN9}
\end{align}

\noi
uniformly in $N\in \N$.

Proceeding as in \eqref{LWP5}
with 
\eqref{BD2}
and 
\eqref{BDX}, 
we have
\begin{align}
\begin{split}
 \|  \vec \I(\A_{j, 4})\|_{X^s_{N}(J)}
& \les |J|^{\frac 12} \bigg\|
\frac{1}{N} \sum_{k = 1}^N 
\wick{\Psi_k^2} 
\bigg\|_{L^2_T W^{-\eps, \infty}_x}  
\Big(\| \vec  V_{N, j} \|_{X^s_{N}(J)} 
+ 
\| \vec  v_j \|_{X^s_{N}(J)} 
\Big)\\
& \les 
 |J|^{\frac 12} 
\| \vec  V_{N, j} \|_{X^s_{N}(J)} 
+   |J|^{\frac 12}\eta  R , 
\end{split}
\label{CN10}
\end{align}

\noi
uniformly in $N\in \N$.
provided that $\eps \le s  \le 1-\eps$. 
By writing
\begin{align}
\A_{j, 5} & =   
\frac{2}{N} \sum_{k = 1}^N 
 V_{N, k} \wick{ \Psi_k \Psi_j } 
 + 
\frac{2}{N} \sum_{k = 1}^N 
\Big( v_{k} \wick{ \Psi_k \Psi_j } 
 -   \E [ v_k \Psi_k  ] \Psi_j\Big)
\label{CN11}
\end{align}

\noi
and 
proceeding as in \eqref{LWP5}
with 
\eqref{BD6}
and
\eqref{BDX}, 
we have
\begin{align}
\begin{split}
 \|  \vec \I(\A_{j, 5})\|_{X^s_{N}(J)}
& \les |J|
\| \wick{\Psi_k\Psi_j} \|_{\A^{(2)}_{N} C_J W^{-\eps, \infty}_x}  
\| \vec  V_{N, k} \|_{X^s_{N}(J)}  +   |J| ^\frac 12 \eta \\
& \le |J|
K 
\| \vec  V_{N, k} \|_{X^s_{N}(J)}  +   |J| ^\frac 12 \eta  , 
\end{split}
\label{CN12}
\end{align}

\noi
uniformly in $N\in \N$, 
where $\A^{(2)}_N B$ is as in \eqref{AN2}.
Lastly, from \eqref{CN5}
and 
 \eqref{BDX}, 
we have 
\begin{align}
\begin{split}
 \|  \vec \I(\A_{j, 6})\|_{X^s_{N}(J)}
& \les |J|^{\frac 12} 
\bigg\| 
\frac{1}{N} \sum_{k = 1}^N 
 \wick{\Psi_k^2 \Psi_j} \bigg\|_{\A_{N, j} L^2_J H^{s-1}_x}\\
& \le  |J|^{\frac 12}  \eta , 
\end{split}
\label{CN13}
\end{align}

\noi
uniformly $N\in \N$.

Therefore, 
putting  \eqref{CN4}
(but starting at time $t_0$), 
\eqref{CN7}, 
\eqref{CN8}, \eqref{CN9},  \eqref{CN10}, 
\eqref{CN12}, and \eqref{CN13}
together
with $J = [t_0, t_1]\subset [0, T]$
such that $ |J|\le 1$, we obtain
\begin{align}
\begin{split}
\|\vec V_{N, j}\|_{X^s_{N}(J)}
& \le 
C_1 
\|\vec  V_{N, j} (t_0) \|_{\A_N \H^s_x} \\
& \quad 
+ C_2 |J|^\frac 12 
K 
\Big( \| \vec V_{N, k}\|_{X^s_{N}(J)}^2 
+ R^2  \Big)
\|\vec  V_{N, j}\|_{X^s_{N}(J)}\\
& \quad 
+C_3 |J|^\frac 12 \eta  (R + K) ,
\end{split}
\label{CN14}
\end{align}

\noi
where the constants $C_1, C_2, C_3 > 0$ are independent of $N\in \N$.
Suppose that 
\begin{align}
\|\vec  V_{N, j}\|_{X^s_{N}(J)}
\le 1
\label{CN15}
\end{align}

\noi
for some interval $J \subset [0, T]$.
Then, 
it follows from \eqref{CN14} that 
there exists $\tau = \tau (R, K) \in (0, 1]$ such that 
\begin{align}
\begin{split}
\|\vec V_{N, j}\|_{X^s_{N}(J)}
& \le 
2 C_1 
\|\vec  V_{N, j} (t_0) \|_{\A_N \H^s_x} 
+2 C_3 \eta, 
\end{split}
\label{CN16}
\end{align}

\noi
provided that $|J| \le \tau$.

While the claimed convergence-in-probability
 \eqref{CN5a}
follows from 
\eqref{BD11}, 
the discussion above, 
\eqref{BD5} (with sufficiently small $\eta_0> 0$), 
and a standard continuity argument, 
we present some details for readers' convenience, 
especially similar iterative arguments
are needed in Step 2 and in Subsection \ref{SUBSEC:6.2}.

Fix small $\dl > 0$.
With $R = R(\dl)\gg 1$ and $K = K(\dl) \gg1$
as in \eqref{BD4} and \eqref{BD7}, 
let $\tau = \tau(R(\dl), K(\dl)) = \tau(\dl)> 0$
as in \eqref{CN16}.
Then, write $[0, T]$ as
\begin{align*}
[0, T] = \bigcup_{\l = 0}^{[T/\tau]} J_\l,
\end{align*}

\noi 
where $J_\l = [\l \tau, (\l+1)\tau] \cap [0, T]$.
Let us work on the first interval $J_0 = [0, \tau]$.
Note  from \eqref{BD5} that 
\begin{align}
\|\vec  V_{N, j} (0) \|_{\A_N \H^s_x} 
\le \eta_0 \ll 1.
\label{CN17}
\end{align}

\noi
Then, in view of  the continuity 
of $\|V_{N, j}\|_{X^s_{N}([0, t])}$
in $t$, 
it follows from a standard continuity argument 
with \eqref{CN15}, \eqref{CN16}, and \eqref{CN17}
that 
\begin{align*}
\|\vec V_{N, j}\|_{X^s_{N}(J_0)}
\le  2 C_1 \eta_0 + 2C_3 \eta, 
\end{align*}

\noi
which in turn implies
\begin{align}
\|\vec  V_{N, j} (\tau ) \|_{\A_N \H^s_x} 
\le  2 C_1 \eta_0 + 2C_3 \eta \ll 1.
 \label{CN18}
\end{align}

\noi
As a consequence, 
it follows from \eqref{CN18},  \eqref{CN16}, and a continuity argument 
 that 
\begin{align}
\|\vec  V_{N, j} (2\tau ) \|_{\A_N \H^s_x} 
\le 
\|\vec V_{N, j}\|_{X^s_{N}(J_1)}
\le 
 (2C_1)^2 \eta_0  + 2C_3 ( 2C_1 +1)  \eta \ll1, 
 \label{CN18a}
\end{align}

\noi
which allows us to iterate the argument on the third interval $J_2$.

After $\l +1$ steps, we obtain
\begin{align}
\begin{split}
\|\vec  V_{N, j} ((\l +1) \tau ) \|_{\A_N \H^s_x} 
\le 
\|\vec V_{N, j}\|_{X^s_{N}(J_\l )}
& \le 
 (2C_1)^{\l+1} \eta_0 +  
2C_3  \sum_{\l' = 0}^{\l} (2C_1)^{\l'} \cdot \eta\\
& \le  (2C_1)^{\l+1} \eta_0 +  
2C_3  
\frac { (2C_1)^{\l+1}}{2C_1 - 1}
\eta.
\end{split}
\label{CN19}
\end{align}

\noi
In order to cover the entire interval $[0, T]$, 
recalling that $\tau$ depends only on $\dl$, 
it suffices to choose 
 $\eta_0 = \eta_0(\dl)$ and  $\eta = \eta(\dl) $ sufficiently small
such that 
\begin{align}
\bigg(\Big[\frac T\tau\Big]+1\bigg)\bigg( (2C_1)^{[\frac T\tau]+1} \eta_0 +  
2C_3  
\frac { (2C_1)^{[\frac T\tau]+1}}{2C_1 - 1}
 \eta\bigg) \le \dl \ll 1.
\label{CN20}
\end{align}

\noi
Then, it follows from 
\eqref{CN19} and \eqref{CN20} that 
\begin{align}
\|\vec V_{N, j}\|_{X^s_N([0, T])}
\le
\sum_{\l = 0}^{[T/\tau]} 
\|\vec V_{N, j}\|_{X^s_{N}(J_\l )}
\le \dl, 
\label{CN21}
\end{align}

\noi
on $\O_N  (R, K, \eta_0, \eta)$
defined in \eqref{BD10}, 
for any 
 $N \ge \max \big(N_1(\eta_0(\dl), \dl), N_2(\eta(\dl), \dl)\big)$, 
 where $N_1$ and $N_2$ are as in 
 \eqref{BD5a} and  \eqref{BD9}, respectively.
In view of 
\eqref{BD11}, 
this proves
the 
convergence-in-probability 
\eqref{CN5a}.

\medskip

\noi
$\bullet$ {\bf Step 2:}
In this part, we establish the  convergence
\eqref{IV3}
for each {\it fixed}  $j \in \N$:
\begin{align}
\| \vec V_{N, j} \|_{C_T \H_x^s}
= \| \vec v_{N, j}  - \vec v_j \|_{C_T \H_x^s}
\too 0  
\label{DN1}
\end{align}

\noi
in probability
as $N \to \infty$.

We first go over the construction of data sets indexed by $N \gg 1$.
Fix $j \in \N$.
\noi
Given  $R \ge 1$, define
the set $\O_{5, j} (R) \subset \O$
by setting
\begin{align}
\O_{5, j}(R) = \Big\{ \o \in \O: 
\| \vec v_j \|_{C_T\H^s_x}\le R\Big\}.
\label{DN1a}
\end{align}

\noi
Then, it follows from  Chebyshev's  inequality
and \eqref{BD1} that 
there exists $R = R(\dl) \gg1 $, independent of $j \in \N$,  such that 
\begin{align}
\PP\big( (\O_{5, j}(R))^c \big) < \frac \dl {10}.
\label{DN1b}
\end{align}

\noi
Given small $\eta_0 > 0$,  
define $\O_{6, j, N}(\eta_0)\subset \O$ by setting 
\begin{align}
\O_{6, j, N}(\eta_0)
= \Big\{ \o \in \O: 
\| \vec V_{N, j} (0)\|_{ \H^s_x} 
\le
\eta_0\Big\}.
\label{BDX0}
\end{align}

\noi
Then, from 
the assumption \eqref{IV1} on the initial data, 
there exists  $N_3 = N_3(j, \eta_0, \dl) \in \N$ such that
\begin{align}
\PP\big( (\O_{6, j, N}(\eta_0))^c \big) < \frac \dl {10}.
\label{DN1c}
\end{align}

\noi
for any $N \ge N_3$.
Given small $\eta > 0$, 
define the set $\O_{7, N}(\eta)\subset \O$
by setting
\begin{align}
\O_{7, N} (\eta)= \Big\{\o \in \O: \| \vec V_{N, k}\|_{X^s_{N}(J)} \le \eta\Big\}
\label{BDX3}
\end{align}

\noi
Then, it follows 
from the convergence-in-probability \eqref{CN5a}
established in 
Step 1 that, given $\dl > 0$, 
there exists $N_4 = N_4(\eta, \dl) \in \N$ such that 
\begin{align}
\PP\big( (\O_{7, N}(\eta))^c \big) < \frac \dl {10}
\label{BDX4}
\end{align}

\noi
for any $N \ge N_4$.
Given $j \in \N$ and $K \ge 1 $, define the set $\O_{8, j, N}(K) \subset \O$ by setting 
\begin{align}
\O_{8, j, N}(K) = \Big\{ \o \in \O: 
 \| \Psi_j \|_{C_T W^{-\eps, \infty}_x} 
 + 
 \| \wick{\Psi_k\Psi_j} \|_{\A_{N, k}C_T W^{-\eps, \infty}_x}  \le K\Big\}.
\label{BDX5}
\end{align}

\noi
Then, it follows from Lemma \ref{LEM:sto1}\,(i) 
(for the uniform (in $j, k$) moment bound
on the $C_T W^{-\eps, \infty}_x$-norm of $\wick{\Psi_k\Psi_j}$)
and  Chebyshev's  inequality
that 
there exists $K  = K(\dl) \gg1 $, independent of $j, N \in \N$,  such that 
\begin{align}
\PP\big( (\O_{8, j, N}(K))^c \big) < \frac \dl {10} .
\label{BDX6}
\end{align}

\noi
Given small $\eta > 0$, 
define the set 
$\O_{9, j, N}(\eta) \subset \O$ by setting 
\begin{align}
\O_{9, j, N}(\eta)
= \Big\{ \o \in \O:  
\|(\vv, \pmb \Psi) \|_{ \Y_{j, N}} \le \eta
 \Big\}, 
\label{BDX7}
\end{align}

\noi
where the $\Y_{j, N}$-norm is defined by 
\begin{align}
\begin{split}
\|(\vv, \pmb \Psi) \|_{ \Y_{j, N}}
& =  \bigg\| \frac 1N \sum_{k = 1}^N 
\wick{\Psi_k^2 \Psi_j} \bigg\|_{ L_T^2 W_x^{- \eps, \infty}}\\
& \quad 
+ \bigg\| \frac 1N \sum_{k = 1}^N \Big( v_k\! \wick{\Psi_k \Psi_j} 
- \E [ v_k\Psi_k ] \Psi_j \Big) \bigg\|_{L_T^2 H_x^{s-1}} .
\end{split}
\label{BDX2}
\end{align}

\noi
Then, 
it follows from Lemmas \ref{LEM:LLN1} and \ref{LEM:LLN3}
that there exists $N_5 = N_5(j, \eta, \dl) \in \N$ such that 
\begin{align}
\PP\big( (\O_{9, j, N}(\eta))^c \big) < \frac \dl {10}
\label{BDX9}
\end{align}

\noi
for $N \ge N_5$.

\medskip

We now turn to estimating the terms
on the right-hand side of \eqref{CN4}
for {\it fixed} $j \in \N$, 
which follows from 
a straightforward modification of the estimates in Step 1.
In the following, we work on 
$\O_{j, N} (R, K, \eta_0,\eta)$ defined by 
\begin{align}
\begin{split}
\O_{j, N}  (R, K, \eta_0, \eta)
& = \O_N  (R, K, \eta_0, \eta)\cap 
\O_{5, j}(R)\cap 
\O_{6, j, N}(\eta_0)\\
& \quad \cap 
 \O_{7, N}(\eta) \cap \O_{8, j, N}(K) \cap \O_{9, j, N}(\eta).
\end{split}
\label{BDX10}
\end{align}

\noi
where 
$\O_N  (R, K, \eta_0, \eta)$, 
$\O_{5, j}(R)$, 
$\O_{6, j, N}(\eta_0)$, 
$ \O_{7, N}(\eta)$, $\O_{8, j, N}(K)$, and  $\O_{9, j, N}(\eta)$
are 
 as in \eqref{BD10}, 
 \eqref{DN1a}, 
 \eqref{BDX0}, 
 \eqref{BDX3}, \eqref{BDX5}, and \eqref{BDX7}, 
 respectively.
Then,   given small $\dl > 0$, 
by choosing 
 $R = R(\dl)$
 and $K = K(\dl)$ sufficiently large, 
 it follows from~\eqref{BDX10}
with \eqref{BD11},  
\eqref{DN1b}, \eqref{DN1c}, 
 \eqref{BDX4}, \eqref{BDX6}, and \eqref{BDX9}
that 
\begin{align}
\PP\big((\O_{j, N}  (R, K, \eta_0, \eta))^c\big) < \dl
\label{BDX11}
\end{align}

\noi
for any $N \ge N_*$, where
\begin{align}
N_* = \max 
\Big(
N_1(\eta_0, \dl), 
 N_2(\eta, \dl), 
N_3(j, \eta_0, \dl), 
N_4(\eta, \dl), N_5(j, \eta, \dl)\Big).
\label{BDX12}
\end{align}

\noi
Here, $N_\l$, $\l = 1, \dots, 5$, 
are as in 
 \eqref{BD5a}, 
\eqref{BD9}, 
\eqref{DN1c}, 
\eqref{BDX4}, 
and 
\eqref{BDX9}, 
respectively.

Let $J \subset [0, T]$ be an interval
with $|J| \le 1$.
Proceeding as in \eqref{CN7}
with 
\eqref{CN6}, 
\eqref{BD2}, 
\eqref{BDX}, \eqref{BD8}, 
\eqref{DN1a}, 
and
\eqref{BDX3}, 
we have 
\begin{align}
\begin{split}
\| \vec \I (\A_{j , 1})\|_{C_J \H^s_x}
&\les 
|J|
\Big( \| \vec V_{N, k}\|_{X^s_{N}(J)}^2 
+ \| \vec v_{k}\|_{X^s_{N}(J)}^2 \Big)
\|\vec  V_{N, j}\|_{C_J \H^s_x}
\\
&\quad + 
|J|
\| \vec V_{N, k}\|_{X^s_{N}(J)}
\Big( \| \vec V_{N, k}\|_{X^s_{N}(J)}
+ \| \vec v_{k}\|_{X^s_{N}(J)} \Big)
\|\vec  v_j \|_{C_J \H^s_x}
\\
&\quad + |J|^\frac 12 
\bigg\| \frac{1}{N} \sum_{k = 1}^N 
\Big( v_k^2 - \E [ v_k^2 ] \Big) \bigg\|_{L_J^2 L^\frac 3{2- s}_x} 
\| \vec v_j \|_{C_J \H^s_x}\\
&\les |J|
 R^2 
\|\vec  V_{N, j}\|_{C_J \H^s_x}
+|J|^\frac 12 
\eta R^2  , 
\end{split}
\label{DN2}
\end{align}

\noi
uniformly in $N\in \N$,
provided that $\frac 12 \le s$ ($< 1$).
Proceeding as in \eqref{CN8}
with 
\eqref{CN7a}, 
\eqref{BD2}, 
\eqref{BD6}, 
\eqref{BDX}, 
\eqref{BD8}, 
\eqref{DN1a}, 
and 
\eqref{BDX3}, 
we have
\begin{align}
\begin{split}
 \| & \vec \I(\A_{j, 2})\|_{C_J \H^s_x}\\
& 
\les 
|J|
\| \Psi_k \|_{\A_{N} C_J W_x^{-\eps, \infty}}
\Big(\| \vec  V_{N, k} \|_{X^s_N(J)} 
+ \| \vec  v_k \|_{X^s_N(J)} \Big)
\| \vec  V_{N, j} \|_{C_J \H^s_x}\\
& \quad + 
|J|
\| \Psi_k \|_{\A_{N} C_J W_x^{-\eps, \infty}}
\| \vec  V_{N, k} \|_{X^s_N(J)} 
\| \vec  v_j \|_{C_J \H^s_x}\\
& \quad + 
|J|^\frac 12  \bigg\| \frac{1}{N} \sum_{k = 1}^N 
\Big( \Psi_k v_{k}
-  \E [ \Psi_k v_k ]\Big)\bigg\|_{L^2_J W^{-\eps, \frac 4{2- s - \eps}}_x}
\| \vec v_j \|_{C_J \H^s_x}\\
& 
\les 
|J|
K R
\| \vec  V_{N, j} \|_{C_J \H^s_x}
+  |J|^\frac 12 
\eta  K R, 
\end{split}
\label{DN3}
\end{align}

\noi
uniformly in $N\in \N$,
provided that $3 \eps \le s$ ($< 1$).
Similarly, 
proceeding as in \eqref{CN9}
with 
\eqref{CN8a}, 
we have 
\begin{align}
\begin{split}
 \| & \vec \I(\A_{j, 3})\|_{C_J \H^s_x}\\
& 
\les 
|J|
\Big(\| \vec  V_{N, k} \|_{X^s_N(J)} 
+ \| \vec  v_k \|_{X^s_N(J)} \Big)
\| \vec  V_{N, k} \|_{X^s_{N}(J)} 
\| \Psi_j \|_{ C_J W_x^{-\eps, \infty}}
\\
& \quad + 
|J|^\frac 12
\bigg\| \frac{1}{N} \sum_{k = 1}^N 
\Big( v_k^2 - \E [ v_k^2 ] \Big) \bigg\|_{L_J^2 
 W_x^{\eps, \frac 2{2-s-\eps}}}
\| \Psi_j \|_{ C_J W^{-\eps, \infty}_x}\\
& 
\les
 |J|^\frac 12 \eta K R, 
\end{split}
\label{DN4}
\end{align}

\noi
uniformly in $N\in \N$.
Proceeding as in \eqref{CN10}
with \eqref{CN5}, 
we have
\begin{align}
\begin{split}
 \|  \vec \I(\A_{j, 4})\|_{C_J \H^s_x}
& \les |J|^{\frac 12} \bigg\|
\frac{1}{N} \sum_{k = 1}^N 
\wick{\Psi_k^2} 
\bigg\|_{L^2_T W^{-\eps, \infty}_x}  
\Big(\| \vec  V_{N, j} \|_{C_J \H^s_x}
+ 
\| \vec  v_j \|_{C_J \H^s_x}
\Big)\\
& \les 
 |J|^{\frac 12} 
\| \vec  V_{N, j} \|_{C_J \H^s_x}
+   |J|^{\frac 12}\eta  R, 
\end{split}
\label{DN5}
\end{align}

\noi
uniformly in $N\in \N$, 
provided that $\eps \le s  \le 1-\eps$. 
Proceeding as in \eqref{CN12} with \eqref{CN11}
and \eqref{BDX5}, 
we have
\begin{align}
\begin{split}
 \|  \vec \I(\A_{j, 5})\|_{C_J \H^s_x}
& \les |J|
\| \wick{\Psi_k\Psi_j} \|_{\A_{N, k} C_J W^{-\eps, \infty}_x}  
\| \vec  V_{N, k} \|_{X^s_{N}(J)} 
+   |J| ^\frac 12 \eta \\
& \les   |J| ^\frac 12 \eta K, 
\end{split}
\label{DN6}
\end{align}

\noi
uniformly in $N\in \N$.
Lastly, from \eqref{CN5}
and 
\eqref{BDX7}, 
we have 
\begin{align}
\begin{split}
 \|  \vec \I(\A_{j, 6})\|_{C_J \H^s_x}
& \les |J|^{\frac 12} 
\bigg\| 
\frac{1}{N} \sum_{k = 1}^N 
 \wick{\Psi_k^2 \Psi_j} \bigg\|_{ L^2_J H^{s-1}_x}\\
& \le  |J|^{\frac 12}  \eta .
\end{split}
\label{DN7}
\end{align}

Therefore, 
putting  \eqref{CN4}
(but starting at time $t_0$), 
\eqref{DN2}, 
\eqref{DN3}, 
\eqref{DN4}, 
\eqref{DN5},
\eqref{DN6},  
and \eqref{DN7}
together
with $J = [t_0, t_1]\subset [0, T]$
such that $ |J|\le 1$, we obtain
\begin{align}
\begin{split}
\|\vec V_{N, j}\|_{C_J \H^s_x}
& \le 
C_4 
\|\vec V_{N, j} (t_0) \|_{ \H^s_x} 
+ C_5 
 |J|^\frac 12 
 R (K+R)
\|\vec  V_{N, j}\|_{C_J \H^s_x}\\
& \quad 
+C_6 |J|^\frac 12  \eta R(K + R) ,
\end{split}
\label{DN8}
\end{align}

\noi
where the constants $C_4, C_5, C_6 > 0$ are independent of $N\in \N$.
Suppose that 
\begin{align*}
\|\vec  V_{N, j}\|_{C_J \H^s_x}
\le 1
\end{align*}

\noi
for some interval $J \subset [0, T]$.
Then, 
it follows from \eqref{DN8} that 
there exists $\tau = \tau (R, K) \in (0, 1]$ such that 
\begin{align*}
\|\vec V_{N, j}\|_{C_J \H^s_x}
& \le 
2 C_4 
\|\vec  V_{N, j} (t_0) \|_{\H^s_x} 
+2 C_6 \eta, 
\end{align*}

\noi
provided that $|J| \le \tau$.
Then, 
by recalling that $\tau$ depends only on $\dl$
and  by choosing
 $\eta_0 = \eta_0(\dl)$ and  $\eta = \eta(\dl) $ sufficiently small, 
it follows from 
a slight modification of the iterative argument in Step~1
with 
\eqref{BDX0}
that 
\begin{align*}
\|\vec V_{N, j}\|_{C_T \H^s_x}
\le \dl
\end{align*}

\noi
on $\O_{j, N}  (R, K, \eta_0, \eta)$, 
satisfying 
\eqref{BDX11}, 
for any 
 $N \ge N_*$, 
 where 
$N_*$ is as in  
\eqref{BDX12}.
 This proves
the 
convergence-in-probability 
\eqref{DN1}, 
and therefore
 concludes the proof of  Theorem \ref{THM:conv1}\,(i).

\subsection{On the convergence rate
under a higher moment assumption}
\label{SUBSEC:5.3}

In this subsection, 
we present a proof of Theorem \ref{THM:conv1}\,(ii)
on a convergence rate under the higher moment assumption~\eqref{IV6a}:
\begin{align}
v_j \in L^6 (\Omega; C([0, T]; H^{s} (\T^2))), \qquad j \in \N
\label{PN1}
\end{align}

\noi
for some $T > 0$ and $\frac 12 \le s < 1$.
Given  $\big(v_j^{(0)}, v_j^{(1)}\big) \in L^6 (\Omega; \H^s (\T^2) )$, 
Proposition \ref{PROP:LWP2}
guarantees~\eqref{PN1}
only for a short time.
We, however,  point out
that, 
as in Subsection \ref{SUBSEC:5.2}, 
the convergence argument presented below
works on $[0, T]$ for any given $T > 0$, 
{\it assuming}  that~\eqref{PN1} 
(together with \eqref{IV5} and \eqref{IV6}) holds.

As compared to the presentation
in Subsection \ref{SUBSEC:5.2}, 
the only difference appears on 
how we handle the terms in \eqref{BD8}
and \eqref{BDX2}.
In the following, we first 
 establish 
suitable replacements
of 
Lemmas~\ref{LEM:LLN2} and \ref{LEM:LLN3}, 
providing the convergence rate of $N^{-\frac 12}$.

The following lemma establishes a replacement of 
 Lemma \ref{LEM:LLN3} 
 (in the spirit of Lemma~\ref{LEM:LLN1}) under the higher moment assumption on $v_j$.

\begin{lemma}\label{LEM:LLN4}
Let $\frac 12  \le s < 1$ and $T > 0$.
Let $\{(\Psi_j, v_j)\}_{j \in \N}$ be as in Lemma \ref{LEM:LLN2}.
Moreover, assume that 
\begin{align}
v_j \in L^p (\Omega; C([0, T]; H^{s} (\T^2)))
\label{PN3a}
\end{align}

\noi
for some $p > 2$.
Then, 
we have
\begin{align}
\sup_{j \in \N} \bigg\| \frac 1N \sum_{k = 1}^N \Big( v_k\! \wick{\Psi_k \Psi_j} 
- \E [ v_k\Psi_k ] \Psi_j \Big) \bigg\|_{L^2_\o L_T^2 H_x^{s-1}} 
& \les_T \frac 1{N^\frac 12}, 
\label{PN2} \\
\bigg\| \frac 1N \sum_{k = 1}^N \Big( v_k \!\wick{\Psi_k \Psi_j} - \E [v_k\Psi_k ] \Psi_j \Big) 
\bigg\|_{L^2_\o\A_{N, j}L_T^2 H_x^{s-1}} 
& \les_T \frac 1{N^\frac 12}, 
\label{PN3}
\end{align}

\noi
uniformly in  $N \in \N$, 
where $\A_N B$ is as in \eqref{AN0}.
As a consequence,
for each $T > 0$, we have 
\begin{align}
\begin{split}
\text{for each fixed $j \in \N:$}
\quad 
\bigg\| \frac 1N \sum_{k = 1}^N \Big( v_k\! \wick{\Psi_k \Psi_j} 
- \E [ v_k\Psi_k ] \Psi_j \Big) \bigg\|_{ L_T^2 H_x^{s-1}} 
& \too 0 , \\
\bigg\| \frac 1N \sum_{k = 1}^N \Big( v_k \!\wick{\Psi_k \Psi_j} - \E [v_k\Psi_k ] \Psi_j \Big) 
\bigg\|_{\A_{N, j}L_T^2 H_x^{s-1}} 
& \too 0 
\end{split}
\label{PN4}
\end{align}

\noi
in probability
at the rate $N^{-\frac 12 }$
as $N \to \infty$
in the sense of Definition \ref{DEF:conv1}.

\end{lemma}

\begin{proof}
As in the proof of Lemma \ref{LEM:LLN1}, 
we focus on proving \eqref{PN2}
and \eqref{PN3}, since \eqref{PN4}
follows from 
\eqref{PN2}
and \eqref{PN3} with Chebyshev's inequality.

In the following, we only consider \eqref{PN3} since \eqref{PN2} follows
from a similar consideration.
From a slight modification of the proof of Lemma \ref{LEM:LLN3}, 
it is easy to see that 
the contribution from the case $k = j$
satisfies  \eqref{PN3}. 
Thus, we only consider the case $k \ne j$ in the following.

Set
\begin{align*}
X_{k, j} = v_k \Psi_k \Psi_j -  \E [v_k\Psi_k ] \Psi_j
\end{align*}

\noi
Then, 
from the mutual independence of $\{(\Psi_k, v_k)\}_{k \in \N}$, 
we have
\begin{align}
\E[X_{k, j}]
 = 
\E\big[v_k  \Psi_k  -  \E [\Psi_k v_k]\big] \cdot \E[ \Psi_j]
 = 0
\label{PN5}
\end{align}

\noi
and 
\begin{align}
\E[X_{k, j}\cdot X_{k', j}]
 = 
\E\big[(v_k  \Psi_k  -  \E [v_k \Psi_k ])(v_{k'}  \Psi_{k'}  -  \E [v_{k'}\Psi_{k'} ])\big] \cdot \E[ \Psi_j^2]
 = 0
\label{PN6}
\end{align}

\noi
for any $k \ne k'$ with $k, k' \ne j$.

A slight modification of \eqref{LNx2} 
with \eqref{PN3a} yields
\begin{align}
\begin{split}
\Big\| \| v_k \Psi_k \Psi_j \|_{L_T^2 H_x^{s-1}} \Big\|_{L^2_\o}
&\les_T  \| v_k \|_{L^p_\o C_T  H_x^{s}}
 \|  \Psi_k \Psi_j \|_{L^\frac{2p}{p-2}_\o C_T  W_x^{- \eps, \infty}}\\
 & \les 1,  
\end{split}
\label{PN7}
\end{align}

\noi
uniformly in $j, k \in \N$.
Similarly, 
a minor modification of \eqref{LNx2a} yields
\begin{align}
\begin{split}
 \Big\| \| \E [ v_k \Psi_k ] \Psi_j \|_{L_T^2 H_x^{s-1}} \Big\|_{L^2_\o}
 & \les_T 1, 
\end{split}
\label{PN8}
\end{align}

\noi
uniformly in $j, k \in \N$.
Hence, proceeding as in \eqref{LN6} with 
\eqref{PN5}, \eqref{PN6}, 
\eqref{PN7},  and \eqref{PN8}, 
we have 
\begin{align*}
\bigg\| \frac 1N \sum_{\substack{k = 1\\k \ne j}}^N X_{k, j}
\bigg\|_{L^2_\o\A_{N, j}L_T^2 H_x^{s-1}}
\le \frac 1{N^\frac 12}
\|  X_{k, j}\|_{\A_{N}^{(2)}L^2_\o L_T^2 H_x^{s-1}}
\les_T \frac 1{N^\frac 12}, 
\end{align*}

\noi
where $\A_{N}^{(2)}B$ is as in \eqref{AN2}.
\end{proof}

In the proof of Lemma \ref{LEM:LLN4}, 
we crucially used the fact that the space-time norm
is $L^2$-based.
Since the terms treated in Lemma \ref{LEM:LLN2} 
involve the $L^r_x$-based Sobolev spaces with $r > 2$, 
we can not proceed as in the proof of 
Lemma \ref{LEM:LLN4}. 
We instead make the observation that 
the terms 
treated in Lemma \ref{LEM:LLN2} 
only appear in 
\eqref{CN6}, \eqref{CN7a}, 
and \eqref{CN8a}:
\begin{align}
\begin{split}
& \frac{1}{N} \sum_{k = 1}^N \Big( v_k^2 - \E [ v_k^2 ] \Big)  v_j, \qquad 
 \frac{2}{N} \sum_{k = 1}^N 
\Big( \Psi_k v_{k}
-  \E [ \Psi_k v_k ]\Big) v_j \\
& \text{and}\qquad 
 \frac{1}{N} \sum_{k = 1}^N 
\Big(v_{k}^2 
- \E [ v_k^2 ] \Big) \Psi_j 
\end{split}
\label{PPN0}
\end{align}

\noi
which do not involve 
the difference $V_{N, j}$ defined in \eqref{CN3}.
The following lemma
directly bounds
the $L^2(\O)$-norms
of the terms in \eqref{PPN0}, 
replacing Lemma \ref{LEM:LLN2}.

\begin{lemma}\label{LEM:LLN5}
Let $\frac 12  \le s < 1$
and $T > 0$.
Let  $\{(\Psi_j, v_j)\}_{j \in \N}$ be as in Lemma \ref{LEM:LLN2}.
Moreover, assume~\eqref{PN1}\textup{:} 
\begin{align}
v_j \in L^6 (\Omega; C([0, T]; H^{s} (\T^2))).
\label{PPN0x}
\end{align}

\noi
Then, 
we have
\begin{align}
\sup_{j \in \N}\bigg\| \frac{1}{N} \sum_{k = 1}^N 
\Big( \Psi_k v_k - \E [\Psi_k v_k] \Big) v_j
\bigg\|_{L^2_\o L_T^2 H^{s-1}_x} 
& \les_T \frac 1{N^\frac 12}, 
\label{PPN1} \\
\sup_{j \in \N}\bigg\| \frac{1}{N} \sum_{k = 1}^N
\Big( v_k^2 - \E [v_k^2] \Big) v_j
\bigg\|_{L^2_\o L_T^2 H^{s-1}_x} &\les_T \frac 1{N^\frac 12}, 
\label{PPN2} \\
\sup_{j \in \N} \bigg\| \frac{1}{N} \sum_{k = 1}^N
\Big( v_k^2 - \E [v_k^2] \Big) \Psi_j
\bigg\|_{L^2_\o L_T^2 H^{s-1}_x} & 
\les_T \frac 1{N^\frac 12}, 
\label{PPN2a} \\
\bigg\| \frac{1}{N} \sum_{k = 1}^N 
\Big( \Psi_k v_k - \E [\Psi_k v_k] \Big) v_j
\bigg\|_{L^2_\o \A_{N, j} L_T^2 H^{s-1}_x} & \les_T \frac 1{N^\frac 12}, 
\label{PPN3} \\
 \bigg\| \frac{1}{N} \sum_{k = 1}^N
\Big( v_k^2 - \E [v_k^2] \Big) v_j
\bigg\|_{L^2_\o \A_{N, j}L_T^2 H^{s-1}_x} & \les_T \frac 1{N^\frac 12}, 
\label{PPN4} \\
\bigg\| \frac{1}{N} \sum_{k = 1}^N
\Big( v_k^2 - \E [v_k^2] \Big) \Psi_j
\bigg\|_{L^2_\o \A_{N, j}L_T^2 H^{s-1}_x} & \les_T \frac 1{N^\frac 12}, 
\label{PPN4a} 
\end{align}

\noi
uniformly in  $N \in \N$.
As a consequence,
for each $T > 0$, we have, 
for each fixed $j \in \N$, 
\begin{align*}
& \bigg\| \frac{1}{N} \sum_{k = 1}^N 
\Big( \Psi_k v_k - \E [\Psi_k v_k] \Big) v_j
\bigg\|_{ L_T^2 H^{s-1}_x} 
+ \bigg\| \frac{1}{N} \sum_{k = 1}^N
\Big( v_k^2 - \E [v_k^2] \Big) v_j
\bigg\|_{ L_T^2 H^{s-1}_x}\\
& 
\quad +   \bigg\| \frac{1}{N} \sum_{k = 1}^N
\Big( v_k^2 - \E [v_k^2] \Big) \Psi_j
\bigg\|_{L_T^2 H^{s-1}_x}
\too 0
\end{align*}

\noi
in probability
at the rate $N^{-\frac 12 }$
as $N \to \infty$
in the sense of Definition \ref{DEF:conv1}, 
and 
\begin{align*}
& \bigg\| \frac{1}{N} \sum_{k = 1}^N 
\Big( \Psi_k v_k - \E [\Psi_k v_k] \Big) v_j
\bigg\|_{ \A_{N, j} L_T^2 H^{s-1}_x} 
+ 
 \bigg\| \frac{1}{N} \sum_{k = 1}^N
\Big( v_k^2 - \E [v_k^2] \Big) v_j
\bigg\|_{ \A_{N, j}L_T^2 H^{s-1}_x} \\
& \quad
+ 
\bigg\| \frac{1}{N} \sum_{k = 1}^N
\Big( v_k^2 - \E [v_k^2] \Big) \Psi_j
\bigg\|_{ \A_{N, j}L_T^2 H^{s-1}_x} \too0
\end{align*}

\noi
in probability
at the rate $N^{-\frac 12 }$
as $N \to \infty$.

\end{lemma}

\begin{proof}
As in the proof of Lemma \ref{LEM:LLN1}, 
we focus on proving \eqref{PPN1}-\eqref{PPN4a}
 since the second claim on convergence in probability at the rate $N^{-\frac 12 }$
follows from 
\eqref{PPN1}-\eqref{PPN4a}
 with Chebyshev's inequality.

Since the claim follows from a slight modification of the proof of Lemma \ref{LEM:LLN4}, 
we keep the discussion brief.
We only consider \eqref{PPN4}  since \eqref{PPN1}, \eqref{PPN2}, 
\eqref{PPN2a}, \eqref{PPN3}, 
and \eqref{PPN4a} follow
from a similar consideration.
We first note that, by arguing as in  
 the proof of Lemma~\ref{LEM:LLN3}, 
the contributions from the case $k = j$
can be treated easily and thus we omit details.
In the following, 
we assume $k \ne j$.

Set
\begin{align}
Y_{k, j}= \Big( v_k^2 - \E [v_k^2] \Big) v_j
\label{PPN5}
\end{align}

\noi
Then, 
from the mutual independence of $\{(\Psi_k, v_k)\}_{k \in \N}$, 
we have
\begin{align}
\E[Y_{k, j}] = 
\E[Y_{k, j}\cdot Y_{k', j}]= 0
\label{PPN6}
\end{align}

\noi
for any $k \ne k'$ with $k, k' \ne j$;
see \eqref{PN5} and \eqref{PN6}.

Proceeding as in \eqref{LWP2} with 
 \eqref{PPN0x},  we have
\begin{align}
\begin{split}
\Big\| \|  v_k^2 v_j \|_{L_T^2 H_x^{s-1}} \Big\|_{L^2_\o}
&\les_T 
  \| v_k \|^2_{L^6_\o C_T  H_x^{s}}
   \| v_j \|_{L^6_\o C_T  H_x^{s}}\\
 & \les 1,  
\end{split}
\label{PPN7}
\end{align}

\noi
uniformly in $j, k \in \N$.
Similarly, 
a minor modification of \eqref{LNx2a} yields
\begin{align}
\begin{split}
\Big\| \|\E[  v_k^2 ]v_j \|_{L_T^2 H_x^{s-1}} \Big\|_{L^2_\o}
 & \les 1,  
\end{split}
\label{PPN8}
\end{align}

\noi
uniformly in $j, k \in \N$.
Then, the claim \eqref{PPN4} follows 
from proceeding as in the proof of Lemma~\ref{LEM:LLN4}
with \eqref{PPN5}, \eqref{PPN6}, \eqref{PPN7}, and \eqref{PPN8}.
We omit details.
\end{proof}

With Lemmas \ref{LEM:LLN4} and 
\ref{LEM:LLN5}, 
Theorem \ref{THM:conv1}\,(ii)
follows from 
a  modification of the argument presented in Subsection \ref{SUBSEC:5.2}.
More precisely, 
we need to prove the following two convergence results under the assumption \eqref{PN1}
together with 
the assumed convergence
in probability at the rate of $N^{-\frac 12 }$ for
initial data \eqref{IV5} and \eqref{IV6}:

\smallskip

\begin{itemize}
\item[(i)]

in Step 1, instead of \eqref{CN5a}, we
prove 
\begin{align}
\| \vec V_{N, j} \|_{X^s_N(T)}
=  \| \vec v_{N, j}  - \vec v_j \|_{X^s_N(T)}
\too 0  
\label{BZ0a}
\end{align}

\noi
in probability at the rate $N^{-\frac 12 }$ as $N \to \infty$
in the sense of Definition \ref{DEF:conv1}.

\medskip
\item[(ii)]
 in Step 2, 
instead  of \eqref{DN1}, 
we prove, for each {\it fixed}  $j \in \N$:
\begin{align}
\| \vec V_{N, j} \|_{C_T \H_x^s}
=  \| \vec v_{N, j}  - \vec v_j \|_{C_T \H_x^s}
\too 0  
\label{BZ0b}
\end{align}

\noi
in probability  at the rate $N^{-\frac 12 }$
as $N \to \infty$ 
in the sense of Definition \ref{DEF:conv1}.
\end{itemize}

\smallskip

In Step 1, 
we first need to replace
$\O_{2, N}(\eta)$ 
and 
$\O_{4, N}(\eta)$ in 
 \eqref{BD5}
and \eqref{BDX}, respectively, as follows.
Given  $K_1 \ge 1$,  
define $\wt \O_{2, N}(K_1)\subset \O$ by setting 
\begin{align}
\wt \O_{2, N}(K_1) = \big\{ \o \in \O: 
N^{\frac 12 } \|\vec  V_{N, j} (0)\|_{\A_N \H^s_x} 
\le 
K_1 \Big\}.
\label{BZ1a}
\end{align}

\noi
Then, it follows from 
\eqref{IV6}
that, given $\dl > 0$,  there exist 
$N_6 = N_6(\dl) \in \N$ and  $K_1  =K_1(\dl) \gg1$, 
 such that 
\begin{align}
\PP\big( (\wt \O_{2, N}(K_1))^c \big) < \frac \dl 4
\label{BZ1b}
\end{align}

\noi
for $N \ge N_6$, 
which replaces    
 \eqref{BD5a}.
Given  $K_2 \ge 1$,  
define $\wt \O_{4, N}(K_2)\subset \O$ by setting 
\begin{align}
\wt \O_{4, N}(K_2)
= \Big\{ \o \in \O:  
N^\frac 12 \|(\vv, \pmb \Psi) \|_{\wt \Y_N} \le K_2 
 \Big\}
\label{BZ1}
\end{align}

\noi
for $K_2 \ge 1$, where the $\wt\Y_N$-norm is defined by 
\begin{align*}
\|(\vv, \pmb \Psi) \|_{\wt \Y_N} 
&  =  \bigg\| \frac{1}{N} \sum_{k = 1}^N 
\wick{\Psi_k^2} \bigg\|_{ L_T^2 W_x^{- \eps, \infty}}
+  \bigg\| \frac 1N \sum_{k = 1}^N 
\wick{\Psi_k^2 \Psi_j} \bigg\|_{ \A_{N, j}L_T^2 W_x^{- \eps, \infty}}\\
& \quad + \bigg\| \frac{1}{N} \sum_{k = 1}^N 
\Big( \Psi_k v_k - \E [\Psi_k v_k] \Big)v_j  \bigg\|_{\A_{N, j}L_T^2
H_x^{s-1}}\\
& \quad 
+ 
\bigg\| \frac{1}{N} \sum_{k = 1}^N
\Big( v_k^2 - \E [v_k^2] \Big) v_j \bigg\|_{\A_{N, j}L_T^2H_x^{s-1}}\\
& \quad 
+ 
\bigg\| \frac{1}{N} \sum_{k = 1}^N
\Big( v_k^2 - \E [v_k^2] \Big) \Psi_j \bigg\|_{\A_{N, j}L_T^2H_x^{s-1}}\\
& \quad
+ \bigg\| \frac 1N \sum_{k = 1}^N \Big( v_k \!\wick{\Psi_k \Psi_j} - \E [v_k\Psi_k ] \Psi_j \Big) 
\bigg\|_{\A_{N, j}L_T^2 H_x^{s-1}}.
\end{align*}

\noi
Then, 
it follows  from Lemmas \ref{LEM:LLN1}, 
 \ref{LEM:LLN4}, and  \ref{LEM:LLN5}
 that, given $\dl > 0$,  there exists $K_2  =K_2(\dl) \gg1$, 
 independent of $N \in \N$, 
 such that 
\begin{align}
\PP\big( (\wt \O_{4, N}(K_2))^c \big) < \frac \dl 4, 
\label{BZ3}
\end{align}

\noi
for any $N \in \N$, 
which replaces    
 \eqref{BD9}.
Here, 
we used the fact that, 
in  
 Lemmas \ref{LEM:LLN1}, 
 \ref{LEM:LLN4}, and  \ref{LEM:LLN5}, 
 we proved the $L^2(\O)$-bounds,
 which allows us to obtain \eqref{BZ3}
for any $N \in \N$ 
(not only for $N \ge N_\dl$ for some $N_\dl \in \N$)
via Chebyshev's inequality.

In the following, we work on 
\begin{align*}
\wt \O_N  (R, K, \eta_0, \eta, K_1, K_2)=
\O_N  (R, K, \eta_0, \eta)
\cap \wt \O_{2, N}(K_1)
\cap \wt \O_{4, N}(K_2), 
\end{align*}

\noi
where 
$ \O_N  (R, K, \eta_0, \eta)$
is as  in 
\eqref{BD10}.
It follows from 
\eqref{BD11}, 
\eqref{BZ1b}, 
and  \eqref{BZ3}
that there exists $N_* = N_*(\eta_0, \eta, \dl) \in \N$ such that 
\begin{align}
\PP\big((\wt \O_N  (R, K, \eta_0, \eta, K_1, K_2))^c\big) <  \dl
\label{BZ5}
\end{align}

\noi
for any $N \ge N_*$.
Moreover, 
it follows from a slight modification of \eqref{CN7}, \eqref{CN8}, 
\eqref{CN9}, \eqref{CN10}, 
\eqref{CN12}, and \eqref{CN13} with \eqref{BZ1}
that 
there exist $C_1, C_2, C_3 > 0$ independent of $N \in \N$ such that 
\begin{align}
\begin{split}
N^\frac 12 \|\vec V_{N, j}\|_{X^s_{N}(J)}
& \le 
C_1 
N^\frac 12 \|\vec  V_{N, j} (t_0) \|_{\A_N \H^s_x} \\
& \quad 
+ C_2 |J|^\frac 12 
K 
\Big( \| \vec V_{N, k}\|_{X^s_{N}(J)}^2 
+ R^2  \Big)\cdot 
N^\frac 12  \|\vec  V_{N, j}\|_{X^s_{N}(J)}\\
& \quad 
+C_3 |J|^\frac 12 R  K_2
\end{split}
\label{BZ6}
\end{align}

\noi
for any $N \ge N_*$, 
where  $J = [t_0, t_1]\subset [0, T]$
such that $ |J|\le 1$, 
which replaces
\eqref{CN14}.

From 
 \eqref{CN21}, we have 
\begin{align*}
\|\vec  V_{N, j}\|_{X^s_{N}(J)}
\le \|\vec  V_{N, j}\|_{X^s_{N}([0, T])}
\le 1
\end{align*}

\noi
for any interval $J \subset [0, T]$.
Then, 
it follows from \eqref{BZ6} that 
there exists $\tau = \tau (R(\dl), K(\dl), K_2(\dl)) = \tau(\dl)\in (0, 1]$ such that 
\begin{align}
N^\frac 12 \|\vec V_{N, j}\|_{X^s_{N}(J)}
& \le 
2 C_1 
N^\frac 12 
\|\vec  V_{N, j} (t_0) \|_{\A_N \H^s_x} 
+2 C_3 , 
\label{BZ8}
\end{align}

\noi
provided that $|J| \le \tau$.
Compare \eqref{BZ8} with \eqref{CN16}.
Hence, by a slight modification of the iterative argument in 
\eqref{CN18a}  and 
\eqref{CN19} with  \eqref{CN20}, 
it follows from  \eqref{BZ8} with \eqref{BZ1a} that 
\begin{align}
\begin{split}
N^\frac 12 \|\vec V_{N, j}\|_{X^s_N([0, T])}
& \le
\sum_{\l = 0}^{[T/\tau]} 
N^{\frac 12} \|\vec V_{N, j}\|_{X^s_{N}(J_\l )}\\
& \le 
\bigg(\Big[\frac T\tau\Big]+1\bigg)\bigg( (2C_1)^{[\frac T\tau]+1} K_1 +  
2C_3  
\frac { (2C_1)^{[\frac T\tau]+1}}{2C_1 - 1}
\bigg). 
\end{split}
\label{BZ9}
\end{align}

\noi
Recalling that 
$\tau \in (0, 1]$ depends only on 
$R(\dl)$, $K(\dl)$, and $K_2(\dl)$, 
we conclude from \eqref{BZ9} with~\eqref{BZ5}
that the desired convergence 
\eqref{BZ0a} in probability at the rate $N^{-\frac 12}$
holds.

In Step 2, we need to show \eqref{BZ0b}
for given  $j \in \N$.
Using 
Lemmas \ref{LEM:LLN1}, 
\ref{LEM:LLN4},  and
\ref{LEM:LLN5}
for fixed $j \in \N$
together with the assumed convergence 
\eqref{IV5} in probability at the rate $N^{-\frac 12}$ of initial data, 
we can easily implement required modifications
as in Step 1 presented above
and thus we omit details.
See Subsection~\ref{SUBSEC:6.2}
for an example of an argument where we obtain a convergence  rate
(in a slightly different setting).
This concludes the proof of Theorem~\ref{THM:conv1}\,(ii).

\begin{remark}
\rm 
Let $T \gg 1$.
In Theorem \ref{THM:conv1}\,(ii), 
we only assume the existence of a solution 
$\{(u_j, \dt u_j)\}_{j \in \N}$
to the mean-field SdNLW \eqref{MF2}
on the time interval $[0, T]$ along with the higher moment assumption 
\eqref{IV6a}.
In particular, 
we did not assume existence
of a solution
$\{ (u_{N, j}, \dt  u_{N, j})\}_{j = 1}^N$
 to \HLS \eqref{NLW4} 
on the time interval $[0, T]$
(whose local-in-time existence
is guaranteed by Theorem \ref{THM:GWP1}\,(i)).
We point out that an iterative application
of the local-in-time convergence argument
together with 
the blowup alternative
\eqref{blow1}
guarantees existence of 
the solution $\{ (u_{N, j}, \dt  u_{N, j})\}_{j = 1}^N$
 to \HLS \eqref{NLW4} 
on the time interval $[0, T]$
for sufficiently large $N = N(\o, T)\gg 1$.
See
\cite{LLOT}
for  an argument, 
where 
existence 
(with a control on the norm)
of a solution to a limiting equation, 
say on a time interval $[0, T]$, 
guarantees existence
(and a control on the norm)
of converging solutions on the same interval.

\end{remark}

\section{Mean-field limit of invariant Gibbs dynamics} 
\label{SEC:6}

In this section, 
we present a proof of 
Theorem~\ref{THM:conv2}\,(iii)
on convergence of the invariant Gibbs dynamics for 
\HLS
\eqref{SdNLW1}
to the limiting invariant dynamics 
for the mean-field SdNLW \eqref{MFx}, 
given by \eqref{phi1}.
In Subsection \ref{SUBSEC:6.1}, 
we go over the construction of the initial data set
on an appropriate probability space $(\O', \F', \PP')$.
In Subsection \ref{SUBSEC:6.2}, 
we then establish convergence of the 
invariant Gibbs dynamics for \HLS 
\eqref{SdNLW1}.

\subsection{On the Gibbsian initial data}
\label{SUBSEC:6.1}

In this subsection, 
we go over the construction of the initial data set
stated in 
Theorem \ref{THM:conv2}\,(iii).
This construction follows from the following proposition
which  couples  the Gibbsian initial data with the independent Gaussian initial data
in a strong manner, 
yielding  
 the $N^{- \frac 12}$-control 
\eqref{rand2}; see also \eqref{randx}. This  is crucial in establishing the dynamical convergence 
at the rate $N^{-\frac 12}$ claimed in 
Theorem \ref{THM:conv2}\,(iii.c).


\begin{proposition}
\label{PROP:rand}
Let $m > 0$.
 Then, there exist a probability space $(\O_1, \F_1, \PP_1)$ and random variables 
 $\big\{ \phi^{(0)}_{j} \big\}_{j \in \N}$ and 
 $\vec \psi_N = \{\psi_{N, j}\}_{j = 1}^N$, $N \in \N$,  on 
 $(\O_1, \F_1, \PP_1)$ such that 
 the family $\big\{ \phi^{(0)}_j \big\}_{j \in \N}$ is independent
 and 
\begin{align}
&\Law \big(\phi^{(0)}_j\big) = \mu_1 \ \  \text{for all } j \in \N 
\quad \text{and}\quad 
\Law (\vec \psi_N )= \rho_N ,
\label{rand1}
\end{align}

\noi
where $\mu_1$
is the massive Gaussian free field with mass $m > 0$
defined  in \eqref{gauss0} 
and $\rho_N$
is  the $O(N)$ linear sigma model 
defined in \eqref{Gibbsx}.
Furthermore, we have
\begin{align}
\sup_{j =1, \dots, N}\Big\|  \big\| \psi_{N, j} - \phi^{(0)}_j \big\|_{H^1_x}\Big\|_{L^2(\O_1)}
 \les N^{-\frac 12 } .
\label{rand2}
\end{align}

\end{proposition}

By assuming 
 Proposition \ref{PROP:rand}, 
we first carry out the construction of 
the initial data sets stated in Theorem \ref{THM:conv2}\,(iii.a) and (iii.b).

By extending the probability space
$(\O_1, \F_1, \PP_1)$ constructed in Proposition \ref{PROP:rand}, 
let $\{ \phi^{(1)}_j \}_{j \in \N}$ be a family of independent random variables 
on $\O_1$ 
such that 
$\Law \big(\phi^{(1)}_j\big) = \mu_0$ for each $j \in \N$, 
where $\mu_0$ is  the white noise measure on $\T^2$ defined in \eqref{gauss1}
with $s = 0$.
Given $N \in \N$, we set
 \[\vec \phi^{(0)}_N =
\{ \phi^{(0)}_{N, j}\}_{j = 1}^N
\deff  
\{ \phi^{(0)}_{j}\}_{j = 1}^N 
\qquad \text{and}\qquad
\vec \phi^{(1)}_N =
\{ \phi^{(1)}_{N, j}\}_{j = 1}^N
\deff  
\{ \phi^{(1)}_{j}\}_{j = 1}^N , 
\]

\noi
where  $\{ \phi^{(0)}_{j} \}_{j \in \N}$ is as in Proposition \ref{PROP:rand}.
For each $N \in \N$, we  define $\vec v^{(0)}_N= \{v^{(0)}_{N, j}\}_{j = 1}^N$ by 
setting 
\begin{align}
v^{(0)}_{N, j} = \psi_{N, j}- \phi^{(0)}_{N, j}
= \psi_{N, j}- \phi^{(0)}_{j}, \qquad j = 1, \dots, N, 
\label{rand3}
\end{align}

\noi
where  $\vec \psi_N = \{\psi_{N, j}\}_{j = 1}^N$
is as in Proposition \ref{PROP:rand}.
Then, from \eqref{rand2}, we have
\begin{align*}
\sup_{j =1, \dots, N}\Big\|  \big\|v^{(0)}_{N, j} \big\|_{H^1_x}\Big\|_{L^2(\O_1)}
 \les N^{-\frac 12 } , 
\end{align*}

\noi
yielding \eqref{randx}.

Let $\{\xi_j \}_{j \in \N}$
be a family of independent space-time white noises
on $\R_+\times \T^2$
which is also independent 
from 
$\{ \phi^{(0)}_j \}_{j \in \N}$, 
$\{ \phi^{(1)}_j \}_{j \in \N}$, 
and $\{\vec \psi_N\}_{N \in \N}$.
Namely, we may assume that 
$\xi_j$
is a $\D'(\R_+\times \T^2)$-valued random variable
defined on another probability space
$(\O_2, \F_2, \PP_2)$.
Then, 
by setting $\PP' = \PP_1 \otimes \PP_2$
and $\O' = \O_1 \times \O_2$, 
we complete
the construction of 
the initial data sets stated in Theorem \ref{THM:conv2}\,(iii).

\medskip

We now turn to  a proof of Proposition \ref{PROP:rand}.
Let us first introduce some notations
and preliminary lemmas.
 Let $\nu_1$ and $\nu_2$ be two probability measures on a Banach space $B$ and let $H \subset B$ be a closed subspace. We define the Wasserstein-2 distance of $\nu_1$ and $\nu_2$ with respect to $H$ 
 by setting
 \begin{align}
\Big(W_H (\nu_1, \nu_2)\Big)^2 = \inf_{X \sim \nu_1, Y \sim \nu_2} \E \Big[ \| X - Y \|_H^2 \Big],
\label{Ruo1}
\end{align}

\noi
where $X \sim \nu_1$ and $Y \sim \nu_2$ mean that $X$ and $Y$ are $B$-valued random variables 
with $\Law(X) = \nu_1$ and $\Law(Y) = \nu_2$, 
and the expectation is taken with respect 
to the probability measure associated with the coupling of $X$ and $Y$.
Equivalently, it can be written as
\begin{align}
\Big(W_H (\nu_1, \nu_2)\Big)^2
= \inf_{\plan\in\Pi(\nu_1, \nu_2)} 
\int_{B \times B}  \|x-y\|_H^2 \,  d\plan (x, y), 
\label{Ruo2}
\end{align}

\noi
where  $\Pi(\nu_1, \nu_2)$ is the set of probability measures $\plan$
on $B \times B$
whose first and second marginals are given by $\nu_1$ and $\nu_2$, 
namely, 
\begin{align}
 \int_{y \in B}  d\plan (x, y) = d\nu_1(x) \qquad \text{and}
\qquad 
 \int_{x \in B}  d\plan (x, y) = d\nu_2(y).
\label{Ruo2a}
\end{align}

\noi
See, for example, 
\cite{FU, FSS}, 
where 
they take $(B, H, \mu)$ to be an abstract Wiener space (for a Gaussian measure
$\mu$ with its Cameron-Martin space given by $H$)
and $\nu_j$, $j = 1, 2$, to be probability measures on $B$.
In this case, the infimum in \eqref{Ruo1} and \eqref{Ruo2} is indeed attained.

In our setting, we take
$B = \big(H^{-\eps} (\T^2)\big)^{\otimes N}$
and consider two probability measures:
the Gaussian measure $\muN$, where $\mu_1$
is the massive Gaussian free field with mass $m > 0$
defined in \eqref{gauss0}, 
and
the coupled $\Phi^4_2$-measure
$\rho_N$ in \eqref{Gibbsx}.
Then, we take 
 $H$ to be its Cameron-Martin space
 of $\muN$, 
 namely, 
$H = \big(H^1 (\T^2)\big)^{\otimes N}$
endowed with the norm:
\begin{align*}
\| \vec f \|_{H^1 (\T^2)^{\otimes N}}^2 =  \sum_{j = 1}^N \| f_j \|_{H^1 (\T^2)}^2
\end{align*}

\noi
for $\vec f = (f_1, \dots, f_N)$.

We now state a crucial lemma
from a recent work by 
Delgadino and Smith 
\cite[Theorem 2.6]{DS}.

\begin{lemma}
\label{LEM:DS}
Let $m > 0$.
Given $N \in \N$, let $\muN$ and $\rho_N$ be as above.
Then, we have 
\begin{align*}
\sup_{N \in \N} W_{(H^1 (\T^2))^{\otimes N}} \big(\muN, \rho_N\big) < \infty.
\end{align*}
\end{lemma}

Lemma \ref{LEM:DS} plays a key role
in studying 
the limiting behavior of  the coupled $\Phi^4_2$-measure
$\rho_N$ in \eqref{Gibbsx}
as $N \to \infty$.
It was first proven by Shen, Smith, Zhu, and Zhu
\cite[Theorem 5.11; see also (5.14)]{SSZZ}
for large masses $m \gg 1$.
In a recent work \cite{DS}, 
Delgadino and Smith 
extended this result to 
arbitrarily small masses $m > 0$
by making use of 
Talagrand's inequality
\cite{Tala, OV, FU}.

By using 
 Lemma~\ref{LEM:DS} 
and exploiting the joint exchangeability of the coupling, 
we first establish  the following lemma.

\begin{lemma}
\label{LEM:exch}
Let $m > 0$. Given $N \in \N$, let $\muN$ and $\rho_N$ be as above. 
Then, for each $N \in \N$, there exist random vectors 
$\vec \phi_N = \{ \phi_{N, j} \}_{j = 1}^N \sim \muN$ and 
$\vec \psi_N =  \{ \psi_{N, j} \}_{j = 1}^N \sim \rho_N$ such that
\begin{align}
\E \Big[ \| \phi_{N, j} - \psi_{N, j} \|_{H^1 (\T^2)}^2 \Big] \les N^{-1}
\label{Ruoj}
\end{align}

\noi
uniformly in $N \in \N$ and $j = 1, \dots, N$.
\end{lemma}

\begin{proof}
Given $N \in \N$, we let $B_N = \big(H^{-\eps} (\T^2)\big)^{\otimes N}$
and
let $S_N$ be the symmetric group on $\{1, \dots, N\}$. For each $\s \in S_N$, we define the permutation operator $P_\s:  B_N \to B_N$ by 
setting
\[P_\s\big( (f_1, \dots, f_N)\big) =  (f_{\s (1)}, \dots, f_{\s (N)}).\]

From Lemma~\ref{LEM:DS}, we know that for each $N \in \N$, there exists $\g_N \in \Pi (\muN, \rho_N)$ such that
\begin{align}
    \int_{B_N \times B_N} \| \vec f - \vec g \|_{H^1 (\T^2)^{\otimes N}}^2 d \g_N (\vec f, \vec g) \les 1
\label{gN1}
\end{align}

\noi
uniformly in $N\in \N$. Let us define
$ \cj \g_N$ by 
\begin{align}
\cj \g_N =  \frac{1}{N!} \sum_{\s \in S_N} (P_\s \otimes P_\s)_\# \g_N ,
\label{gN1a}
\end{align}

\noi
where $(P_\s \otimes P_\s)_\# \g_N$ denotes the pushforward measure of $\g_N$ 
under the map $P_\s \otimes P_\s$. Note that $\cj \g_N$ is jointly exchangeable in the sense that
\begin{align}
    (P_\s \otimes P_\s)_\# \cj \g_N = \cj \g_N .
\label{exch}
\end{align}

We claim that (i)~\eqref{gN1} holds with $\g_N$ replaced by $\cj \g_N $ and 
(ii)~$\cj \g_N \in \Pi (\mu_1^{\otimes N}, \rho_N)$. 
We first verify~(i).
 By \eqref{gN1a},  a change of variables, the symmetry 
 (in components) of the $H^1 (\T^2)^{\otimes N}$-norm, and \eqref{gN1}, we have
\begin{align}
\begin{split}
    &\int_{B_N \times B_N} \| \vec f - \vec g \|_{H^1 (\T^2)^{\otimes N}}^2 d \cj \g_N (\vec f, \vec g) \\
    &\quad= \frac{1}{N!} \sum_{\s \in S_N} \int_{B_N \times B_N} \| P_\s \vec f - P_\s \vec g \|_{H^1 (\T^2)^{\otimes N}}^2 d \g_N (\vec f, \vec g) \\
&\quad= \int_{B_N \times B_N} \| \vec f - \vec g \|_{H^1 (\T^2)^{\otimes N}}^2 d \g_N (\vec f, \vec g) \\
&\quad\les 1,
\end{split}
\label{gN2}
\end{align}

\noi
as desired. 
Next, we verify (ii). 
Let $\pi_1, \pi_2 : B_N \times B_N \to B_N$ be the projections onto the first and the second component, respectively: $\pi_1: (\vec f, \vec g) \mapsto \vec f$ and $\pi_2: (\vec f, \vec g) \mapsto \vec g$. Then, we have
\begin{align}
\begin{split}
    (\pi_1)_\# \cj \g_N &= \frac{1}{N!} \sum_{\s \in S_N} (\pi_1)_\# (P_\s \otimes P_\s)_\# \g_N \\
    &= \frac{1}{N!} \sum_{\s \in S_N} (P_\s)_\# (\pi_1)_\# \g_N \\
    &= \frac{1}{N!} \sum_{\s \in S_N} (P_\s)_\# \muN \\
    &= \muN ,
\label{gN3a}
\end{split}
\end{align}

\noi
where  the last equality follows from invariance of  $\muN$  under permutations. 
Noting that  $\rho_N$ is also invariant under permutations, 
a similar computation yields
\begin{align}
    (\pi_2)_\# \cj \g_N = \rho_N .
\label{gN3b}
\end{align}

\noi
Hence, the claim (ii) follows from \eqref{gN3a} and \eqref{gN3b}.

It remains to take random vectors $\vec \phi_N = \{\phi_{N, j}\}_{j = 1}^N$ and $\vec \psi_N = \{\psi_{N, j}\}_{j = 1}^N$ such that $(\vec \phi_N, \vec \psi_N) \sim \cj \g_N$. From the claim (ii) above, we have $\vec \phi_N \sim \muN$ and $\vec \psi_N \sim \rho_N$. The desired uniform estimate~\eqref{Ruoj}  follows from \eqref{gN2} and the joint exchangeability property~\eqref{exch}.
\end{proof}

We point out that 
Proposition~\ref{PROP:rand}
does not immediately follow
from Lemma \ref{LEM:exch} 
since $\vec \phi_N$ depends on $N \in \N$.
In order to amend this issue, we need the following gluing lemma, which is an
infinite-dimensional generalization of \cite[Lemma~7.6]{Vill}. Given a measurable space $E$, we denote by $\mathcal{P} (E)$ the set of probability measures on $E$.

\begin{lemma}
\label{LEM:glue}
Let $I$ be an index set. Let $E$ and $E_j$,  $j \in I$,  be Polish spaces. 
Given
 $\mu \in \mathcal{P} (E)$, 
 suppose that for each $j \in I$,  
 $\nu_j$ is a probability measure on $E \times E_j$ with marginal $\mu$ on $E$
 in the sense of \eqref{Ruo2a}.
Then, there exists a probability measure $\pi \in \mathcal{P} \big( E \times \prod_{j \in I} E_j \big)$ with marginal $\nu_j$ on $E \times E_j$ for each $j \in I$.
\end{lemma}

\begin{proof}
Let $J \subset I$ be a finite set. By the disintegration-of-measure theorem 
(\cite[Theorem on p.\,148]{GM} and \cite[Proposition 452O]{Fre}), for each $j \in J$, there exists a measurable map $x \in E \mapsto \pi_j^x\in \mathcal{P} (E_j)$, uniquely determined $\mu$-almost everywhere, such that
\begin{align*}
\nu_j (A_j) = \int_E \pi_j^x (A_j^x) d \mu (x)
\end{align*}

\noi
for any measurable set $A_j \subset E \times E_j$, where
\begin{align*}
A_j^x = \{ y \in E_j : (x, y) \in A_j \}.
\end{align*}

\noi
We then define  $\pi_J \in \mathcal{P} (E \times \prod_{j \in J} E_j)$ by setting
\begin{align*}
\pi_J (A_J) =  \int_E \Big( \bigotimes_{j \in J} \pi_j^x \Big) (A_J^x) d \mu (x)
\end{align*}

\noi
for any measurable set $A_J \subset E \times \prod_{j \in J} E_j$, where
\begin{align*}
A_J^x = \Big\{ y \in \prod_{j \in J} E_j : (x, y) \in A_J \Big\}.
\end{align*}

\noi
It is not difficult to check that $\pi_J$ has marginal $\nu_j$ on $E \times E_j$ for each $j \in I$.
Furthermore, 
one can  easily see that the collection of the probability measures $\{ \pi_J : J \subset I \text{ is finite} \}$ is consistent. Then, by the Kolmogorov extension theorem
(\cite[Appendix D]{Bass}), there exists a measure $\pi \in \mathcal{P} (E \times \prod_{j \in I} E_j)$ that satisfies the desired properties.
\end{proof}

We also need the following lemma; see \cite[(v) on p.\,52]{Sri}.

\begin{lemma}
\label{LEM:polish}
A countable product of Polish spaces is Polish.
\end{lemma}

We are now ready to present a proof of Proposition~\ref{PROP:rand}.

\begin{proof}[Proof of Proposition~\ref{PROP:rand}]

It follows from 
Lemma~\ref{LEM:exch}
that, for each $N \in \N$, 
there exist
$\vec \phi_N =  \{\phi_{N, j}\}_{j = 1}^N \sim \muN$ 
and   $\vec \psi_N = \{\psi_{N, j}\}_{j = 1}^N \sim \rho_N$
such that  \eqref{Ruoj} holds.
For each $N \in \N$, we extend 
$\vec \phi_N$ to 
a sequence of independent random variables
$\vec \phi^\infty_{N} = \{\phi_{N, j}^\infty\}_{j \in \N}$
by setting
$\phi_{N, j}^\infty = \phi_{N, j}$ for $j = 1, \dots, N$
and $\Law(\phi_{N, j}^\infty) =  \mu_1$
for $j \ge  N+1$.
Then, 
$\vec \phi^\infty_N$
is an $\big(H^{- \eps} (\T^2)\big)^{\otimes \N}$-valued random variable
with the law $\mu_1^{\otimes \N}$. 
Note 
from  Lemma~\ref{LEM:polish}
that the space $\big(H^{- \eps} (\T^2)\big)^{\otimes \N}$ is a Polish space.

We now apply Lemma~\ref{LEM:glue} by taking $I = \N$, 
$E = \big(H^{-\eps} (\T^2)\big)^{\otimes \N}$, 
$E_N = \big(H^{- \eps} (\T^2)\big)^{\otimes N}$, 
  $\mu = \mu_1^{\otimes \N}$, 
and  $\nu_N = \Law (\vec \phi_N^\infty, \vec \psi_N)$ for each $N \in \N$.
By Lemma \ref{LEM:glue}, 
there exists 
 a probability measure $\pi$ on $E \times \prod_{N \in \N} E_N$ with marginal $\nu_N$ on $E \times E_N$ for each $N \in \N$.

 The probability measure $\pi$ then gives rise to random variables 
 $\vec \phi^{(0)} =  \big\{\phi^{(0)}_j\big\}_{j \in \N}$ on $E =\big( H^{-\eps} (\T^2)\big)^{\otimes \N}$ 
 and $\vec \psi_N = \{\psi_{N, j}\}_{j = 1}^N$ on $E_N = \big(H^{- \eps} (\T^2)\big)^{\otimes N}$ for each $N \in \N$, 
 satisfying \eqref{rand1}, 
 such that,  for each $N \in \N$, we have
\begin{align*}
\sup_{N \in \N}\E \Big[ \big\| \big\{\phi^{(0)}_j\big\}_{j = 1}^N - \vec \psi_N \big\|_{H^1 (\T^2)^{\otimes N}}^2 \Big] 
< \infty.
\end{align*}

\noi
Hence,   the desired bound \eqref{rand2} follows from  \eqref{Ruoj} 
in Lemma \ref{LEM:exch}.
\end{proof}

\subsection{Convergence of invariant Gibbs dynamics}
\label{SUBSEC:6.2}

In this subsection, we prove  convergence
of the invariant Gibbs dynamics for \HLS \eqref{SdNLW1}
claimed
in Theorem~\ref{THM:conv2}\,(iii.c).
Since the desired convergence  follows from 
a combination of the local well-posedness argument
(the proof of Proposition~\ref{PROP:LWP1})
presented 
in Subsection \ref{SUBSEC:3.1}
and the iterative argument
presented in Subsection~\ref{SUBSEC:5.2}, 
we keep our discussion brief here, 
indicating necessary modifications.
In the following, we fix
 $\frac 12 \le s < 1$
 and $T > 0$
 and often drop the dependence
on these parameters.

Let 
 $\big\{ \big(\phi^{(0)}_{j}, \phi^{(1)}_{j}\big)\big \}_{j \in \N}$ and 
 $\vec \psi_N = \{\psi_{N, j}\}_{j = 1}^N$, $N \in \N$,  
be as in  
 Theorem~\ref{THM:conv2}\,(iii)
 constructed in Subsection \ref{SUBSEC:6.1}, 
 and let 
 $\vec v^{(0)}_N= \big\{v^{(0)}_{N, j}\big\}_{j = 1}^N$
be as in \eqref{rand3}.
Let $\Phi_j$, $j \in \N$, be the solution to \eqref{phi1}.

By writing \eqref{NLW6} in the Duhamel formulation, we have 
\begin{align}
\begin{split}
v_{N, j} 
& = S(t) \big( v_{N, j}^{(0)}, 0 \big) 
- \frac{1}{N} \sum_{k = 1}^N  \I\Big( 
v_{N, k}^2 v_{N, j}
+ 2 \Phi_k v_{N, k} v_{N, j}
  + v_{N, k}^2 \Phi_j \\
& 
\hphantom{XXXXXXXXXXX}
 + \wick{\Phi_k^2} v_{N, j} + 2 v_{N, k} \wick{ \Phi_k \Phi_j } 
+ \wick{\Phi_k^2 \Phi_j} \Big), 
\end{split}
\label{KN1}
\end{align}

\noi
where $S (t)$ and $\I$ are as in  \eqref{D4} and  \eqref{D5}, 
respectively.
Denote by $u_j$ the limit
of $u_{N, j} = \Phi_j + v_{N, j}$.
As mentioned in Subsection \ref{SUBSEC:1.3}, 
the (formal) limit $u_j$ solves \eqref{MFx} 
with  initial data $\big(\phi^{(0)}_j, \phi^{(1)}_j\big)$.
Hence, recalling  
 that 
 $u_j \equiv \Phi_j$
is a solution to \eqref{MFx}
with  $(u_j, \dt u_j)|_{t = 0} = \big(\phi^{(0)}_j, \phi^{(1)}_j\big)$
and 
that the uniqueness
of $(u_j, \dt u_j)$ holds  in the class \eqref{class4}, 
we conclude that 
$(u_j, \dt u_j)  = (\Phi_j, \dt \Phi_j)$.
Namely, we have $v_j = u_j  - \Phi_j\equiv 0$ for any $j \in \N$, 
which in particular (and trivially) implies
\begin{align*}
(v_j, \dt v_j) = (0, 0)  \in 
L^p(\O; C([0, T]; \H^s(\T^2)))
\end{align*}

\noi
for any $p \ge 1$
and $T > 0$.
This  is the main difference and the simplification
as compared to Subsection \ref{SUBSEC:5.2}.

\medskip

\noi
$\bullet$ {\bf Step 1:}
In this first step, we  establish the following convergence
for each $T \gg 1$:
\begin{align}
\| \vec v_{N, j} \|_{X^s_N(T)}
\too 0  
\label{KN1a}
\end{align}

\noi
in probability (with respect to $\PP'$) 
at the rate $N^{-\frac 12 }$ 
as $N \to \infty$
in the sense of Definition \ref{DEF:conv1}, 
where $X^s_N(T)$ is as in \eqref{XN1}.

Fix a target time $T \gg 1$
and small $\dl > 0$.
Let $\frac 12 \le s < 1$.
In the following, we suppress dependence
on $T$,  $s$, and  small $\eps > 0$.
Given  $K_1 \ge 1$, define the set $\O'_{1, N}(K_1)\subset \O'$
by setting
\begin{align}
\O'_{1, N} (K_1)
= \Big\{\o' \in \O':
N^{\frac 12} \|  v^{(0)}_{N, j} \|_{\A_N H_x^s} \le K_1 \Big\}, 
\label{KN2}
\end{align}

\noi
where $\A_N B$ is as in \eqref{AN0}.
Then, it follows from  \eqref{randx} and Chebyshev's  inequality
 that 
there exists $K_1 = K_1(\dl) \gg1 $,
independent of $N \in \N$,  such that 
\begin{align}
\PP'\big( (\O'_{1, N}(K_1))^c \big) < \frac \dl 6
\label{KN3}
\end{align}

\noi
for any $N \in \N$.
Next, given $K_2 \ge1$ define 
the sets $\O'_{2, N}(K_2)\subset \O'$
by setting
\begin{align}
\O'_{2, N} (K_2)
= \Big\{\o' \in \O':
 \|  \pmb \Phi_N \|_{\ZZ^\eps_N (T)} \le K_2\Big\}, 
\label{KN4}
\end{align}

\noi
where 
the 
$\ZZ^\eps_N(T)$-norm  is as in  \eqref{X1}
and \eqref{X1a}, 
and $\pmb \Phi_N$ is given by 
\begin{align*}
\pmb \Phi_N = \big\{
\Phi_j,  \, \wick{\Phi_k \Phi_j} \, , 
 \, 
\vec 0_{j, k}  \big\}_{j, k = 1}^N
\end{align*}

\noi
with $\vec 0_{j, k} = 0$ for any $j, k = 1, \dots, N$.
\noi
Then, it follows from Lemma \ref{LEM:sto1}\,(iii)
that 
there exists $K_2 = K_2(\dl) \gg1 $, independent of $N \in \N$,  such that 
\begin{align}
\PP'\big((\O'_{2, N}(K_2))^c \big) < \frac \dl 6
\label{KN5}
\end{align}

\noi
for any $N \in \N$.
Given  $K_3 \ge 1$,   define 
the set $\O'_{3, N}(K_3)\subset \O'$
by setting
\begin{align}
\O'_{3, N}(K_3)
& = \Bigg\{\o' \in \O': 
N^{\frac 12}
\bigg\| \frac{1}{N} \sum_{k = 1}^N \wick{\Phi_k^2 \Phi_j} 
\bigg\|_{\A_{N, j} L_T^2 W_x^{- \eps, \infty}} \leq K_3 \Bigg\}.
\label{KN6}
\end{align}

\noi
Then, it follows from 
 Lemma \ref{LEM:LLN1}
that 
 there exists $K_3  = K_3(\dl) \gg1$, 
 independent of $N \in \N$, 
 such that 
\begin{align}
\PP'\big( (\O'_{3, N}(K_3))^c \big) < \frac \dl 6
\label{KN7}
\end{align}

\noi
for any $N \in \N$.
In the following, we work on 
$\O'_N ( K_1, K_2, K_3)$ defined by 
\begin{align}
\O_N'  ( K_1, K_2, K_3)= \bigcap_{m = 1}^3\O_{m, N}'(K_m).
\label{KN8}
\end{align}

\noi
Given small $\dl > 0$, 
by choosing 
$K_m = K_m(\dl)$ sufficiently large, $m = 1, 2, 3$, 
 it follows from 
\eqref{KN8}
with  \eqref{KN3}, \eqref{KN5},  and \eqref{KN7}
that 
\begin{align}
\PP'\big((\O_N'  ( K_1, K_2, K_3))^c\big) < \frac \dl 2
\label{KN9}
\end{align}

\noi
for any $N \in \N$.

\medskip

We now consider \eqref{KN1}
on an interval $J = [t_0, t_1] \subset [0, T]$ with $|J| \le 1$, 
 starting at time $t_0$.
By  applying a slight modification of 
\eqref{LWP2}-\eqref{LWP7}, 
where $\pmb Z_N = \big\{Z_j^{(1)},  Z_{k, j}^{(2)}, Z_{k, j}^{(3)}\big\}_{j, k = 1}^N$
are as in~\eqref{X3}, 
with  \eqref{KN4} and \eqref{KN6}, 
 we obtain
\begin{align}
\begin{split}
 N^{\frac 12 }  \|\vec v_{N, j}\|_{X^s_{N}(J)}
& 
\le 
C_1 
N^{\frac 12} \|\vec  v_{N, j} (t_0) \|_{\A_N \H^s_x} \\
& \quad 
+ C_2 |J|
\Big( \| \vec v_{N, k}\|_{X^s_{N}(J)}^2
+ K_2^2  \Big)
N^{\frac 12 } \|\vec  v_{N, j}\|_{X^s_{N}(J)}\\
& \quad 
+C_3 |J|^\frac 12 K_3, 
\end{split}
\label{KN10}
\end{align}

\noi
where
the $X^s_{N}(J)$-norm is as in 
\eqref{CN5b} 
and the constants $C_1, C_2, C_3 > 0$ are independent of $N\in \N$.
Suppose that 
\begin{align}
\|\vec  v_{N, j}\|_{X^s_{N}(J)}
\le 1
\label{KN11}
\end{align}

\noi
for some interval $J \subset [0, T]$.
Then, 
it follows from \eqref{KN10} that 
there exists $\tau = \tau (K_2(\dl), K_3(\dl) ) = \tau(\dl) \in (0, 1]$ such that 
\begin{align}
\begin{split}
N^{\frac 12}\|\vec v_{N, j}\|_{X^s_{N}(J)}
& \le 
2 C_1 
N^{\frac 12}\|\vec  v_{N, j} (t_0) \|_{\A_N \H^s_x} 
+2 C_3 , 
\end{split}
\label{KN12}
\end{align}

\noi
provided that $|J| \le \tau$.

Given small $\eta_0, \eta > 0$, 
it follows from 
\eqref{KN2}
that there exists $N_1 = N_1(K_1, \eta_0, \eta) \in \N $
such that 
\begin{align}
\|\vec  v_{N, j} (0) \|_{\A_N \H^s_x} 
= \| (v_{N, j}^{(0)}, 0 ) \|_{\A_N \H^s_x} 
\le N^{-\frac 12} K_1 \le \eta_0
\quad \text{and}
\quad 
N^{-\frac 12} \le \eta
\label{KN13}
\end{align}

\noi
for any $N \ge N_1$.
In particular, 
from \eqref{KN12} and \eqref{KN13}, we have
\begin{align}
\begin{split}
\|\vec v_{N, j}\|_{X^s_{N}(J)}
& \le 
2 C_1 
\|\vec  v_{N, j} (t_0) \|_{\A_N \H^s_x} 
+2 C_3 \eta , 
\end{split}
\label{KN14}
\end{align}

\noi
provided that $|J| \le \tau$.
Compare \eqref{KN14} with \eqref{CN16}.
Hence, the iterative argument
at the end of Step 1
of the proof of Theorem \ref{THM:conv1}\,(i)
presented in Subsection \ref{SUBSEC:5.2}
(see 
\eqref{CN18a},
\eqref{CN19}, and \eqref{CN21})
with \eqref{KN13}
yields
\begin{align*}
\|\vec  v_{N, j}\|_{X^s_{N}(J)}
\le \|\vec  v_{N, j}\|_{X^s_{N}([0, T])}
\le 1
\end{align*}

\noi
for any interval $J \subset [0, T]$
with  $|J| \le \tau$, 
guaranteeing 
\eqref{KN11}
for any such interval $J$.
Hence, by 
using \eqref{KN12} with \eqref{KN2} 
and applying 
a slight modification of the iterative argument in 
\eqref{CN18a}  and 
\eqref{CN19} with  \eqref{CN20}
as in \eqref{BZ9}, 
we conclude that 
there exists 
$L = L(T, K_1, \tau) = L(T, \dl) \ge 1$
such that 
\begin{align}
\begin{split}
N^\frac 12 \|\vec v_{N, j}\|_{X^s_N([0, T])}
& \le L
\end{split}
\label{KN15}
\end{align}

\noi
for any $N \ge N_1$.
Finally, 
 the desired convergence 
\eqref{KN1a} in probability at the rate $N^{-\frac 12}$
follows from \eqref{KN9} and \eqref{KN15}.

\medskip

\noi
$\bullet$ {\bf Step 2:}
In the following, we briefly discuss
how to obtain
 the following convergence
in probability
at the rate $N^{-\frac 12}$ for each {\it fixed}  $j \in \N$:
\begin{align}
 \| \vec v_{N, j} \|_{C_T \H_x^s}
\too 0  
\label{JN1}
\end{align}

\noi
as $N \to \infty$.

We first go over the construction of data sets indexed by $N \gg 1$.
Fix $j \in \N$.
Given $K_4 \ge 1$, define the set $\O'_{4, j, N}(K_4)\subset \O'$
by setting
\begin{align}
\O'_{4, j, N} (K_4)
= \Big\{\o' \in \O':
N^{\frac 12} \|  v^{(0)}_{N, j} \|_{ H_x^s} \le K_4\Big\}.
\label{JN2}
\end{align}

\noi
Then, it follows from  \eqref{randx} and Chebyshev's  inequality
 that 
 there exists $K_4 = K_4(\dl) \gg1 $,
independent of $j, N \in \N$,  such that 
\begin{align}
\PP'\big( (\O'_{4,j,  N}(K_4))^c \big) < \frac \dl 8
\label{JN3}
\end{align}

\noi
for any $N \in \N$.
Given $K_5 \ge 1$, 
define the set $\O'_{5, N}(K_5)\subset \O'$
by setting
\begin{align}
\O'_{5, N} (K_5)= \Big\{\o' \in \O': 
N^{\frac 12}
\| \vec v_{N, k}\|_{X^s_{N}(T)} \le K_5\Big\}.
\label{JN4}
\end{align}

\noi
Then, it follows 
from the convergence-in-probability \eqref{KN1a}
established in
Step 1 that, given $\dl > 0$, 
 there exists $K_5 = K_5(\dl) \gg1 $
 such that 
\begin{align}
\PP'\big( (\O'_{5, N}(K_5))^c \big) < \frac \dl 8
\label{JN5}
\end{align}

\noi
for any $N \ge N_1 = N_1(T, \dl)$, 
where $N_1$ is as in \eqref{KN13}.
Given  $K_6 \ge 1 $, define the set $\O'_{6, j, N}(K_6) \subset \O'$ by setting 
\begin{align}
\O'_{6, j, N}(K_6) = \Big\{ \o' \in \O': 
 \| \Phi_j \|_{C_T W^{-\eps, \infty}_x} 
 + 
 \| \wick{\Phi_k\Phi_j} \|_{\A_{N, k}C_T W^{-\eps, \infty}_x}  \le K_6\Big\}.
\label{JN6}
\end{align}

\noi
Then, it follows from Lemma \ref{LEM:sto1}\,(i) 
(for the uniform (in $j, k$) moment bound
on the $C_T W^{-\eps, \infty}_x$-norm of $\wick{\Phi_k\Phi_j}$)
and  Chebyshev's  inequality
that 
there exists $K_6  = K_6(\dl) \gg1 $, independent of $j, N \in \N$,  such that 
\begin{align}
\PP'\big( (\O'_{6, j, N}(K_6))^c \big) < \frac \dl {8} 
\label{JN7}
\end{align}

\noi
for any $N \in \N$.
Lastly, 
given $K_7 \ge 1$, define 
the set $\O'_{7, j, N}(K_7)\subset \O'$
by setting
\begin{align}
\O'_{7, j, N}(K_7)
& = \Bigg\{\o' \in \O': 
N^{\frac 12 }
\bigg\| \frac{1}{N} \sum_{k = 1}^N \wick{\Phi_k^2 \Phi_j} 
\bigg\|_{ L_T^2 W_x^{- \eps, \infty}} \leq K_7 \Bigg\}.
\label{JN8}
\end{align}

\noi
Then, it follows from 
 Lemma \ref{LEM:LLN1}
that 
 there exists $K_7 = K_7(\dl) \gg1 $,
independent of $j, N \in \N$,  such that 
\begin{align}
\PP'\big( (\O'_{7, j, N}(K_7))^c \big) < \frac \dl 8
\label{JN9}
\end{align}

\noi
for any $N \in \N$.
In the following, we work on 
$\O'_{j, N} (K_1, \dots, K_7)$ defined by 
\begin{align}
\begin{split}
\O_{j, N}'  (K_1,\dots,  K_7)
& = 
\O_{N}'   (K_1, K_2, K_3)\cap 
\O_{4, j, N}'(K_4)\\
& \quad \cap  \O_{5, N}' (K_5) 
\cap \O_{6, j, N}' (K_6)
 \cap \O_{7, j,  N}' (K_7) 
\end{split}
\label{JN10}
\end{align}

\noi
where 
$\O_{ N}'  (K_1, K_2, K_3)$  is as in 
\eqref{KN8}.
Given small $\dl > 0$, 
by choosing 
  $K_m = K_m(\dl)$ sufficiently large, $m = 1, \dots,  7$, 
 it follows from 
\eqref{JN10}
with  \eqref{KN9}, \eqref{JN3}, 
\eqref{JN5}, \eqref{JN7},
and \eqref{JN9}
that 
\begin{align}
\PP'\big((\O_{j, N}'  ( K_1, \dots, K_7))^c\big) <  \dl 
\label{JN11}
\end{align}

\noi
for any $N \ge N_1$, 
where $N_1$ is as in \eqref{KN13}.

For fixed $j \in \N$, 
we now consider \eqref{KN1}
on an interval $J = [t_0, t_1]\subset [0, T]$ with $|J| \le 1$, 
 starting at time $t_0$.
By  applying a slight modification of 
\eqref{LWP2}-\eqref{LWP7}
(without taking the $\l^2$-average in $j$;
see also the estimates
in  Step~2
of the proof of Theorem \ref{THM:conv1}\,(i) presented in Subsection \ref{SUBSEC:5.2})
with  \eqref{KN4}, \eqref{JN4}, 
\eqref{JN6},
 and \eqref{JN8},
 we have 
\begin{align*}
\begin{split}
N^{\frac 12} \|\vec v_{N, j}\|_{C_J \H_x^s}
& \le 
C_4 
N^{\frac 12} \|\vec  v_{N, j} (t_0) \|_{\H^s}\\
& \quad 
+ C_5 |J|
\Big( 
 K^2_2  + K_5^2 \Big)
\cdot N^{\frac 12} \|\vec  v_{N, j}\|_{C_J \H_x^s}\\
& \quad 
+C_6 |J|^\frac 12 \Big(K_5^2 K_6 + K_7\Big) ,
\end{split}
\end{align*}

\noi
where the constants $C_4, C_5, C_6 > 0$ are independent of $N\in \N$.
Then, 
by choosing small $\tau = \tau (K_2(\dl), K_5(\dl), K_6(\dl), K_7(\dl) ) = \tau(\dl) > 0$, 
we obtain
\begin{align}
\begin{split}
N^{\frac 12}\|\vec v_{N, j}\|_{C_J\H^s_x}
& \le 
2 C_4
N^{\frac 12}\|\vec  v_{N, j} (t_0) \|_{ \H^s_x} 
+2 C_6 , 
\end{split}
\label{JN14}
\end{align}

\noi
provided that $|J| \le \tau$.
Hence, by iteratively applying \eqref{JN14}
with \eqref{JN2}
as in \eqref{BZ9}, 
we conclude that 
there exists 
$L = L(T,  \tau)  \ge 1$
such that 
\begin{align}
\begin{split}
N^\frac 12 \|\vec v_{N, j}\|_{C_T\H^s_x}
& \le L
\end{split}
\label{JN15}
\end{align}

\noi
for any $N \ge N_1$.
Therefore, 
 the desired convergence 
\eqref{JN1} in probability at the rate $N^{-\frac 12}$
follows from \eqref{JN11} and \eqref{JN15}.

This concludes the proof of Theorem \ref{THM:conv2}\,(iii).

\begin{ackno} \rm
The authors would  like to express their gratitude to the anonymous
referees for the helpful comments which improved the quality of the paper.
They would  like to thank Leonardo Tolomeo and Jiawei Li for helpful discussions. 
They are also grateful to Scott Smith
for telling them about the work~\cite{DS}
and 
 would also like to express their sincere gratitude
to Rongchan Zhu for asking a question on the convergence rate,
which led to an improvement 
of the main results.
S.L.~would like to thank the School of Mathematics at the University of Edinburgh for its hospitality, where this
manuscript was prepared.

R.L.~was funded by the Deutsche Forschungsgemeinschaft (DFG, German Research Foundation) -- Project-ID
211504053 -- SFB 1060.
R.L. and S.L. acknowledge support from the Deutsche Forschungsgemeinschaft (DFG, German Research Foundation) under Germany's Excellence Strategy -- EXC-2047/1 -- 390685813.
R.L. and S.L. were also funded by the Deutsche Forschungsgemeinschaft (DFG, German Research 
Foundation) -- Project-ID 539309657 -- SFB 1720.
T.O.~was supported by the European Research Council (grant no.~864138 ``SingStochDispDyn")
and  by the EPSRC 
Mathematical Sciences
Small Grant (grant no.~EP/Y033507/1).
T.O.~also
acknowledges support from  
the NSFC (grant no.~W2531005).

\end{ackno}


\begin{thebibliography}{99}









%
%


\bibitem{Bass}
R.~Bass, 
{\it Stochastic processes.}
Cambridge Series in Statistical and Probabilistic Mathematics, 33. Cambridge University Press, Cambridge, 2011. xvi+390 pp. 


\bibitem{BOP1}
\'A.~B\'enyi, T.~Oh, O.~Pocovnicu, 
{\it Wiener randomization on unbounded domains and an application to almost sure well-posedness of NLS}, Excursions in harmonic analysis. Vol. 4, 3--25, Appl. Numer. Harmon. Anal., Birkh\"auser/Springer, Cham, 2015. 

\bibitem{BOZ}
\'A.~B\'enyi, T.~Oh, T.~Zhao,
{\it Fractional Leibniz rule on the torus}, 
Proc. Amer. Math. Soc. 153 (2025), no. 1, 207--221.


\bibitem{BO94}
J.~Bourgain,
{\it Periodic nonlinear Schr\"odinger equation and invariant measures}, 
Comm. Math. Phys. 166 (1994), no. 1, 1--26.


\bibitem{BO96}
J.~Bourgain,
{\it Invariant measures for the 2D-defocusing nonlinear Schr\"odinger equation}, 
Comm. Math. Phys. 176 (1996), no. 2, 421--445.


\bibitem{Bring2}
B.~Bringmann,
{\it Invariant Gibbs measures for the three-dimensional wave equation with a Hartree nonlinearity II: dynamics}, 
J. Eur. Math. Soc. (JEMS) 26 (2024), no. 6, 1933--2089.


\bibitem{BDNY}
B.~Bringmann, Y.~Deng, A.~Nahmod, H.~Yue,
{\it Invariant Gibbs measures for the three dimensional cubic nonlinear wave equation}, 
Invent. Math. 236 (2024), no. 3, 1133--1411.



\bibitem{BS}
D.~Brydges, G.~Slade,
{\it 
Statistical mechanics of the 2-dimensional focusing nonlinear Schr\"odinger equation}, 
Comm. Math. Phys.
182 (1996), no. 2, 485--504.


\bibitem{BLL}
E.~Brun, G.~Li, R.~Liu,
{\it Global well-posedness of the energy-critical stochastic nonlinear wave equations}, 
J. Differential Equations 397 (2024), 316--348.


\bibitem{BT2}
N.~Burq, N.~Tzvetkov, 
{\it Probabilistic well-posedness for the cubic wave equation},  J. Eur. Math. Soc. 16 (2014), no. 1, 1--30. 






\bibitem{CKSTT1}
J.~Colliander, M.~Keel, G.~Staffilani, H.~Takaoka, T.~Tao, 
{\it  Almost conservation laws and global rough solutions to a nonlinear Schr\"odinger equation},  Math. Res. Lett. 9 (2002), no. 5-6, 659--682.




%
%
%




\bibitem{CLL}
A.~Czernik, R.~Liu, S.~Liu,
{\it On the Smoluchowski-Kramers approximation for the hyperbolic  
$O(N)$  linear sigma model and its mean-field limit}, 
arXiv:2606.15214 [math.AP].




\bibitem{DPD}
G.~Da Prato, A.~Debussche,
{\it Strong solutions to the stochastic quantization equations}, 
Ann. Probab. 31 (2003), no. 4, 1900--1916.


\bibitem{DPT}
G.~Da Prato, L.~Tubaro,
{\it Wick powers in stochastic PDEs: an introduction}, 
Quantum and stochastic mathematical physics, 1--15.
Springer Proc. Math. Stat., 377.


\bibitem{DS}
M.~Delgadino, S.A.~Smith,
{\it Mass generation
for the two dimensional $O(N)$
linear sigma model in the large $N$ limit}, 
arXiv:2601.19630 [math.PR].




%
%
%


\bibitem{Deya}
A.~Deya, 
{\it A nonlinear wave equation with fractional perturbation},
 Ann. Probab. 47 (2019), no. 3, 1775--1810.

%

\bibitem{FSS}
S.~Fang, J.~Shao, K.-T.~Sturm, 
{\it Wasserstein space over the Wiener space}, 
Probab. Theory Related Fields 146 (2010), no. 3-4, 535--565.


\bibitem{FU}
D.~Feyel, A.S.~\"Ust\"unel, 
{\it Monge-Kantorovitch measure transportation and Monge-Amp\`ere equation on Wiener space},
Probab. Theory Related Fields 128 (2004), no. 3, 347--385. 



\bibitem{Fre}
D.H.~Fremlin, 
{\it Measure theory. Vol. 4. Topological measure spaces.}
Measure theory. Vol. 4. Topological measure spaces. Part I, II. Corrected second printing of the 2003 original. Torres Fremlin, Colchester, 2006. Part I: 528 pp.; Part II: 439+19 pp. (errata).

\bibitem{FHI}
R.~Fukuizumi, M.~Hoshino, T.~Inui,
{\it Non relativistic and ultra relativistic limits in 2D stochastic nonlinear damped Klein-Gordon equation}, 
Nonlinearity 35 (2022), no. 6, 2878--2919.




\bibitem{GM}
S.~Graf, R.~D.~Mauldin,
{\it A classification of disintegrations of measures}, 
Measure and measurable dynamics (Rochester, NY, 1987), 147--158.
Contemp. Math., 94
American Mathematical Society, Providence, RI, 1989.

\bibitem{GOTT}
D.~Greco, T.~Oh, L.~Tao, L.~Tolomeo,
{\it  Critical threshold for weakly interacting log-correlated focusing Gibbs measures}, Proc. Amer. Math. Soc. Ser. B 12 (2025), 150--165. 

\bibitem{GKO}
M.~Gubinelli, H.~Koch, T.~Oh,
{\it Renormalization of the two-dimensional stochastic nonlinear wave equations}, 
Trans. Amer. Math. Soc. 370 (2018), no. 10, 7335--7359.


\bibitem{GKO2}
M.~Gubinelli, H.~Koch, T.~Oh,
{\it Paracontrolled approach to the three-dimensional stochastic nonlinear wave equation with quadratic nonlinearity}, 
J. Eur. Math. Soc. (JEMS) 26 (2024), no. 3, 817--874.


\bibitem{GKOT}
M.~Gubinelli, H.~Koch, T.~Oh, L.~Tolomeo,
{\it Global dynamics for the two-dimensional stochastic nonlinear wave equations}, 
Int. Math. Res. Not. IMRN. (2022), no. 21, 16954--16999. 


%



\bibitem{Kuo}
H.~Kuo, 
{\it Introduction to stochastic integration}.  Universitext. Springer, New York, 2006. xiv+278 pp.



\bibitem{LRS}
J.~Lebowitz, H.~Rose, E.~Speer, 
{\it Statistical mechanics of the nonlinear Schr\"odinger equation}, J. Statist. Phys. 50 (1988), no. 3-4, 657--687.


\bibitem{LT}
M.~Ledoux, M.~Talagrand,
{\it Probability in Banach spaces.
Isoperimetry and processes}.
Reprint of the 1991 edition. Classics in Mathematics. Springer-Verlag, Berlin, 2011. xii+480 pp. 



\bibitem{LLOT}
G.~Li, J.~Li, T.~Oh, N.~Tzvetkov, 
{\it Probabilistic well-posedness of dispersive PDEs beyond variance blowup I: Benjamin-Bona-Mahony equation}, 
arXiv:2509.02344 [math.AP].








\bibitem{McK}
H.~P.~McKean,
{\it Statistical mechanics of nonlinear wave equations. IV. Cubic Schr\"odinger}, 
Comm. Math. Phys. 168 (1995), no. 3, 479--491.
{\it Erratum: ``Statistical mechanics of nonlinear wave equations. IV. Cubic Schr\"odinger''}, 
Comm. Math. Phys. 173 (1995), no. 3, 675.


%
%



\bibitem{Nelson2}
E.~Nelson, 
{\it A quartic interaction in two dimensions}, 
 1966 Mathematical Theory of Elementary Particles (Proc. Conf., Dedham, Mass., 1965), pp. 69--73, M.I.T. Press, Cambridge, Mass.
%


\bibitem{Nua}
D.~Nualart,
{\it The Malliavin calculus and related topics}.
Second edition.
Probab. Appl. (N. Y.)
Springer-Verlag, Berlin, 2006. xiv+382 pp.


\bibitem{OO}
T.~Oh, M.~Okamoto,
{\it Comparing the stochastic nonlinear wave and heat equations: a case study}, 
Electron. J. Probab. 26 (2021), Paper No. 9, 44 pp.



\bibitem{OOR}
T.~Oh, M.~Okamoto, T.~Robert,
{\it A remark on triviality for the two-dimensional stochastic nonlinear wave equation},
 Stochastic Process. Appl. 130 (2020), no. 9, 5838--5864. 


\bibitem{OOT1}
T.~Oh, M.~Okamoto, L.~Tolomeo,
{\it Focusing $\Phi^4_3$-model with a Hartree-type nonlinearity}, 
Mem. Amer. Math. Soc. 304 (2024), no. 1529, vi+143 pp. 

\bibitem{OOT}
T.~Oh, M.~Okamoto, L.~Tolomeo,
{\it Stochastic quantization of the $\Phi^3_3$-model},  
Mem. Eur. Math. Soc., 16,
EMS Press, Berlin, 2025. viii+145 pp. 



\bibitem{OOTz}
T.~Oh, M.~Okamoto, N.~Tzvetkov,
{\it Uniqueness and non-uniqueness of the Gaussian free field evolution under the two-dimensional Wick ordered cubic wave equation}, 
Ann. Inst. Henri Poincar\'e Probab. Stat. 60 (2024), no. 3, 1684--1728.


\bibitem{OPTz}
T.~Oh, O.~Pocovnicu, N.~Tzvetkov,
{\it Probabilistic local well-posedness of the cubic nonlinear wave equation in negative Sobolev spaces},
Ann. Inst. Fourier (Grenoble) 72 (2022) no. 2, 771--830. 


\bibitem{ORTz}
T.~Oh, T.~Robert, N.~Tzvetkov,
{\it Stochastic nonlinear wave dynamics on compact surfaces}, 
Ann. H. Lebesgue 6 (2023), 161--223.





\bibitem{OST2}
T.~Oh, K.~Seong, L.~Tolomeo, 
{\it 
A remark on Gibbs measures with log-correlated Gaussian fields},
  Forum Math. Sigma.
12 (2024), e50,  40 pp.



\bibitem{OST1}
T.~Oh, P.~Sosoe, L.~Tolomeo,
{\it  Optimal integrability threshold for Gibbs measures associated with focusing NLS on the torus}, 
 Invent. Math. 227 (2022), no. 3, 1323--1429. 


\bibitem{OTh}
T.~Oh, L.~Thomann,
{\it A pedestrian approach to the invariant Gibbs measures for the 2-$d$ defocusing nonlinear Schr\"odinger equations}, 
Stoch. Partial Differ. Equ. Anal. Comput. 6 (2018), no. 3, 397--445.


\bibitem{OTh2}
T.~Oh, L.~Thomann,
{\it Invariant Gibbs measures for the 2-$d$ defocusing nonlinear wave equations}, 
Ann. Fac. Sci. Toulouse Math.  29 (2020), no. 1, 1--26.


\bibitem{OTWZ}
T.~Oh, L.~Tolomeo, Y.~Wang, G.~Zheng,
{\it Hyperbolic $P(\Phi)_2$-model on the plane, }
Comm. Math. Phys.
 407 (2026), no. 2, Paper No. 34. 


\bibitem{OWZ}
T.~Oh, Y.~Wang, Y.~Zine,
{\it Three-dimensional stochastic cubic nonlinear wave equation with almost space-time white noise}, 
Stoch. Partial Differ. Equ. Anal. Comput. 10 (2022), no. 3, 898--963. 


\bibitem{OV}
F.~Otto, C.~Villani, 
{\it Generalization of an inequality by Talagrand and links with the logarithmic Sobolev inequality},
J. Funct. Anal. 173 (2000), no. 2, 361--400. 

\bibitem{RSS}
S.~Ryang, T.~Saito, K.~Shigemoto,
{\it Canonical stochastic quantization}, 
Progr. Theoret. Phys. 73 (1985), no. 5, 1295--1298.


\bibitem{Shen}
H.~Shen,  
{\it A stochastic PDE approach to large $N$ problems in quantum field theory: A survey}, 
J. Math. Phys. 63 (2022), no. 8, 081103-1--081103-21. 


\bibitem{SSZZ}
H.~Shen, S.~A.~Smith, R.~Zhu, X.~Zhu,
{\it Large $N$ limit of the $O(N)$ linear sigma model via stochastic quantization}, 
Ann. Probab. 50 (2022), no. 1, 131--202.


\bibitem{SZZ1}
H.~Shen, R.~Zhu, X.~Zhu, 
{\it Large N limit of the $O(N)$ linear sigma model in 3D}, 
Comm. Math. Phys. 394 (2022), no. 3, 953--1009. 


\bibitem{SZZ2}
H.~Shen, R.~Zhu, X.~Zhu, 
{\it Large $N$ limit and $1/N$ expansion of invariant observables in $O(N)$ linear $\s$-model via SPDE}, 
 Probab. Theory Related Fields 191 (2025), no. 3-4, 853--932. 


\bibitem{Simon}
B.~Simon,
{\it The $P (\phi)_2$ Euclidean (quantum) field theory}.
Princeton Ser. Phys.
Princeton University Press, Princeton, NJ, 1974. xx+392 pp.

\bibitem{Sznit}
A.-S.~Sznitman, 
{\it Topics in propagation of chaos},  \'Ecole d'\'Et\'e de Probabilit\'es de Saint-Flour XIX--1989, 165--251, Lecture Notes in Math., 1464, Springer, Berlin, 1991. 

\bibitem{Sri}
S.M.~Srivastava, 
{\it A course on Borel sets.}
Graduate Texts in Mathematics, 180. Springer-Verlag, New York, 1998. xvi+261 pp.



\bibitem{STzX}
C.~Sun, N.~Tzvetkov, W.~Xu,
{\it Weak universality results for a class of nonlinear wave equations}, 
 Ann. Inst. Fourier.  (Grenoble).
 76 (2026), no. 2, 525--576. 

\bibitem{Tala}
M.~Talagrand, 
{\it Transportation cost for Gaussian and other product measures}, 
Geom. Funct. Anal. 6 (1996), no. 3, 587--600.


\bibitem{TTz}
L.~Thomann, N.~Tzvetkov, 
{\it Gibbs measure for the periodic derivative nonlinear Schr\"odinger equation},
Nonlinearity 23 (2010), no. 11, 2771--2791.


\bibitem{Tolo1}  
L.~Tolomeo, 
{\it Global well posedness of the two-dimensional stochastic nonlinear wave equation on an unbounded domain}, 
Ann. Probab. 49 (2021), no. 3, 1402--1426. 


\bibitem{Tolo2}
L.~Tolomeo,
{\it Ergodicity for the hyperbolic $P(\Phi)_2$-model},
arXiv:2310.02190 [math.PR].


\bibitem{Tz10}
N.~Tzvetkov,
{\it Construction of a Gibbs measure associated
to the periodic Benjamin-Ono equation}, 
Probab. Theory Related Fields 146 (2010), no. 3-4, 481--514.


\bibitem{Vaart}
A.W.~van der Vaart,
{\it Asymptotic statistics}. 
Camb. Ser. Stat. Probab. Math., 3.
Cambridge University Press, Cambridge, 1998. xvi+443 pp.




\bibitem{Vill}
C.~Villani,
{\it Topics in optimal transportation}.
Grad. Stud. Math., 58
American Mathematical Society, Providence, RI, 2003. xvi+370 pp.


\bibitem{Wilson}
K.~Wilson, 
{\it Quantum field-theory models in less than 4 dimensions}, 
Phys. Rev. D (3) 7 (1973), 2911--2926. 

\bibitem{Zine}
Y.~Zine,
{\it Smoluchowski-Kramers approximation for singular stochastic wave equations in two dimensions}, 
 Electron. J. Probab. 30 (2025), Paper No. 88, 49 pp.






\end{thebibliography}
\end{document}